\documentclass[a4paper,11pt,reqno]{amsart}
\usepackage{amssymb,amsmath,hyperref}
\title{Deviation of the rank and crank modulo 11}
\author{Nikolay E. Borozenets}
\address{Department of Mathematics and Computer Science, Saint Petersburg State University, Saint Petersburg, 199034, Russia}
\email{nikolayborozenets.spbumcs@gmail.com}

\newcommand\N{\mathbb{N}}
\newcommand\Z{\mathbb{Z}}
\newcommand\Q{\mathbb{Q}}
\newcommand\R{\mathbb{R}}

\renewcommand\Theta{\vartheta}
\newcommand{\Mod}[1]{\ (\mathrm{mod}\ #1)}

\newtheorem{theorem}{Theorem}
\newtheorem{conjecture}[theorem]{Conjecture}
\newtheorem{lemma}[theorem]{Lemma}
\newtheorem{corollary}[theorem]{Corollary}
\newtheorem{proposition}[theorem]{Proposition}
\theoremstyle{definition}
\newtheorem{definition}[theorem]{Definition}
\newtheorem{remark}[theorem]{Remark}

\numberwithin{theorem}{section} 
\numberwithin{equation}{section}

\makeatletter
\@namedef{subjclassname@2020}{\textup{2020} Mathematics Subject Classification}
\makeatother

\begin{document}

\date{27 December 2023}















\begin{abstract}
    In this paper, we build on recent results of Frank Garvan and Rishabh Sarma as well as classical results of Bruce Berndt in order to establish the $11$-dissection of the deviations of the rank and crank modulo $11$. Using our new dissections we re-derive results of Garvan, Atkin, Swinnerton-Dyer, Hussain and Ekin. By developing and exploiting positivity conditions for quotients of theta functions, we will also prove new rank-crank inequalities and make several conjectures. For other applications of our methods, we will prove new congruences for rank moments as well as the Andrews’ smallest parts function and Eisenstein series. 
\end{abstract}
\maketitle
\section{Introduction}
A partition of a positive integer $n$ is a weakly-decreasing sequence of positive integers whose sum is $n$. We denote the number of partitions of $n$ by $p(n)$. Among the most famous results in the theory of partitions are Ramanujan’s congruences:
\begin{gather*}
p(5n+4) \equiv 0 \Mod 5,\\
p(7n+5) \equiv 0 \Mod 7,\\
p(11n+6) \equiv 0 \Mod {11}.
\end{gather*}

In 1944, Dyson \cite{D} conjectured combinatorial interpretations of the first two congruences. He defined the rank of a partition as the largest part minus the number of parts and conjectured that the rank modulo $5$ divided the partitions of $5n+4$ into $5$ equal classes and that the rank modulo $7$ divided the partitions of $7n+5$ into $7$ equal classes. For example the partitions of the number $4$ are $(4),(3,1),(2,2),(2,1,1),(1,1,1,1)$ and their ranks are $3,1,0,-1,-3$ respectively, giving an equinumerous distribution of the partitions of $4$ into the five residue classes modulo $5$. His modulo $5$ and modulo $7$ rank conjectures were proved by Atkin and Swinnerton-Dyer \cite{ASD}.

Although the rank does not explain Ramanujan’s third congruence, Dyson conjectured another function, which he called the crank, that would divide the partitions of $11n+6$ into $11$ equal classes. Andrews and Garvan later discovered the crank \cite{AG}. For a partition $\pi$, let $\lambda(\pi)$ denote the largest part, $\Theta(\pi)$ the number of  ones, and $\mu(\pi)$ the number of parts larger than $\Theta(\pi)$. The crank of $\pi$, denoted $c(\pi)$, is defined as follow
\begin{equation*}
c(\pi) := \begin{cases}
   \lambda(\pi), \ \ &\text{when} \ \  \Theta(\pi)=0,\\
   \mu(\pi) - \Theta(\pi), \ \ &\text{otherwise}.
 \end{cases}
\end{equation*}
The cranks of the five partitions of $4$ are $4,0,2,-2,-4$ respectively, giving an equinumerous distribution of the partitions of $4$ into the five residue classes modulo $5$. 

Let $N(m,n)$ denote the number of partitions of $n$ with rank $m$ and $N(a, r,n)$ denote the number of partitions of $n$ with $\text{rank} \equiv a \Mod r$. Note that there is a symmetric property $N(a,r,n)=N(r-a,r,n)$. Let $M(m,n)$ denote the number of partitions of $n$ with crank $m$ and $M(a, r,n)$ denote the number of partitions of $n$ with $\text{crank} \equiv a \Mod r$. As with the rank, the crank has the symmetry property $M(a,r,n)=M(r-a,r,n)$. Andrews and Garvan \cite{AG} showed
\begin{gather*}
M(a, 5, 5n+4) = \frac{p(5n+4)}{5}, \ \ \text{for} \ 0 \leq a \leq 4,\\
M(a, 7, 7n+5) = \frac{p(7n+5)}{7}, \ \ \text{for} \ 0 \leq a \leq 6,\\
M(a, 11, 11n+6) = \frac{p(11n+6)}{11}, \ \ \text{for} \ 0 \leq a \leq 10.
\end{gather*}
Let $q := e^{2\pi i z}$ be a nonzero complex number with $\operatorname{Im}(z) > 0$. Using the terminology above, we define the deviation of the rank from the expected value to be
\begin{equation*}
D(a,r) = D(a,r;q) := \sum_{n=0}^{\infty} \left(N(a,r,n) - \frac{p(n)}{r}\right)q^n
\end{equation*}
and the deviation of the crank from the expected value to be
\begin{equation*}
D_{C}(a,r) = D_{C}(a,r;q) := \sum_{n=0}^{\infty} \left(M(a,r,n) - \frac{p(n)}{r}\right)q^n.
\end{equation*}
Recall the $q$-Pochhammer notation, defined by
\begin{gather*}
(x)_n=(x;q)_n:=\prod_{i=0}^{n-1}(1-xq^i).
\end{gather*}
Also recall the definition of the universal mock theta function
\begin{equation*} 
 g(x;q) := x^{-1}\left(-1+\sum_{n=0}^{\infty} \frac{q^{n^2}}{(x)_{n+1}(q/x)_n} \right)
\end{equation*}
and the definition of the theta function
\begin{equation*} 
j(x;q) := (x)_{\infty}(q/x)_{\infty}(q)_{\infty} = \sum_{k = -\infty}^{\infty} (-1)^k q^{\binom{k}{2}}x^k,
\end{equation*}
where the equivalence of product and sum follows from Jacobi’s triple
product identity. Let $a$ and $m$ be integers with $m$ positive. We introduce
\begin{gather} \label{definition:JPX}
J_{a,m}:=j(q^a;q^m), \ \ J_{m}:=J_{m,3m}=(q^{m};q^{m})_{\infty} \ \ \text{and} \ \ P_{i} := J_{i,11}, \ \ X_{i} := J_{11i,121}.
\end{gather}
To formulate our main results we need the following notation.
\begin{definition} \label{definition:v11G11Theta} We define
\begin{align} 
\begin{split}
v_{11}(a_0, a_1, a_2, a_3, a_4&,  a_5, a_7, a_8, a_9, a_{10})\\
&:= \frac{J_{121}^2}{11} \Big( a_0\frac{1}{X_1} + a_1q\frac{X_5}{X_2X_3} +a_2q^2\frac{X_3}{X_1X_4} +a_3q^3\frac{X_2}{X_1X_3} + a_4q^4\frac{1}{X_2}\\
&\qquad + a_5q^5\frac{X_4}{X_2X_5} + a_7q^7\frac{1}{X_3}+a_8q^{19}\frac{X_1}{X_4X_5} +a_9q^{9}\frac{1}{X_4} + a_{10}q^{10}\frac{1}{X_5} \Big ),
\end{split}
\end{align}
\begin{align*} 
G_{11}(b_0,b_4,b_7,b_9,b_{10})&:=
b_0q^{22}g(q^{22};q^{121})+b_4q^{37}g(q^{44};q^{121})\\
& \qquad +b_{7}q^{40}g(q^{55};q^{121})+b_9q^{31}g(q^{33};q^{121})
+b_{10}[q^{-1}+q^{10}g(q^{11};q^{121})],
\end{align*}
and
\begin{equation} \label{expression:vterm}
\Theta(a_1,a_2,a_3,a_4,a_5) := \frac{J_{11}^6}{J_{1}^2} \Big[a_1\frac{q^{2}}{P_{4}P_{5}^2}+a_2\frac{1}
{P_{1}^2P_{3}}+a_3\frac{q}{P_{1}P_{4}^2}+a_4\frac{q}{P_{2}^2P_{5}}+a_5\frac{q}{P_{2}P_{3}^2}\Big].
\end{equation}
\end{definition}
The next two theorems give the $11$-dissections for the deviations of the crank and rank modulo $11$.
\begin{theorem} \label{theorem:crankdissection} We have
\begin{align*}
D_{C}(0,11) &= v_{11}(10,-12,-2,8,6,4,-4,-6,-8,2),\\
D_{C}(1,11) &= v_{11}(-1,10,-2,-3,-5,4,7,5,3,2),\\
D_{C}(2,11) &= v_{11}(-1,-1,9,-3,6,-7,-4,-6,3,2) ,\\
D_{C}(3,11) &= v_{11}(-1,-1,-2,8,-5,4,-4,5,3,-9) ,\\
D_{C}(4,11) &= v_{11}(-1,-1,-2,-3,6,-7,7,5,-8,2) ,\\
D_{C}(5,11) &= v_{11}(-1,-1,-2,-3,-5,4,-4,-6,3,2).
\end{align*}
\end{theorem}
\begin{theorem} \label{theorem:rankdissection} We have
\begin{align*}
D(0,11) &= G_{11}(-2,0,0,0,0) + v_{11}(10,-12,-2,8,6,4,18,-6,-8,2) + \sum_{m=0}^{10} \Theta_{0,m}(q^{11})q^m,\\
D(1,11) &= G_{11}(1,0,-1,0,0) + v_{11}(-1,10,-2,-3,-5,4,-4,-6,3,2) + \sum_{m=0}^{10} \Theta_{1,m}(q^{11})q^m,\\
D(2,11) &= G_{11}(0,0,1,0,-1) + v_{11}(-1,-1,9,-3,6,-7,-4,5,3,2) + \sum_{m=0}^{10} \Theta_{2,m}(q^{11})q^m,\\
D(3,11) &= G_{11}(0,0,0,1,1) + v_{11}(-1,-1,-2,-3,17,-7,-4,5,3,-9) + \sum_{m=0}^{10} \Theta_{3,m}(q^{11})q^m,\\
D(4,11) &= G_{11}(0,1,0,-1,0) + v_{11}(-1,-1,-2,-3,6,4,-4,5,-8,13) + \sum_{m=0}^{10} \Theta_{4,m}(q^{11})q^m,\\
D(5,11) &= G_{11}(0,-1,0,0,0) + v_{11}(-1,-1,-2,-3,-5,4,-4,-6,3,-9) + \sum_{m=0}^{10} \Theta_{5,m}(q^{11})q^m,
\end{align*}
where 
\begin{align*}
\Theta_{0,6}(q) &= \Theta(0,0,2,2,-2),\\
\Theta_{1,6}(q) &= \Theta(-1,1,-1,-2,1),\\
\Theta_{2,6}(q) &=\Theta(1,0,-1,2,0) ,\\
\Theta_{3,6}(q) &=\Theta(1,0,1,-1,-1) ,\\
\Theta_{4,6}(q) &=\Theta(0,-1,1,0,2),\\
\Theta_{5,6}(q) &=\Theta(-1,0,-1,0,-1),
\end{align*}
and the other $\Theta_{a,m}(q)$ are given explicitly in Section \ref{section:devrank}.
\end{theorem}
\begin{remark} 
The dissections for the deviations of the rank and crank modulo $5$, $7$ and $4$, $8$ were found by Hickerson and Mortenson \cite{HM17, M} and we use the setting from there.
\end{remark} 

\subsection{A prelude to new proofs of classical results}\label{subsection:preludeold}
Using Theorems \ref{theorem:crankdissection} and \ref{theorem:rankdissection}, we give new proofs of classical results. In Section \ref{subsection:crankequal}  we re-derive crank equalities found by Garvan \cite{G88} such as, for $n \geq 0$,
\begin{gather*} 
M(0,11,11n+1) + M(1,11,11n+1) = 2M(2,11,11n+1),\\
M(2,11,11n+1) = M(3,11,11n+1) = M(4,11,11n+1) = M(5,11,11n+1).
\end{gather*}
In Section \ref{subsection:crankinequal} we re-derive crank-crank inequalities, which were first proved by Ekin \cite{E} and Berkovich and Garvan \cite{BG}, such as, for $n \geq 0$,
\begin{equation*}
M(1,11,11n+1) \geq \frac{p(11n+1)}{11} \geq M(2,11,11n+1) \geq M(0,11,11n+1).
\end{equation*}
In Section \ref{subsection:partfunccong} we re-derive congruences for the partition function, which were establish by Atkin and Swinnerton-Dyer \cite{ASD}, such as
\begin{equation*}
\sum_{n = 0}^{\infty} p(11n)q^n \equiv \frac{J_{11}^2}{P_1} \Mod {11}.
\end{equation*}
In Section \ref{subsection:linearrangcong} we re-derive linear rank congruences due to Atkin and Hussain \cite{AH}, such as, for $n \geq 0$,
\begin{equation*} 
N(2,11;11n)-5N(3,11;11n)-2N(4,11;11n)+6N(5,11;11n) \equiv 0 \Mod {11}.
\end{equation*}

\subsection{A prelude to new results and new conjectures} \label{subsection:preludenew} In Section \ref{section:mainresults}  we present new rank and rank-crank inequalities, such as, for $n \geq 0$,
\begin{equation*} 
2N(2,11,11n)+N(3,11,11n) + N(5,11,11n) \geq 4N(4,11,11n), 
\end{equation*}
which follow directly from the positivity of Fourier coefficients
of theta quotients as described in Section \ref{subsection:positivitytechniques}. Then as a corollary we derive two-term, four-term and six-term inequalities, such as, for $n \geq 0$,
\begin{align*} 
M(1,11,11n) &\geq N(4,11,11n),\\
N(2,11,11n)+N(3,11,11n) &\geq N(4,11,11n)+M(1,11,11n),\\
N(2,11,11n)+2N(3,11,11n)&\geq N(5,11,11n)+2M(1,11,11n).
\end{align*}
Using a numerical computing environment, it is possible to generate higher order inequalities for eight-terms, ten-terms, etc.

As another application of Theorem \ref{theorem:rankdissection} we present in Section \ref{section:mainresults} new congruences for rank moments, such as
\begin{equation*} 
\sum_{n=0}^{\infty}\left(\sum_{m=-\infty}^{\infty} m^2N(m,11n+6) \right)q^n \equiv \Theta(-4,3,1,5,-2)\Mod {11}.
\end{equation*}
and congruences for the Andrews’ smallest parts function $\operatorname{spt}(n)$, such as
\begin{equation*} 
\sum_{n=0}^{\infty}\operatorname{spt}(11n+6)q^n \equiv \Theta(2,4,5,3,1) \Mod {11}.
\end{equation*}
In Section \ref{subsection:conjinequal} we state new conjectural two-term rank and rank-crank inequalities, such as
\begin{multline*} 
N(0,11,11n) \geq_{3} N(1,11,11n) \geq N(2,11,11n) \geq_{1} M(0,11,11n) \geq \frac{p(11n)}{11}\geq \\ \geq M(1,11,11n) \geq N(3,11,11n) \geq_{2} N(4,11,11n) \geq N(5,11,11n),
\end{multline*}
where $A_n \geq B_n$ means that $A_n \geq B_n$ for all $n \geq 0$ and $A_n \geq_m B_n$ means that $A_n \geq B_n$ for all $n \geq m$. Also in Section \ref{subsection:conjinequal} we state Conjecture \ref{conjecture:positivitycor}, which is the generalization of our observations on the positivity of Fourier coefficients
of theta quotients provided in Section \ref{subsection:positivitytechniques} and we state Conjecture \ref{conjecture:arbinequal}, which is the generalization of our new rank and rank-crank inequalities from Section \ref{section:mainresults}.

\subsection{A guide to the paper} In Section \ref{section:mainresults} we state our new rank-crank inequalities and new congruences as described in Section \ref{subsection:preludenew}. In Section \ref{section:proofdevcrank} we prove Theorem \ref{theorem:crankdissection}, which gives us the $11$-dissection of the deviation of the crank modulo $11$. In Section \ref{section:devrank} we present all the dissection elements of the deviation of the rank and prove Theorem \ref{theorem:rankdissection}.

In Section \ref{section:oldresultsproofs} we provide new proofs for classical results
as described in Section \ref{subsection:preludeold}. In Section \ref{subsection:crankinequal} we observe known results on crank inequalities as described in Section \ref{subsection:preludenew}. In Section \ref{subsection:positivitytechniques} we develop techniques to exploit the positivity of Fourier
coefficients. In Section \ref{subsection:rankcrankinequal} we use results from Section \ref{subsection:positivitytechniques} to prove some examples of new rank-crank inequalities. Proofs of other new rank-crank inequalities are straightforward and similar but for the sake of prosperity and to underscore the role played by the dissections of Theorems \ref{theorem:crankdissection} and \ref{theorem:rankdissection}, we place in Section \ref{section:fullcalc} all the proofs for the new results found in Section \ref{section:mainresults}. In Section \ref{subsection:conjinequal} we conjecture rank-crank inequalities and state Conjecture \ref{conjecture:positivitycor}, which describes in general the positivity of Fourier coefficients of sums of theta quotients. Also in Section \ref{subsection:conjinequal} we state Conjecture \ref{conjecture:arbinequal}, which describes the general case of rank-crank inequalities. In Section \ref{section:momentscalc} we prove new congruences for rank moments, Andrews’ smallest parts function $\operatorname{spt}(n)$ and Eisenstein series.

\section{The Main Results in Full}\label{section:mainresults}
Using Theorem \ref{theorem:crankdissection} and Theorem \ref{theorem:rankdissection} we are able to find new rank and rank-crank inequalities.

\begin{theorem} \label{theorem:rankcrankbasicineq} Consider $n\geq 0$. For $N_i = N(i,11,11n)$ and $M_i = M(i,11,11n)$ we have 
\begin{align*}
N_0 + 2N_1 + M_1 &\geq 2N_2+N_4+M_0,\\
 N_0 +2N_1+3N_2 + M_1 &\geq 3N_3 + 3N_5+M_0,\\
  2N_2+N_3 + N_5&\geq 4N_4,\\
   N_2+5N_3+3N_4+M_0 &\geq N_0+2N_1+6N_5+M_1.
\end{align*}
For $N_i = N(i,11,11n+1)$ and $M_i = M(i,11,11n+1)$ we have 
\begin{align*}
 N_0 +4N_2+N_4+M_0&\geq 2N_1+3N_3+N_5+M_2,\\
  N_0+3N_1+6N_3+M_0&\geq 4N_2+6N_4+M_2,\\
  2N_1+6N_4 &\geq N_2+3N_3+4N_5 ,\\
   N_1+2N_2+3N_5 +3M_2 &\geq N_0+N_3+4N_4+2M_0+M_1,\\
    3N_2+2N_3+N_4+3M_2 &\geq N_0+N_1+4N_5+2M_0+M_1.
\end{align*}
For $N_i = N(i,11,11n+2)$ and $M_i = M(i,11,11n+2)$ we have 
\begin{align*}
 3N_2+N_4 &\geq 2N_0+2N_5\\
2N_1+2N_3+N_5+2M_0 &\geq 2N_2+3N_4+2M_2\\
3N_0+2N_2+M_2 &\geq N_1+N_3+3N_5+M_0 \\
2N_0+N_1+N_3+3N_4+M_0 &\geq 4N_2+3N_5+M_2\\
N_0+3N_1+N_2+8N_5+3M_0 &\geq 8N_3+5N_4+3M_2.
\end{align*}
For $N_i = N(i,11,11n+3)$ and $M_i = M(i,11,11n+3)$ we have 
\begin{align*}
N_0+2N_3+M_1 &\geq N_1+2N_5+M_0,\\
5N_1+2N_2+2N_4 &\geq 2N_0 + 4N_3 + 3N_5,\\
2N_0+4N_3+N_5+M_0 &\geq N_1+3N_2+3N_4+M_1,\\
6N_2+3N_5+5M_1 &\geq N_0+N_1+2N_3+5N_4+5M_0 ,\\
4N_0+2N_1+4N_4+3M_0 &\geq 7N_2+3N_3+3M_1.
\end{align*}
For $N_i = N(i,11,11n+4)$ and $M_i = M(i,11,11n+4)$ we have 
\begin{align*}
4N_1+3N_3+5M_1 &\geq 2N_0 + 3N_2 + N_4+N_5+5M_0,\\
  4N_0 +5N_2 +3M_0 &\geq 5N_1+2N_3+N_4+N_5+3M_1,\\
   3N_1+N_2+M_0 &\geq 2N_0 +2N_3+M_1 ,\\
   3N_0 +N_4+N_5+M_1 &\geq 3N_2+2N_3+M_0.
\end{align*}
For $N_i = N(i,11,11n+5)$ and $M_i = M(i,11,11n+5)$ we have 
\begin{align*}
3N_0+N_2+N_5+2M_2 &\geq2N_1+3N_3+2M_0,\\
7N_1+N_4+3M_2 &\geq 3N_0+2N_2+N_3+2N_5+3M_0,\\
2N_0+N_2+3N_3+N_5+4M_0&\geq 3N_1+4N_4+4M_2,\\
4N_2+7N_3+2N_4 &\geq 4N_0+2N_1+7N_5 ,\\
4N_1 + N_2 +2N_4+N_5+5M_2 &\geq 2N_0+6N_3+5M_0 .
\end{align*}
For $N_i = N(i,11,11n+7)$ and $M_i = M(i,11,11n+7)$ we have 
\begin{align*}
 N_0 +2N_4 +2M_0&\geq 2N_3+N_5+2M_1,\\
3N_0+N_1+N_2+7N_3+4N_5&\geq 5N_4+4M_0+7M_1,\\
4N_1+4N_2+3N_4+4M_1 &\geq  4N_0 +6N_3+N_5 + 4M_0,\\
N_0 + 5N_3 + 2N_4&\geq 2N_1+2N_2+4N_5.
\end{align*}
For $N_i = N(i,11,11n+8)$ and $M_i = M(i,11,11n+8)$ we have 
\begin{align*}
N_0 + 2N_2 + 3N_4 + 5M_0 &\geq 3N_1+3N_5+5M_1,\\
  2N_1 +4N_3+2N_5+3M_1 &\geq N_0+2N_2+5N_4+3M_0,\\
   2N_1 + 5N_2 + 2N_5 + M_1 &\geq 3N_0 + 5N_3+N_4+M_0,\\
    7N_0 + 4N_1 + 5N_4 +4N_5&\geq 8N_2+N_3+4M_0+7M_1,\\
    5N_0 + 6N_1 +4N_3+2N_4&\geq N_2 +5N_5 +6M_0+5M_1.
\end{align*}
For $N_i = N(i,11,11n+9)$ and $M_i = M(i,11,11n+9)$ we have 
\begin{align*}
4N_2+M_0 &\geq 2N_1+N_3+N_4+M_1,\\
4N_1+N_3+N_4&\geq N_0+3N_2+2N_5,\\
 4N_0+3N_3+3N_4+4M_0 &\geq 2N_1+5N_2+3N_5+4M_1 ,\\
  3N_1+4N_2+7N_5 +3M_1 &\geq 2N_0 +6N_3+6N_4+3M_0.
\end{align*}
For $N_i = N(i,11,11n+10)$ and $M_i = M(i,11,11n+10)$ we have 
\begin{align*}
 3N_1+2N_5+2M_3 &\geq 2N_2+2N_3+N_4+2M_0,\\
3N_0 + N_2 + N_3 + M_0 &\geq  3N_1 + 2N_4 + M_3,\\
 2N_1+4N_4 + M_3 &\geq N_0+N_2+N_3+3N_5+M_0,\\
6N_2+6N_3+6M_0+5M_3&\geq 6N_0 +6N_1+3N_4+8N_5.
\end{align*}
\end{theorem}
As a corollary of Theorem \ref{theorem:rankcrankbasicineq} we can derive two-term and four-term rank-crank inequalities. Similar two-term and four-term rank and rank-crank inequalities for modulo $2,3$ and $4$ were proved and conjectured by Andrews and Lewis \cite{AL}, for modulo $8$ were conjectured by Lewis \cite{L}, for modulo $10$ were proved and conjectured by Mao \cite{Ma}.
\begin{corollary} \label{corollary:inequallevel2level4} Consider $n\geq 0$. For $N_i = N(i,11,11n)$ and $M_i = M(i,11,11n)$ we have 
\begin{align*}
M_1 &\geq N_4,\\
N_{2}+N_{3} &\geq N_{4}+M_{1}.
\end{align*}
For $N_i = N(i,11,11n+1)$ and $M_i = M(i,11,11n+1)$ we have 
\begin{align*}
N_2 &\geq M_2,\\
M_2 &\geq N_4,\\
N_{2}+N_{4} &\geq N_{3}+N_{5}.
\end{align*}
For $N_i = N(i,11,11n+2)$ and $M_i = M(i,11,11n+2)$ we have 
\begin{align*}
N_1 &\geq M_2,\\
M_0 &\geq N_5,\\
2 M_{0} &\geq N_{2}+N_{5},\\
N_{2}+M_{0} &\geq N_{0}+N_{5},\\
N_{1}+N_{3} &\geq M_{0}+M_{2},\\
N_{1}+N_{5} &\geq N_{4}+M_{2}.
\end{align*}
For $N_i = N(i,11,11n+3)$ and $M_i = M(i,11,11n+3)$ we have 
\begin{align*}
N_0 &\geq M_0,\\
M_1 &\geq N_4,\\
N_{0}+M_{0} &\geq 2 N_{2},\\
N_{2}+M_{1} &\geq N_{4}+M_{0},\\
N_{0}+N_{3} &\geq N_{2}+M_{1}.
\end{align*}
For $N_i = N(i,11,11n+4)$ and $M_i = M(i,11,11n+4)$ we have 
\begin{align*}
N_1 &\geq M_0,\\
2 N_{1}&\geq N_{0}+M_{0},\\
N_{1}+M_{1}&\geq N_{2}+M_{0},\\
N_{1}+N_{3}&\geq 2 M_{0}.
\end{align*}
For $N_i = N(i,11,11n+5)$ and $M_i = M(i,11,11n+5)$ we have 
\begin{align*}
N_2 &\geq M_0,\\
M_2 &\geq N_5,\\
N_{2}+N_{3} &\geq M_{0}+M_{2},\\
N_{2}+M_{2} &\geq N_{0}+N_{5},\\
N_{0}+M_{0} &\geq N_{1}+N_{4},\\
N_{2}+N_{3} &\geq N_{1}+N_{5},\\
N_{0}+N_{2} &\geq N_{1}+M_{0}.
\end{align*}
For $N_i = N(i,11,11n+7)$ and $M_i = M(i,11,11n+7)$ we have 
\begin{align*}
N_0 &\geq M_1,\\
M_0 &\geq N_4,\\
N_{3}+N_{4} &\geq N_{5}+M_{0}.
\end{align*}
For $N_i = N(i,11,11n+8)$ and $M_i = M(i,11,11n+8)$ we have
\begin{align*}
N_2 &\geq M_1,\\
M_0 &\geq N_3,\\
M_{0}+M_{1} &\geq N_{2}+N_{4},\\
N_{1}+N_{5}  &\geq 2 N_{3},\\
N_{0}+M_{0}  &\geq 2 N_{2},\\
N_{3}+M_{0}  &\geq N_{4}+M_{1},\\
N_{2}+N_{3}  &\geq 2 M_{1},\\
2 M_{0}  &\geq N_{1}+N_{5}.
\end{align*}
For $N_i = N(i,11,11n+9)$ and $M_i = M(i,11,11n+9)$ we have
\begin{align*}
N_2 &\geq M_1,\\
M_0 &\geq N_5,\\
M_{1}+M_{0} &\geq N_{2}+N_{5} ,\\
N_{0}+M_{0} &\geq N_{2}+M_{1},\\
N_{2}+N_{5}&\geq N_{3}+N_{4}.
\end{align*}
For $N_i = N(i,11,11n+10)$ and $M_i = M(i,11,11n+10)$ we have
\begin{align*}
N_1 &\geq M_0,\\
M_3 &\geq N_4,\\
N_{0}+M_{0}&\geq N_{1}+N_{4}, \\
N_{2}+N_{3} &\geq 2 M_{0}, \\
N_{1}+N_{4}  &\geq 2 M_{0}, \\ 
N_{1}+M_{3}   &\geq N_{2}+N_{3}.
\end{align*}
\end{corollary}
We can also consider six-term rank-crank inequalities.
\begin{corollary} \label{corollary:inequallevel6} Consider $n\geq 0$. For $N_i = N(i,11,11n)$ and $M_i = M(i,11,11n)$ we have 
\begin{align*}
N_{2}+2 N_{3}&\geq N_{5}+2 M_{1},\\
2 N_{3}+N_{4}&\geq 2 N_{5}+M_{1},\\
3 M_{1}&\geq N_{2}+N_{4}+N_{5},\\
N_{3}+2 M_{1}&\geq N_{2}+2 N_{5}.
\end{align*}
For $N_i = N(i,11,11n+1)$ and $M_i = M(i,11,11n+1)$ we have 
\begin{align*}
2 N_{1}+N_{3} &\geq N_{2}+2 N_{4},\\
2 N_{1}+N_{4} &\geq N_{2}+2 N_{5},\\
N_{1}+2 N_{4} &\geq 2 N_{5}+M_{0},\\
N_{1}+N_{3}+N_{5} &\geq 3 N_{4},\\
2 N_{2}+N_{4} &\geq 2 N_{5}+M_{0}.
\end{align*}
For $N_i = N(i,11,11n+2)$ and $M_i = M(i,11,11n+2)$ we have 
\begin{align*}
N_{0}+N_{1}+N_{4} &\geq N_{2}+M_{0}+M_{2},\\
3 M_{0}  &\geq N_{2}+N_{3}+N_{4},\\
N_{1}+N_{2}+N_{5}&\geq 2 M_{0}+M_{2},\\
3 M_{0} &\geq N_{0}+N_{4}+N_{5},\\
N_{1}+N_{3}+M_{0} &\geq N_{0}+2 N_{5},\\
N_{1}+N_{3}+M_{0} &\geq N_{2}+N_{5}+M_{2},\\
N_{1}+2 M_{0} &\geq N_{2}+N_{4}+M_{2},\\
N_{1}+N_{2}+N_{3} &\geq N_{0}+N_{5}+M_{2},\\
N_{1}+N_{2}+M_{0} &\geq N_{0}+N_{4}+M_{2}.
\end{align*}
For $N_i = N(i,11,11n+3)$ and $M_i = M(i,11,11n+3)$ we have 
\begin{align*}
2 N_{1}+N_{4} &\geq N_{3}+N_{5}+M_{1},\\
N_{0}+N_{1}+N_{5} &\geq M_{0}+2 M_{1},\\
N_{0}+N_{1}+N_{4}  &\geq N_{2}+2 M_{1},\\
M_{0}+2 M_{1} &\geq N_{2}+N_{3}+N_{5},\\
2 N_{0}+N_{5}  &\geq 2 N_{2}+N_{4},\\
N_{0}+N_{3}+N_{5} &\geq N_{4}+M_{0}+M_{1},\\
3 M_{1}  &\geq N_{1}+2 N_{4},\\
2 N_{2}+N_{5} &\geq N_{4}+2 M_{0},\\
N_{0}+N_{3}+M_{1} &\geq N_{1}+2 N_{4},\\
N_{0}+2 N_{3} &\geq M_{0}+2 M_{1}.
\end{align*}
For $N_i = N(i,11,11n+4)$ and $M_i = M(i,11,11n+4)$ we have 
\begin{align*}
2 N_{0}+M_{0} &\geq N_{1}+2 N_{3},\\
2 N_{0}+N_{2} &\geq N_{1}+N_{3}+M_{1},\\
N_{0}+N_{2}+M_{1} &\geq N_{1}+N_{4}+N_{5},\\
N_{0}+N_{2}+N_{3}  &\geq N_{4}+N_{5}+M_{0},\\
N_{0}+2 M_{1}   &\geq2 N_{2}+N_{3} ,\\
3 M_{1}   &\geq N_{2}+N_{4}+N_{5}.
\end{align*}
For $N_i = N(i,11,11n+5)$ and $M_i = M(i,11,11n+5)$ we have 
\begin{align*}
2 N_{3}+M_{0} &\geq 2 N_{5}+M_{2},\\
N_{2}+2 N_{3} &\geq N_{0}+2 N_{5},\\
N_{3}+M_{0}+M_{2} &\geq N_{0}+2 N_{5},\\
N_{3}+2 M_{0} &\geq N_{1}+N_{4}+N_{5},\\
2 N_{2}+N_{3} &\geq N_{0}+N_{5}+M_{0},\\
N_{1}+N_{2}+N_{5}  &\geq N_{3}+2 M_{0},\\
M_{0}+2 M_{2}  &\geq N_{0}+N_{3}+N_{5},\\
N_{1}+N_{2}+N_{4} &\geq 3 M_{0},\\
N_{0}+N_{2}+N_{5} &\geq 2 M_{0}+M_{2}.
\end{align*}
For $N_i = N(i,11,11n+7)$ and $M_i = M(i,11,11n+7)$ we have 
\begin{align*}
N_{1}+N_{2}+M_{1} &\geq N_{0}+N_{5}+M_{0},\\
N_{1}+N_{2}+M_{1} &\geq N_{0}+N_{3}+N_{4},\\
M_{0} +2 M_{1} &\geq N_{0}+2 N_{5},\\
N_{0}+2 N_{3} &\geq M_{0} + 2 M_{1},\\
N_{1}+N_{2}+N_{4} &\geq 2 M_{0} + M_{1}.
\end{align*}
For $N_i = N(i,11,11n+8)$ and $M_i = M(i,11,11n+8)$ we have
\begin{align*}
N_{1}+N_{5}+M_{1} &\geq N_{2}+N_{3}+N_{4},\\
M_{0}+2 M_{1}  &\geq N_{0}+2 N_{4},\\
N_{1}+N_{2}+N_{5}  &\geq N_{0}+2 N_{4},\\
N_{1}+N_{3}+N_{5}  &\geq N_{4}+M_{0}+M_{1},\\
N_{1}+N_{5}+M_{0}  &\geq N_{2}+2 N_{4},\\
N_{0}+N_{1}+N_{5}  &\geq 2 N_{2}+N_{3},\\
N_{2}+N_{3}+M_{0}  &\geq N_{0}+2 N_{4},\\
N_{2}+M_{0}+M_{1}  &\geq N_{0}+N_{3}+N_{4},\\
N_{1}+N_{2}+N_{5}  &\geq M_{0}+2 M_{1},\\
N_{0}+N_{1}+N_{5}  &\geq N_{2}+2 M_{1},\\
2 N_{2}+M_{0}  &\geq N_{0}+N_{4}+M_{1},\\
N_{0}+N_{3}+M_{0} &\geq N_{2}+2 M_{1},\\
2 N_{2}+M_{0}  &\geq N_{0}+2 N_{3},\\
N_{0}+N_{4}+M_{0} &\geq N_{2}+N_{3}+M_{1},\\
N_{0}+N_{3}+N_{4} &\geq 3 M_{1},\\
2 N_{2}+N_{4} &\geq 3 M_{1}.
\end{align*}
For $N_i = N(i,11,11n+9)$ and $M_i = M(i,11,11n+9)$ we have
\begin{align*}
2 N_{1}+M_{1}&\geq N_{0}+N_{3}+N_{4},\\
N_{0}+N_{3}+N_{4} &\geq N_{2}+N_{5}+M_{1},\\
M_{0} + 2 M_{1} &\geq N_{2}+N_{3}+N_{4},\\
N_{0}+N_{5}+M_{0} &\geq N_{2}+N_{3}+N_{4},\\
N_{1}+N_{2}+N_{5} &\geq M_{0} + 2 M_{1},\\
N_{0}+N_{3}+N_{4} &\geq 3 M_{1},\\
N_{0}+2 M_{0} &\geq N_{1}+N_{2}+N_{5},\\
N_{0}+N_{2}+N_{5} &\geq 3 M_{1},\\
N_{0}+2 M_{0} &\geq N_{1}+N_{3}+N_{4},\\
N_{2}+M_{0}+M_{1} &\geq N_{1}+N_{3}+N_{4},\\
N_{0}+N_{2}+M_{0} &\geq N_{1}+2 M_{1},\\
N_{0}+N_{2}+M_{0} &\geq 2 N_{1}+N_{5}.
\end{align*}
For $N_i = N(i,11,11n+10)$ and $M_i = M(i,11,11n+10)$ we have
\begin{align*}
N_{0}+N_{2}+N_{3} &\geq N_{1}+N_{4}+M_{0}, \\
2 M_{0}+M_{3} &\geq N_{1}+N_{4}+N_{5}, \\
N_{2}+N_{3}+M_{3} &\geq N_{1}+N_{4}+N_{5}, \\
N_{4}+M_{0}+M_{3} &\geq N_{0}+2 N_{5}, \\
N_{2}+N_{3}+M_{3} &\geq N_{0}+N_{5}+M_{0}, \\
2 M_{0}+M_{3} &\geq N_{2}+N_{3}+N_{4}, \\
N_{0}+N_{1}+N_{5} &\geq N_{2}+N_{3}+N_{4}, \\
M_{0}+2 M_{3}  &\geq N_{0}+N_{4}+N_{5}, \\
2 N_{1}+N_{5} &\geq N_{2}+N_{3}+M_{0},\\
N_{1}+2 M_{3}  &\geq N_{0}+N_{5}+M_{0},\\
2 N_{1}+N_{5} &\geq 3 M_{0}.
\end{align*}
\end{corollary}
\begin{remark} Note that the four-term and six-term inequalities in Corollary \ref{corollary:inequallevel2level4} and Corollary \ref{corollary:inequallevel6} are selected so that
they are corollaries of inequalities in Theorem \ref{theorem:rankcrankbasicineq} but they are not corollaries of Conjecture \ref{conjecture:rankcrankconj} and they cannot be obtained from each other using Conjecture \ref{conjecture:rankcrankconj}. 
\end{remark}
For residue $6$ we have the following results.
\begin{theorem} \label{theorem:inequalrank6basis} Consider $n\geq 0$. For $N_i = N(i,11,11n+6)$ we have
\begin{align*}
2N_1+N_2+2N_4 &\geq 2N_3+3N_5,\\
2N_0+N_2+2N_5 &\geq 2N_1+N_3+2N_4,\\
N_0+N_1+3N_3+N_4 &\geq 4N_2+2N_5,\\
N_0 + 6N_1+4N_2 &\geq 2N_3+5N_4+4N_5.
\end{align*}
\end{theorem}
\begin{corollary} \label{corollary:ineqrank6level2} Consider $n\geq 0$. For $N_i = N(i,11,11n+6)$ we have
\begin{equation*}
N_0 \geq \frac{p(11n+6)}{11} \ \ \text{and} \ \ \frac{p(11n+6)}{11} \geq N_5.
\end{equation*}
\end{corollary}
\begin{corollary} \label{corollary:ineqrank6level4level6} Consider $n\geq 0$. For $N_i = N(i,11,11n+6)$ we have
\begin{align*}
N_{0}+N_{3}  &\geq N_{1}+N_{4},\\
N_{1}+N_{2}+N_{4} &\geq N_{3}+2 N_{5},\\
N_{2}+2 N_{3}  &\geq N_{1}+N_{4}+N_{5},\\
N_{0}+N_{1}+N_{4}  &\geq N_{2}+2 N_{5},\\
 3 N_{3}  &\geq N_{2}+2 N_{5},\\
N_{0}+2 N_{3}  &\geq 2 N_{2}+N_{5}.
\end{align*}
\end{corollary}
We prove Theorem \ref{theorem:rankcrankbasicineq}, Corollary \ref{corollary:inequallevel2level4}, Corollary \ref{corollary:inequallevel6}, Theorem \ref{theorem:inequalrank6basis}, Corollary \ref{corollary:ineqrank6level2} and Corollary \ref{corollary:ineqrank6level4level6} in Section \ref{subsection:rankcrankinequal}. In the spirit of Corollary \ref{corollary:inequallevel2level4} and Corollary \ref{corollary:ineqrank6level2} we introduce new conjectural rank-crank inequalities in Section \ref{subsection:conjinequal}. 

Recall 
\begin{equation*}
P_i := J_{i,11}.
\end{equation*}
Define the following sums of theta quotients.
\begin{definition} \label{definition:mainbracket} We define
\begin{equation*}
[c_1,c_2,c_3,c_4,c_5] := \frac{J_{11}^2}{J_1^3} \Big(c_1P_5^2 P_4 + c_2 q^2 P_1^2 P_3 + c_3 q P_4^2 P_1 + c_4 q P_2^2P_5 + c_5 q P_3^2P_2\Big).
\end{equation*}
\end{definition}
\begin{definition} \label{definition:residuebrackets} We define
\begin{align*}
[c_1, c_2, c_3, c_4, c_5; c_6]_{0} &:=  \frac{1}{P_1} [c_1, c_2, c_3, c_4, c_5] + c_{6} \frac{J_{11}^2}{J_{1}^3}\frac{qP_1P_3P_4P_5}{P_2^2},\\
[c_1, c_2, c_3, c_4, c_5; c_6]_{1} &:= \frac{P_5}{P_2P_3} [c_1, c_2, c_3, c_4, c_5] + c_6 \frac{J_{11}^2}{J_{1}^3}  \frac{qP_2^2P_4}{P_1},\\
[c_1, c_2, c_3, c_4, c_5; c_6]_{2} &:=  \frac{P_3}{P_1P_4} [c_1, c_2, c_3, c_4, c_5] + c_6 \frac{J_{11}^2}{J_{1}^3}  \frac{q^3P_1^2P_2}{P_5},\\
[c_1, c_2, c_3, c_4, c_5; c_6]_{3} &:= \frac{P_2}{P_1P_3}  [c_1, c_2, c_3, c_4, c_5] + c_6 \frac{J_{11}^2}{J_{1}^3} \frac{qP_3^2P_5}{P_4},\\
[c_1, c_2, c_3, c_4, c_5; c_6]_{4} &:= \frac{1}{P_2}   [c_1, c_2, c_3, c_4, c_5] + c_6 \frac{J_{11}^2}{J_{1}^3}  \frac{q^{2}P_1P_2P_3P_5}{P_4^2},\\
[c_1, c_2, c_3, c_4, c_5; c_6]_{5} &:= \frac{P_4}{P_2P_5}   [c_1, c_2, c_3, c_4, c_5] + c_6 \frac{J_{11}^2}{J_{1}^3}  \frac{qP_1P_5^2}{P_3} ,\\
[c_1, c_2, c_3, c_4, c_5; c_6]_{7} &:= \frac{1}{P_3}   [c_1, c_2, c_3, c_4, c_5] + c_{6} \frac{J_{11}^2}{J_{1}^3}  \frac{q^2P_1P_2P_3P_4}{P_5^2},\\
[c_1, c_2, c_3, c_4, c_5; c_6]_{8} &:=\frac{ qP_1}{P_4P_5}   [c_1, c_2, c_3, c_4, c_5] + c_6 \frac{J_{11}^2}{J_{1}^3}  \frac{P_3P_4^2}{P_2},\\
[c_1, c_2, c_3, c_4, c_5; c_6]_{9} &:= \frac{1}{P_4}   [c_1, c_2, c_3, c_4, c_5] + c_6 \frac{J_{11}^2}{J_{1}^3}  \frac{qP_1P_2P_4P_5}{P_3^2},\\
[c_1, c_2, c_3, c_4, c_5; c_6]_{10} &:= \frac{1}{P_5}   [c_1, c_2, c_3, c_4, c_5] + c_6 \frac{J_{11}^2}{J_{1}^3}  \frac{q^{-1}P_2P_3P_4P_5}{P_1^2}.
\end{align*}
\end{definition}
Recall the following notation:
\begin{equation*}
\Theta(a_1,a_2,a_3,a_4,a_5) := \frac{J_{11}^6}{J_{1}^2} \Big[a_1\frac{q^{2}}{P_{4}P_{5}^2}+a_2\frac{1}
{P_{1}^2P_{3}}+a_3\frac{q}{P_{1}P_{4}^2}+a_4\frac{q}{P_{2}^2P_{5}}+a_5\frac{q}{P_{2}P_{3}^2}\Big].
\end{equation*}
Rank and crank moments are defined as \cite{AtG}
\begin{align*} 
N_{k}(n) &:= \sum_{m = -\infty}^{\infty} m^k N(m,n),\\
M_{k}(n) &:= \sum_{m = -\infty}^{\infty} m^k M(m,n)
\end{align*}
for even $k \in \N$. Define $T_{a,m}(q)$ to be the elements of the $11$-dissection of the generating functions for rank and crank moments:
\begin{align*}
\sum_{n=0}^{\infty} N_k(n)q^n &=: \sum_{m=0}^{10} T_{k,m}(q^{11})q^m,\\
\sum_{n=0}^{\infty} M_k(n)q^n &=: \sum_{m=0}^{10} T^{C}_{k,m}(q^{11})q^m.
\end{align*}
The reformulation of this definition then reads
\begin{equation}\label{definition:disselemrankmoment}
T_{k,m}(q) := \sum_{n=0}^{\infty} N_k(11n+m)q^n,
\end{equation} 
\begin{equation}\label{definition:disselemcrankmoment}
T^{C}_{k,m}(q) := \sum_{n=0}^{\infty} M_k(11n+m)q^n.
\end{equation}

As an another application of Theorem \ref{theorem:crankdissection} and Theorem \ref{theorem:rankdissection} we derive new congruences for the rank and crank moments and for the Andrews’ smallest parts function $\operatorname{spt}(n)$, where $\operatorname{spt}(n)$ denotes the number of smallest parts in the partitions of $n$. For example, the partitions of the number $4$ are $(4),(3,1),(2,2),(2,1,1),(1,1,1,1)$ with $1,1,2,2,4$ being the number of smallest parts respectively, so we see $\operatorname{spt}(4) = 10$. Garvan \cite[Theorem 5.1]{G10} considered congruences for $T_{k,6}(q)$ in terms of Eisenstein series. We will provide congruences for $T^{C}_{k,m}(q)$ for residues $m \neq 6$ and congruences for $T_{k,m}(q)$ for residues $m \in \{1,2,3,5,6,8\}$ in terms of theta quotients  $\Theta(a_1,a_2,a_3,a_4,a_5)$ from Definition \ref{definition:v11G11Theta} and $[c_1,c_2,c_3,c_4,c_5;c_6]_{m}$ from Definition \ref{definition:residuebrackets}.

\begin{theorem} \label{theorem:sptcong} Using the notation of Definition \ref{definition:v11G11Theta} and Definition \ref{definition:residuebrackets}, we have
\begin{align*}
\sum_{n=0}^{\infty} \operatorname{spt}(11n+1) q^n &\equiv [1,5,-4,1,-5;1]_{1} \Mod {11},\\
\sum_{n=0}^{\infty}\operatorname{spt}(11n+2)q^n &\equiv [3,-5,3,-2,5;-5]_{2} \Mod {11},\\
\sum_{n=0}^{\infty}\operatorname{spt}(11n+3)q^n &\equiv [5,2,-2,-4,4;4]_{3} \Mod {11},\\
\sum_{n=0}^{\infty}\operatorname{spt}(11n+5)q^n &\equiv [3,1,3,4,-4;-3]_{5} \Mod {11},\\
\sum_{n=0}^{\infty}\operatorname{spt}(11n+6)q^n &\equiv \Theta(2,4,5,3,1) \Mod {11},\\
\sum_{n=0}^{\infty}\operatorname{spt}(11n+8)q^n &\equiv [-1,-2,2,-1,3;2]_{8} \Mod {11}.
\end{align*}
\end{theorem}
\begin{theorem} \label{theorem:rankmomentcong} Using the notation of Definition \ref{definition:v11G11Theta} and Definition \ref{definition:residuebrackets}, for residue $1$ modulo $11$ we have
\begin{align*}
T_{2,1}(q) &\equiv [0,-1,-5,-4,-3;-2]_{1} \Mod {11},\\
T_{4,1}(q) &\equiv 0 \Mod {11},\\
T_{6,1}(q) &\equiv [0,-1,-4,2,5;3]_{1} \Mod {11},\\
T_{8,1}(q) &\equiv [0,5,1,4,-4;-3]_{1} \Mod {11}.
\end{align*}
For residue $2$ modulo $11$ we have
\begin{align*}
T_{2,2}(q) &\equiv [2,2,-3,-4,4;-1]_{2} \Mod {11},\\
T_{4,2}(q) &\equiv [2,5,-1,0,3;1]_{2} \Mod {11},\\
T_{6,2}(q) &\equiv [2,4,-2,2,5;3]_{2} \Mod {11},\\
T_{8,2}(q) &\equiv [2,-1,2,4,1;5]_{2} \Mod {11}.
\end{align*}
For residue $3$ modulo $11$ we have
\begin{align*}
T_{2,3}(q) &\equiv [-3,0,-3,1,-4;3]_{3} \Mod {11},\\
T_{4,3}(q) &\equiv [-1,-4,2,5,4;5]_{3} \Mod {11},\\
T_{6,3}(q) &\equiv [-4,0,3,3,-2;5]_{3} \Mod {11},\\
T_{8,3}(q) &\equiv [-5,2,0,4,-5;5]_{3} \Mod {11}.
\end{align*}
For residue $5$ modulo $11$ we have
\begin{align*}
T_{2,5}(q) &\equiv [-2,5,1,-1,4;-5]_{5} \Mod {11},\\
T_{4,5}(q) &\equiv [-4,3,-3,5,-4;-2]_{5} \Mod {11},\\
T_{6,5}(q) &\equiv [-5,2,-1,0,2;1]_{5} \Mod {11},\\
T_{8,5}(q) &\equiv [4,0,-2,2,5;-3]_{5} \Mod {11}.
\end{align*}
For residue $6$ modulo $11$ we have
\begin{align*}
T_{2,6}(q) &\equiv \Theta(-4,3,1,5,-2)\Mod {11},\\
T_{4,6}(q) &\equiv \Theta(-2,-4,-5,-3,-1)\Mod {11},\\
T_{6,6}(q) &\equiv \Theta(1,5,-5,4,2)\Mod {11},\\
T_{8,6}(q) &\equiv \Theta(-5,-5,1,-2,-2)\Mod {11}.
\end{align*}
For residue $8$ modulo $11$ we have
\begin{align*}
T_{2,8}(q) &\equiv [-1,-4,-1,5,-3;-4]_{8} \Mod {11},\\
T_{4,8}(q) &\equiv [-3,5,-5,-3,-2;-5]_{8} \Mod {11},\\
T_{6,8}(q) &\equiv [3,4,3,-3,1;3]_{8} \Mod {11},\\
T_{8,8}(q) &\equiv [0,2,-1,5,-1;5]_{8} \Mod {11}.
\end{align*}
\end{theorem}

\begin{theorem} \label{theorem:crankmomentcong} We have
\begin{align*} 
T^{C}_{2,0}(q) \equiv T^{C}_{4,0}(q)  \equiv T^{C}_{6,0}(q)  \equiv T^{C}_{8,0}(q)  &\equiv 0 \Mod {11},\\
T^{C}_{2,1}(q)  \equiv T^{C}_{4,1}(q) \equiv T^{C}_{6,1}(q) \equiv T^{C}_{8,1}(q) &\equiv 2 \frac{J_{11}^2P_5}{P_2P_3} \Mod {11},\\
7T^{C}_{2,2}(q) \equiv  10T^{C}_{4,2}(q) \equiv  8T^{C}_{6,2}(q) \equiv  2T^{C}_{8,2}(q) &\equiv  \frac{J_{11}^2P_3}{P_1P_4} \Mod {11},\\
8T^{C}_{2,3}(q) \equiv  7T^{C}_{4,3}(q) \equiv  2T^{C}_{6,3}(q) \equiv  10T^{C}_{8,3}(q) &\equiv  \frac{J_{11}^2P_2}{P_1P_3} \Mod {11},\\
8T^{C}_{2,4}(q) \equiv  9T^{C}_{4,4}(q) \equiv  3T^{C}_{6,4}(q) \equiv  6T^{C}_{8,4}(q) &\equiv  \frac{J_{11}^2}{P_2} \Mod {11},\\
3T^{C}_{2,5}(q) \equiv 2T^{C}_{4,5}(q) \equiv  8T^{C}_{6,5}(q) \equiv  5T^{C}_{8,5}(q) &\equiv  \frac{J_{11}^2P_4}{P_2P_5} \Mod {11},\\
T^{C}_{2,7}(q) \equiv 7T^{C}_{4,7}(q) \equiv  10T^{C}_{6,7}(q) \equiv 5T^{C}_{8,7}(q) &\equiv  \frac{J_{11}^2}{P_3} \Mod {11},\\
7T^{C}_{2,8}(q) \equiv  9T^{C}_{4,8}(q) \equiv  9T^{C}_{6,8}(q) \equiv 7T^{C}_{8,8}(q) &\equiv  \frac{J_{11}^2qP_1}{P_4P_5} \Mod {11},\\
T^{C}_{2,9}(q) \equiv  9T^{C}_{4,9}(q) \equiv   4T^{C}_{6,9}(q) \equiv  3T^{C}_{8,9}(q) &\equiv  \frac{J_{11}^2}{P_4} \Mod {11},\\
3T^{C}_{2,10}(q) \equiv 4T^{C}_{4,10}(q) \equiv 9T^{C}_{6,10}(q) \equiv T^{C}_{8,10}(q) &\equiv \frac{J_{11}^2}{P_5} \Mod {11}.
\end{align*}
\begin{remark} The congruences among $M_k(11n+m)$ given by Theorem \ref{theorem:crankmomentcong} are corollaries of the general formulas found by Atkin and Garvan \cite[(6.6)-(6.8)]{AtG}.
\end{remark}
\end{theorem}
Using \cite[Theorem 5.1]{G10} we can also deduce congruences for Eisenstein series $E_4$ and $E_6$ defined as
\begin{equation*}
E_j(q) := 1-\frac{2n}{B_n}\sum_{n=1}^{\infty} \sigma_{j-1}(n)q^n,
\end{equation*}
where $B_n$ is the $n$-th Bernoulli number and $\sigma_{k}(n) = \sum_{d | n} d^k$.
\begin{corollary} \label{corollary:eisensteincong} Using the notation of Definition \ref{definition:v11G11Theta}, we have
\begin{align*}
E_4(q) &\equiv \frac{1}{J_1^2J_{11}}\Theta(-1,1,1,1,1) \Mod {11},\\
E_6(q) &\equiv \frac{1}{J_1^2J_{11}}\Theta(-3,1,5,4,-2) \Mod {11}.
\end{align*}
\end{corollary}

\section{Proof of Theorem \ref{theorem:crankdissection}}\label{section:proofdevcrank}
In this section we demonstrate how to obtain the $11$-dissection for the crank deviation as stated in Theorem \ref{theorem:crankdissection}.
\begin{proof}[Proof of Theorem \ref{theorem:crankdissection}] From \cite{AG} we know that the two-variable generating function for the crank has the form
\begin{equation*}
F(z;q) := \sum_{n=0}^{\infty} \sum_{m=-\infty}^{\infty} M(m,n)z^m q^n = \frac{(q)_{\infty}}{(zq)_{\infty}(z^{-1}q)_{\infty}}.
\end{equation*}
Here for $n \leq 1$ we set
\begin{equation*}
M(m,n) := \begin{cases}
    -1, &\text{if} \ (m,n) = (0,1),\\
    1, &\text{if} \ (m,n) = (0,0), (1,1), (-1,1),\\
    0, &\text{otherwise}.
\end{cases}
\end{equation*}
From this we can find the formula for the deviation of crank, which was also mentioned in \cite[(2.12)]{M}:
\begin{equation} \label{equation:devcrankformula}
D_{C}(a,r) = \frac{1}{r} \sum_{j=1}^{r-1} \zeta_{r}^{-aj}\frac{(q)_{\infty}}{(\zeta_{r}^{j}q)_{\infty}(\zeta_{r}^{-j}q)_{\infty}} =  \frac{1}{r} \sum_{j=1}^{r-1} \zeta_{r}^{-aj} F(\zeta_{r}^{j};q) 
\end{equation}
Let us take the $11$-dissection for $F(\zeta_{11};q)$, which is given in \cite[Theorem 7.1]{B} and apply an identity, which can be verified by rearranging terms:
\begin{equation*}
X_1X_2X_3X_4X_5 = J_{11}J_{121}^4,
\end{equation*}
where $X_i$ is defined in (\ref{definition:JPX}). Then we will obtain
\begin{multline*}
F(\zeta^{j}_{11};q) =  J_{121}^2 \Big( \frac{1}{X_1} + (A_1-1)q\frac{X_5}{X_2X_3} + A_2q^2\frac{X_3}{X_1X_4} +(A_3+1)q^3\frac{X_2}{X_1X_3} \\+ (A_2+A_4+1)q^4\frac{1}{X_2} - (A_2+A_4)q^5\frac{X_4}{X_2X_5} + (A_1+A_4)q^7\frac{1}{X_3} \\-(A_2+A_5+1)q^{19}\frac{X_1}{X_4X_5} - (A_4+1)q^{9}\frac{1}{X_4} - A_3q^{10}\frac{1}{X_5} \Big ),
\end{multline*}
where $A_n = \zeta_{11}^{jn} + \zeta_{11}^{-jn}$ and $\zeta_r := e^{2 \pi i / r}$ is a primitive root of unity. We put the previous expressions into (\ref{equation:devcrankformula}) with $r=11$ and directly obtain Theorem \ref{theorem:crankdissection}.
\end{proof}

\section{$11$-dissection for the deviation of the rank}\label{section:devrank}
Define $Q_{a,m}(q)$ to be the elements of the $11$-dissection of the deviation of the rank:
\begin{equation*}
D(a,11) =: \sum_{m=0}^{10}  Q_{a,m}(q^{11})q^m.
\end{equation*} 
The reformulation of this definition then reads
\begin{equation} \label{definition:rankdisselem}
Q_{a,m}(q) := \sum^{\infty}_{n = 0} \left(N(a,11,11n+m) - \frac{p(11n+m)}{11}\right)q^n.
\end{equation}
For convenience we will decompose
\begin{equation*}
Q_{a,m}(q) = Q^{\text{th}}_{a,m}(q)+Q^{\text{mck}}_{a,m}(q),
\end{equation*}
where $Q^{\text{mck}}_{a,m}$ is a mock part of $Q_{a,m}$, that is, it corresponds to the terms in $G_{11}$.  Note that $Q^{\text{mck}}_{a,m}$ can be non-zero only for $(a,m) \in \{(0,0),(1,0),(4,4),(5,4),(1,7),(2,7),(3,9),(4,9),(2,10),(3,10)\}$ as stated in Theorem \ref{theorem:rankdissection}.  For residue $0$ we have
\begin{align}
\begin{split}\label{equation:Qmockpart0_1}
Q^{\text{mck}}_{0,0}(q) &= -2q^2g(q^{2};q^{11}),
\end{split}\\
\begin{split}\label{equation:Qmockpart0_2}
Q^{\text{mck}}_{1,0}(q) &= q^2g(q^{2};q^{11}).
\end{split}
\end{align}
For residue $4$ we have
\begin{align}
\begin{split}\label{equation:Qmockpart4_1}
Q^{\text{mck}}_{4,4}(q) &=  q^{3}g(q^{4};q^{11}),
\end{split}\\
\begin{split}\label{equation:Qmockpart4_2}
Q^{\text{mck}}_{5,4}(q) &=  -q^{3}g(q^{4};q^{11}).
\end{split}
\end{align}
For residue $7$ we have
\begin{align}
\begin{split}\label{equation:Qmockpart7_1}
Q^{\text{mck}}_{1,7}(q) &=  -q^{3}g(q^{5};q^{11}),
\end{split}\\
\begin{split}\label{equation:Qmockpart7_2}
Q^{\text{mck}}_{2,7}(q) &=  q^{3}g(q^{5};q^{11}).
\end{split}
\end{align}
For residue $9$ we have
\begin{align}
\begin{split}\label{equation:Qmockpart9_1}
Q^{\text{mck}}_{3,9}(q) &=  q^{2}g(q^{3};q^{11}),
\end{split}\\
\begin{split}\label{equation:Qmockpart9_2}
Q^{\text{mck}}_{4,9}(q) &=  -q^{2}g(q^{3};q^{11}).
\end{split}
\end{align}
For residue $10$ we have
\begin{align}
\begin{split}\label{equation:Qmockpart10_1}
Q^{\text{mck}}_{2,10}(q) &=  -[q^{-1}+g(q;q^{11})],
\end{split}\\
\begin{split}\label{equation:Qmockpart10_2}
Q^{\text{mck}}_{3,10}(q) &=  q^{-1}+g(q;q^{11}).
\end{split}
\end{align}
\begin{remark} We can establish mock parts of $Q_{a,m}$ by using \cite[Theorem 4.1]{HM17}.
\end{remark}
The remainder $Q_{a,m}-Q^{\text{mck}}_{a,m}$ we will call the theta part and denote $Q^{\text{th}}_{a,m}$. Also recall a useful identity found by O'Brien \cite[p. 6]{OB}:
\begin{equation*} 
J_{1}^3 = P_5^2 P_4 - q^{2}P_1^2 P_3  - qP_4^2 P_1 - qP_2^2P_5 - qP_3^2P_2.
\end{equation*}
By Definition \ref{definition:mainbracket} this identity is transformed to
\begin{equation} \label{identity:obrien}
J_{11}^2 = [1,-1,-1,-1,-1].
\end{equation}
Note that by using identity (\ref{identity:obrien}) we can convert $v_{11}$-term of $Q_{a,m}$ to the form $[c_1,c_2,c_3,c_4,c_5;c_6]_{m}$. For example we know
\begin{equation*}
\frac{J_{11}^2P_5}{P_2P_3} =  [1,-1,-1,-1,-1;0]_{1}.
\end{equation*}

Then we can formulate Theorem \ref{theorem:rankdissection} in its full form, that is, write $\Theta_{a,m}(q)$ explicitly. The mock parts of the dissection elements for residue $0$ are stated in (\ref{equation:Qmockpart0_1}) and (\ref{equation:Qmockpart0_2}). The theta parts of the dissection elements for residue $0$ are
\begin{align*}
Q^{\text{th}}_{0,0}(q) &=  \frac{1}{11}[10,56,-32,-10,-10;22]_{0} = \frac{10}{11}\frac{J_{11}^2}{P_1} + [0,6,-2,0,0;2]_{0} =: \frac{10}{11}\frac{J_{11}^2}{P_1} + \Theta_{0,0}(q),\\
Q^{\text{th}}_{1,0}(q) &=  \frac{1}{11}[-1,-10,23,1,1;-11]_{0} = -\frac{1}{11}\frac{J_{11}^2}{P_1} + [0,-1,2,0,0;-1]_{0 }=: -\frac{1}{11}\frac{J_{11}^2}{P_1} + \Theta_{1,0}(q), \\
Q^{\text{th}}_{2,0}(q) &=  \frac{1}{11}[-1,-32,1,12,1;0]_{0} = -\frac{1}{11}\frac{J_{11}^2}{P_1}+[0,-3,0,1,0;0]_{0}=: -\frac{1}{11}\frac{J_{11}^2}{P_1} + \Theta_{2,0}(q),\\
Q^{\text{th}}_{3,0}(q) &=  \frac{1}{11}[-1,23,-21,1,12;0]_{0} = -\frac{1}{11}\frac{J_{11}^2}{P_1} +[0,2,-2,0,1;0]_{0}=: -\frac{1}{11}\frac{J_{11}^2}{P_1} + \Theta_{3,0}(q),\\
Q^{\text{th}}_{4,0}(q) &=  \frac{1}{11}[-1,-10,23,-21,1;0]_{0} = -\frac{1}{11}\frac{J_{11}^2}{P_1}+[0,-1,2,-2,0;0]_{0}=: -\frac{1}{11}\frac{J_{11}^2}{P_1} + \Theta_{4,0}(q),\\
Q^{\text{th}}_{5,0}(q) &=  \frac{1}{11}[-1,1,-10,12,-10;0]_{0} = -\frac{1}{11}\frac{J_{11}^2}{P_1} + [0,0,-1,1,-1;0]_{0}=: -\frac{1}{11}\frac{J_{11}^2}{P_1} + \Theta_{5,0}(q).
\end{align*}
Dissection elements for residue $1$ are
\begin{align*}
Q_{0,1}(q) &=  \frac{1}{11}[10,12,12,12,-10;-22]_{1} = -\frac{12}{11}\frac{J_{11}^2P_5}{P_2P_3} + [2,0,0,0,-2;-2]_{1} =:-\frac{12}{11}\frac{J_{11}^2P_5}{P_2P_3} + \Theta_{0,1}(q),\\
Q_{1,1}(q) &=  \frac{1}{11}[-1,-10,-10,1,23;0]_{1}= \frac{10}{11}\frac{J_{11}^2P_5}{P_2P_3} + [-1,0,0,1,3;0]_{1}=:\frac{10}{11}\frac{J_{11}^2P_5}{P_2P_3}+ \Theta_{1,1}(q), \\
Q_{2,1}(q) &=  \frac{1}{11}[-1,12,1,-10,1;11]_{1}= -\frac{1}{11}\frac{J_{11}^2P_5}{P_2P_3} + [0,1,0,-1,0;1]_{1}=:-\frac{1}{11}\frac{J_{11}^2P_5}{P_2P_3}+ \Theta_{2,1}(q), \\
Q_{3,1}(q) &=  \frac{1}{11}[-1,-10,12,1,-21;11]_{1}= -\frac{1}{11}\frac{J_{11}^2P_5}{P_2P_3} + [0,-1,1,0,-2;1]_{1}=:-\frac{1}{11}\frac{J_{11}^2P_5}{P_2P_3} + \Theta_{3,1}(q),\\
Q_{4,1}(q) &=  \frac{1}{11}[-1,1,-10,12,-10;0]_{1}= -\frac{1}{11}\frac{J_{11}^2P_5}{P_2P_3} + [0,0,-1,1,-1;0]_{1}=:-\frac{1}{11}\frac{J_{11}^2P_5}{P_2P_3}+ \Theta_{4,1}(q), \\
Q_{5,1}(q) &=  \frac{1}{11}[-1,1,1,-10,12;-11]_{1}  = -\frac{1}{11}\frac{J_{11}^2P_5}{P_2P_3} + [0,0,0,-1,1;-1]_{1}=:-\frac{1}{11}\frac{J_{11}^2P_5}{P_2P_3}+ \Theta_{5,1}(q).
\end{align*}
Dissection elements for residue $2$ are
\begin{align*}
Q_{0,2}(q) &=  \frac{1}{11}[-2,-20,24,2,2;-22]_{2} = -\frac{2}{11}\frac{J_{11}^2P_3}{P_1P_4}+ [0,-2,2,0,0;-2]_{2}=:-\frac{2}{11}\frac{J_{11}^2P_3}{P_1P_4}+ \Theta_{0,2}(q),\\
Q_{1,2}(q) &=  \frac{1}{11}[9,13,-31,2,2;11]_{2} = -\frac{2}{11}\frac{J_{11}^2P_3}{P_1P_4}+ [1,1,-3,0,0;1]_{2}=:-\frac{2}{11}\frac{J_{11}^2P_3}{P_1P_4}+ \Theta_{1,2}(q),\\
Q_{2,2}(q) &=  \frac{1}{11}[-2,-9,24,-9,2;11]_{2} = \frac{9}{11}\frac{J_{11}^2P_3}{P_1P_4}+ [-1,0,3,0,1;1]_{2}=:\frac{9}{11}\frac{J_{11}^2P_3}{P_1P_4}+ \Theta_{2,2}(q),\\
Q_{3,2}(q) &=  \frac{1}{11}[-2,13,13,2,-9;-22]_{2} = -\frac{2}{11}\frac{J_{11}^2P_3}{P_1P_4}+ [0,1,1,0,-1;-2]_{2}=:-\frac{2}{11}\frac{J_{11}^2P_3}{P_1P_4}+ \Theta_{3,2}(q),\\
Q_{4,2}(q) &=  \frac{1}{11}[-2,-9,-20,13,2;22]_{2} =  -\frac{2}{11}\frac{J_{11}^2P_3}{P_1P_4}+ [0,-1,-2,1,0;2]_{2}=:-\frac{2}{11}\frac{J_{11}^2P_3}{P_1P_4}+ \Theta_{4,2}(q),\\
Q_{5,2}(q) &=  \frac{1}{11}[-2,2,2,-9,2;-11]_{2}= -\frac{2}{11}\frac{J_{11}^2P_3}{P_1P_4}+ [0,0,0,-1,0;-1]_{2}=:-\frac{2}{11}\frac{J_{11}^2P_3}{P_1P_4}+ \Theta_{5,2}(q).
\end{align*}
Dissection elements for residue $3$ are
\begin{align*}
Q_{0,3}(q) &=  \frac{1}{11} [8,36,-8,-8,14;-22]_{3} =  \frac{8}{11}\frac{J_{11}^2 P_2}{P_1P_3} + [0,4,0,0,2;-2]_{3}=:\frac{8}{11}\frac{J_{11}^2 P_2}{P_1P_3}+ \Theta_{0,3}(q),\\
Q_{1,3}(q) &= \frac{1}{11} [-3,-30,3,3,3;22]_{3}= -\frac{3}{11}\frac{J_{11}^2 P_2}{P_1P_3} +[0,-3,0,0,0;2]_{3}=:-\frac{3}{11}\frac{J_{11}^2 P_2}{P_1P_3}+ \Theta_{1,3}(q),  \\
Q_{2,3}(q) &=  \frac{1}{11} [8,14,3,3,-8;-22]_{3}= -\frac{3}{11}\frac{J_{11}^2 P_2}{P_1P_3} + [1,1,0,0,-1;-2]_{3}=:-\frac{3}{11}\frac{J_{11}^2 P_2}{P_1P_3}+ \Theta_{2,3}(q), \\
Q_{3,3}(q) &=  \frac{1}{11} [-3,3,-8,3,-8;22]_{3}= -\frac{3}{11}\frac{J_{11}^2 P_2}{P_1P_3} +[0,0,-1,0,-1;2]_{3}=:-\frac{3}{11}\frac{J_{11}^2 P_2}{P_1P_3}+ \Theta_{3,3}(q), \\
Q_{4,3}(q) &=  \frac{1}{11} [-3,14,14,-8,3;-11]_{3}= -\frac{3}{11}\frac{J_{11}^2 P_2}{P_1P_3} +[0,1,1,-1,0;-1]_{3}=:-\frac{3}{11}\frac{J_{11}^2 P_2}{P_1P_3}+ \Theta_{4,3}(q),\\
Q_{5,3}(q) &=  \frac{1}{11} [-3,-19,-8,3,3;0]_{3}= -\frac{3}{11}\frac{J_{11}^2 P_2}{P_1P_3} +[0,-2,-1,0,0;0]_{3}=:-\frac{3}{11}\frac{J_{11}^2 P_2}{P_1P_3}+ \Theta_{5,3}(q).
\end{align*}
The mock parts of the dissection elements for residue $4$ are stated in (\ref{equation:Qmockpart4_1}) and (\ref{equation:Qmockpart4_2}). The theta parts of the dissection elements for residue $4$ are
\begin{align*}
Q^{\text{th}}_{0,4}(q) &=  \frac{1}{11} [6,-6,16,-28,16;0]_{4}= \frac{6}{11}\frac{J_{11}^2}{P_2}+ [0,0,2,-2,2;0]_{4}=:\frac{6}{11}\frac{J_{11}^2}{P_2}+ \Theta_{0,4}(q),\\
Q^{\text{th}}_{1,4}(q) &=  \frac{1}{11} [6,5,-17,5,5;0]_{4}= -\frac{5}{11}\frac{J_{11}^2}{P_2}+ [1,0,-2,0,0;0]_{4}=:-\frac{5}{11}\frac{J_{11}^2}{P_2}+ \Theta_{1,4}(q),\\
Q^{\text{th}}_{2,4}(q) &=  \frac{1}{11} [-5,-6,5,27,-6;0]_{4}= \frac{6}{11}\frac{J_{11}^2}{P_2}+ [-1,0,1,3,0;0]_{4}=:\frac{6}{11}\frac{J_{11}^2}{P_2}+ \Theta_{2,4}(q),\\
Q^{\text{th}}_{3,4}(q) &=  \frac{1}{11} [6,5,16,-17,-17;0]_{4}= \frac{17}{11}\frac{J_{11}^2}{P_2}+ [-1,2,3,0,0;0]_{4}=:\frac{17}{11}\frac{J_{11}^2}{P_2}+ \Theta_{3,4}(q),\\
Q^{\text{th}}_{4,4}(q) &=  \frac{1}{11} [-5,-6,-17,27,-6;-11]_{4}= \frac{6}{11}\frac{J_{11}^2}{P_2}+ [-1,0,-1,3,0;-1]_{4}=:\frac{6}{11}\frac{J_{11}^2}{P_2}+ \Theta_{4,4}(q),\\
Q^{\text{th}}_{5,4}(q) &=  \frac{1}{11} [-5,5,5,-28,16;11]_{4}= -\frac{5}{11}\frac{J_{11}^2}{P_2}+ [0,0,0,-3,1;1]_{4}=:-\frac{5}{11}\frac{J_{11}^2}{P_2}+ \Theta_{5,4}(q).
\end{align*}
Dissection elements for residue $5$ are
\begin{align*}
Q_{0,5}(q) &=  \frac{1}{11} [4,18,-4,-4,-4;22]_{5}= \frac{4}{11}\frac{J_{11}^2 P_4}{P_2P_5}+[0,2,0,0,0;2]_{5}=:\frac{4}{11}\frac{J_{11}^2 P_4}{P_2P_5}+ \Theta_{0,5}(q),\\
Q_{1,5}(q) &=  \frac{1}{11} [4,-15,7,7,-4;0]_{5}= \frac{4}{11}\frac{J_{11}^2 P_4}{P_2P_5}+[0,-1,1,1,0;0]_{5}=:\frac{4}{11}\frac{J_{11}^2 P_4}{P_2P_5}+ \Theta_{1,5}(q),\\
Q_{2,5}(q) &=  \frac{1}{11} [4,7,-4,7,7;-22]_{5}= -\frac{7}{11}\frac{J_{11}^2 P_4}{P_2P_5}+[1,0,-1,0,0;-2]_{5}=:-\frac{7}{11}\frac{J_{11}^2 P_4}{P_2P_5}+ \Theta_{2,5}(q),\\
Q_{3,5}(q) &=  \frac{1}{11} [-7,-4,-4,7,7;11]_{5}= -\frac{7}{11}\frac{J_{11}^2 P_4}{P_2P_5}+[0,-1,-1,0,0;1]_{5}=:-\frac{7}{11}\frac{J_{11}^2 P_4}{P_2P_5}+ \Theta_{3,5}(q),\\
Q_{4,5}(q) &=  \frac{1}{11} [4,7,7,-37,-4;11]_{5}= \frac{4}{11}\frac{J_{11}^2 P_4}{P_2P_5}+[0,1,1,-3,0;1]_{5}=:\frac{4}{11}\frac{J_{11}^2 P_4}{P_2P_5}+ \Theta_{4,5}(q),\\
Q_{5,5}(q) &=  \frac{1}{11} [-7,-4,-4,18,-4;-11]_{5}= \frac{4}{11}\frac{J_{11}^2 P_4}{P_2P_5}+[-1,0,0,2,0;-1]_{5}=:\frac{4}{11}\frac{J_{11}^2 P_4}{P_2P_5}+ \Theta_{5,5}(q).
\end{align*}
The mock parts of the dissection elements for residue $7$ are stated in (\ref{equation:Qmockpart7_1}) and (\ref{equation:Qmockpart7_2}). The theta parts of the dissection elements for residue $7$ are
\begin{align*}
Q^{\text{th}}_{0,7}(q) &= \frac{1}{11}[18,26,4,-18,-18;0]_{7}= \frac{18}{11}\frac{J_{11}^2}{P_3}+[0,4,2,0,0;0]_{7}=:\frac{18}{11}\frac{J_{11}^2}{P_3}+ \Theta_{0,7}(q),\\
Q^{\text{th}}_{1,7}(q) &= \frac{1}{11}[-4,-29,-7,4,37;11]_{7}= -\frac{4}{11}\frac{J_{11}^2}{P_3}+[0,-3,-1,0,3;1]_{7}=:-\frac{4}{11}\frac{J_{11}^2}{P_3}+ \Theta_{1,7}(q),\\
Q^{\text{th}}_{2,7}(q) &= \frac{1}{11}[7,26,4,15,-29;-11]_{7}= -\frac{4}{11}\frac{J_{11}^2}{P_3}+[1,2,0,1,-3;-1]_{7}=:-\frac{4}{11}\frac{J_{11}^2}{P_3}+ \Theta_{2,7}(q),\\
Q^{\text{th}}_{3,7}(q) &= \frac{1}{11}[-4,-18,4,-7,15;0]_{7}= -\frac{4}{11}\frac{J_{11}^2}{P_3}+[0,-2,0,-1,1;0]_{7}=:-\frac{4}{11}\frac{J_{11}^2}{P_3}+ \Theta_{3,7}(q),\\
Q^{\text{th}}_{4,7}(q) &= \frac{1}{11}[-4,15,-7,-7,4;0]_{7}= -\frac{4}{11}\frac{J_{11}^2}{P_3}+[0,1,-1,-1,0;0]_{7}=:-\frac{4}{11}\frac{J_{11}^2}{P_3}+ \Theta_{4,7}(q),\\
Q^{\text{th}}_{5,7}(q) &= \frac{1}{11}[-4,-7,4,4,-18;0]_{7}= -\frac{4}{11}\frac{J_{11}^2}{P_3}+[0,-1,0,0,-2;0]_{7}=:-\frac{4}{11}\frac{J_{11}^2}{P_3}+ \Theta_{5,7}(q).
\end{align*}
Dissection elements for residue $8$ are
\begin{align*}
Q_{0,8}(q) &=  \frac{1}{11}[38,6,6,-16,6;0]_{8}= -\frac{6}{11}\frac{J_{11}^2 qP_1}{P_4P_5}+[4,0,0,-2,0;0]_{8}=:-\frac{6}{11}\frac{J_{11}^2 qP_1}{P_4P_5}+ \Theta_{0,8}(q),\\
Q_{1,8}(q) &=  \frac{1}{11}[-17,-5,-5,6,6;11]_{8}= -\frac{6}{11}\frac{J_{11}^2 qP_1}{P_4P_5}+[-1,-1,-1,0,0;1]_{8}=:-\frac{6}{11}\frac{J_{11}^2 qP_1}{P_4P_5}+ \Theta_{1,8}(q),\\
Q_{2,8}(q) &=  \frac{1}{11}[16,6,-5,6,-5;0]_{8}= \frac{5}{11}\frac{J_{11}^2 qP_1}{P_4P_5}+[1,1,0,1,0;0]_{8}=:\frac{5}{11}\frac{J_{11}^2 qP_1}{P_4P_5}+ \Theta_{2,8}(q),\\
Q_{3,8}(q) &=  \frac{1}{11}[-6,-5,17,-5,-5;0]_{8}= \frac{5}{11}\frac{J_{11}^2 qP_1}{P_4P_5}+[-1,0,2,0,0;0]_{8}=:\frac{5}{11}\frac{J_{11}^2 qP_1}{P_4P_5}+ \Theta_{3,8}(q),\\
Q_{4,8}(q) &=  \frac{1}{11}[-17,6,-16,-5,-5;0]_{8}=\frac{5}{11}\frac{J_{11}^2 qP_1}{P_4P_5}+[-2,1,-1,0,0;0]_{8}=: \frac{5}{11}\frac{J_{11}^2 qP_1}{P_4P_5}+ \Theta_{4,8}(q),\\
Q_{5,8}(q) &=  \frac{1}{11}[5,-5,6,6,6;-11]_{8}= -\frac{6}{11}\frac{J_{11}^2 qP_1}{P_4P_5}+[1,-1,0,0,0;-1]_{8}=:-\frac{6}{11}\frac{J_{11}^2 qP_1}{P_4P_5}+ \Theta_{5,8}(q).
\end{align*}
The mock parts of the dissection elements for residue $9$ are stated in (\ref{equation:Qmockpart9_1}) and (\ref{equation:Qmockpart9_2}). The theta parts of the dissection elements for residue $9$ are
\begin{align*}
Q^{\text{th}}_{0,9}(q) &=  \frac{1}{11} [14,8,-36,8,8;0]_{9}= -\frac{8}{11}\frac{J_{11}^2}{P_4}+ [2,0,-4,0,0;0]_{9}=:-\frac{8}{11}\frac{J_{11}^2}{P_4}+ \Theta_{0,9}(q),\\
Q^{\text{th}}_{1,9}(q) &=  \frac{1}{11} [3,-14,19,-3,8;0]_{9}= \frac{3}{11}\frac{J_{11}^2}{P_4}+ [0,-1,2,0,1;0]_{9}=:\frac{3}{11}\frac{J_{11}^2}{P_4}+ \Theta_{1,9}(q),\\
Q^{\text{th}}_{2,9}(q) &=  \frac{1}{11} [3,19,8,-3,-3;0]_{9}= \frac{3}{11}\frac{J_{11}^2}{P_4}+ [0,2,1,0,0;0]_{9}=:\frac{3}{11}\frac{J_{11}^2}{P_4}+ \Theta_{2,9}(q),\\
Q^{\text{th}}_{3,9}(q) &=  \frac{1}{11} [3,-14,19,-3,-14;-11]_{9}= \frac{3}{11}\frac{J_{11}^2}{P_4}+ [0,-1,2,0,-1;-1]_{9}=:\frac{3}{11}\frac{J_{11}^2}{P_4}+ \Theta_{3,9}(q),\\
Q^{\text{th}}_{4,9}(q) &=  \frac{1}{11} [-8,8,-14,8,-3;11]_{9}= -\frac{8}{11}\frac{J_{11}^2}{P_4}+ [0,0,-2,0,-1;1]_{9}=:-\frac{8}{11}\frac{J_{11}^2}{P_4}+ \Theta_{4,9}(q),\\
Q^{\text{th}}_{5,9}(q) &=  \frac{1}{11} [-8,-3,-14,-3,8;0]_{9}= \frac{3}{11}\frac{J_{11}^2}{P_4}+ [-1,0,-1,0,1;0]_{9}=:\frac{3}{11}\frac{J_{11}^2}{P_4}+ \Theta_{5,9}(q).
\end{align*}
The mock parts of the dissection elements for residue $10$ are stated in (\ref{equation:Qmockpart10_1}) and (\ref{equation:Qmockpart10_2}). The theta parts of the dissection elements for residue $10$ are
\begin{align*}
Q^{\text{th}}_{0,10}(q) &=  \frac{1}{11} [2,-24,20,20,-2;0]_{10}=\frac{2}{11}\frac{J_{11}^2}{P_5} + [0,-2,2,2,0;0]_{10}=:\frac{2}{11}\frac{J_{11}^2}{P_5}+ \Theta_{0,10}(q),\\
Q^{\text{th}}_{1,10}(q) &=  \frac{1}{11} [13,20,-13,-2,-2;0]_{10}= \frac{2}{11}\frac{J_{11}^2}{P_5} + [1,2,-1,0,0;0]_{10}=:\frac{2}{11}\frac{J_{11}^2}{P_5}+ \Theta_{1,10}(q),\\
Q^{\text{th}}_{2,10}(q) &=  \frac{1}{11} [-42,-13,-2,-24,-2;11]_{10}=\frac{2}{11}\frac{J_{11}^2}{P_5} + [-4,-1,0,-2,0;1]_{10}=:\frac{2}{11}\frac{J_{11}^2}{P_5}+ \Theta_{2,10}(q),\\
Q^{\text{th}}_{3,10}(q) &=  \frac{1}{11} [46,9,9,9,9;-11]_{10}=-\frac{9}{11}\frac{J_{11}^2}{P_5} + [5,0,0,0,0;-1]_{10}=:-\frac{9}{11}\frac{J_{11}^2}{P_5}+ \Theta_{3,10}(q),\\
Q^{\text{th}}_{4,10}(q) &=  \frac{1}{11} [-9,-13,-13,20,-2;0]_{10}= \frac{13}{11}\frac{J_{11}^2}{P_5} + [-2,0,0,3,1;0]_{10}=:\frac{13}{11}\frac{J_{11}^2}{P_5}+ \Theta_{4,10}(q),\\
Q^{\text{th}}_{5,10}(q) &=  \frac{1}{11} [-9,9,9,-13,-2;0]_{10}= -\frac{9}{11}\frac{J_{11}^2}{P_5} + [0,0,0,-2,-1;0]_{10}=:-\frac{9}{11}\frac{J_{11}^2}{P_5}+ \Theta_{5,10}(q).
\end{align*}

In the next proof we will show how to obtain these formulas.

\begin{proof}[Proof of Theorem \ref{theorem:rankdissection}]
Recall that the rank generating function has the form
\begin{equation*}
R(z;q) := \sum_{n=0}^{\infty} \sum_{m=-\infty}^{\infty} N(m,n)z^m q^n = \sum_{n=0}^{\infty} \frac{q^{n^2}}{(zq)_n(z^{-1}q)_n}.
\end{equation*}
Let the $P_{j,m}(q)$ to be elements of the $11$-dissection of $R(\zeta_{11}^j;q)$:
\begin{equation*}
R(\zeta_{11}^j;q) =: \sum_{m=0}^{10}  P_{j,m}(q^{11})q^m.
\end{equation*}
So reformulation of this definition is
\begin{equation*} 
P_{j,m}(q) := \sum_{n=0}^{\infty}\left(\sum_{k=0}^{10}N(k,11,11n+m)\zeta_{11}^{kj}\right)q^n.
\end{equation*}
Note that $P_{j,m}(q)$ can be found from $P_{1,m}(q)$ by taking $\zeta_{11}^{j}$ instead of $\zeta_{11}$. Recall the definition of the Dedekind eta-function 
\begin{gather*}
\eta(z) := q^{\frac{1}{24}}\prod_{n=1}^{\infty} (1-q^n)
\end{gather*}
and the notation of Biagioli \cite{Bia}
\begin{gather*} 
f_{p,k}(z) := q^{\frac{(p-2k)^2}{8p}} (q^k;q^p)_{\infty}(q^{p-k};q^p)_{\infty}(q^p;q^p)_{\infty},\\
j(p, \overrightarrow{n},z) := \eta(pz)^{n_0} \prod_{k=1}^{\frac{1}{2}(p-1)} f_{p,k}(z)^{n_k}.
\end{gather*}
where $p \geq 1$, $p \nmid k$, $\overrightarrow{n} \in \Z^{\frac{1}{2}(p+1)}$. In \cite[(6.13)]{G19} Garvan found $P_{1,6}(q)$
\begin{equation} \label{equation:garvanformula6}
(q^{11};q^{11})_{\infty} \sum_{n=0}^{\infty}\left(\sum_{k=0}^{10}N(k,11,11n+6)\zeta_{11}^{k}\right)q^{n+1} = \sum_{r=1}^{5} c_{6, r} j(11, \pi_{r}(\overrightarrow{n}),z)
\end{equation}
where $c_{6, r} \in \Z[\zeta_{11}]$ are given explicitly in \cite[Section 6.4]{G19}, $\overrightarrow{n}=(15,-4,-2,-3,-2,-2)$ and $\pi_{r}$ is the permutation on $\{1,2,3,4,5\}$, defined as $\pi_{r}(i) = i'$, where $ri'\equiv \pm i \Mod {11}$.
As described in \cite[(1.9)]{GS} for $p>3$ prime, $1\leq a \leq \frac{1}{2}(p-1)$ we denote
\begin{gather*} 
\Phi_{p,a}(q) := \begin{cases}
    \sum\limits_{n=0}^{\infty} \frac{q^{pn^2}}{(q^a;q^p)_{n+1}(q^{p-a};q^p)_{n}}, &\text{if} \ 0<6a<p,\\
    -1+\sum\limits_{n=0}^{\infty} \frac{q^{pn^2}}{(q^a;q^p)_{n+1}(q^{p-a};q^p)_{n}}, &\text{if} \ p<6a<3p.
\end{cases}
\end{gather*}
Recently Garvan and Sarma found $P_{1,7}(q)$ \cite[(7.5)]{GS}
\begin{multline} \label{equation:garvanformula7}
q^{\frac{1}{11}}(q^{11};q^{11})_{\infty} \left(\sum_{n=0}^{\infty}\left(\sum_{k=0}^{10}N(k,11,11n+7)\zeta_{11}^{k}\right ) q^{n+1} + a_{7}q^{-1} \Phi_{11,5}(q)\right) \\= \frac{f_{11,5}(z)}{f_{11,1}(z)}\sum_{r=1}^{5} c_{7,r} j(11, \pi_{r}(\overrightarrow{n}),z) + \frac{f_{11,4}(z)}{f_{11,5}(z)}\sum_{r=1}^{5} d_{7,r} j(11, \pi_{r}(\overrightarrow{n}),z)
\end{multline}
where $a_{7}, c_{7,r}, d_{7,r} \in \Z[\zeta_{11}]$ are given explicitly in \cite[Section 7.1.2]{GS}.
Also they found $P_{1, 8}(q)$ \cite[(7.6)]{GS}
\begin{multline}\label{equation:garvanformula8}
q^{\frac{2}{11}}(q^{11};q^{11})_{\infty} \sum_{n=0}^{\infty}\left(\sum_{k=0}^{10}N(k,11,11n+8)\zeta_{11}^{k}\right ) q^{n+1} \\= \frac{f_{11,4}(z)}{f_{11,1}(z)}\sum_{r=1}^{5} c_{8,r} j(11, \pi_{r}(\overrightarrow{n}),z) + \frac{f_{11,3}(z)}{f_{11,4}(z)}\sum_{r=1}^{5} d_{8,r} j(11, \pi_{r}(\overrightarrow{n}),z),
\end{multline} 
where $c_{8,r}, d_{8,r} \in \Z[\zeta_{11}]$ are given explicitly in \cite[Section 7.1.3]{GS}. Then using \cite[Theorem 4.11]{GS} we can find other $P_{1, m}(q)$ using formulas for $P_{1, 7}(q)$ and $P_{1, 8}(q)$. We derive
\begin{multline} \label{equation:garvanformula0}
q^{\frac{5}{11}}(q^{11};q^{11})_{\infty} \left(\sum_{n=0}^{\infty}\left(\sum_{k=0}^{10}N(k,11,11n)\zeta_{11}^{k}\right ) q^{n} + a_{0}\Phi_{11,2}(q)\right) \\= \frac{f_{11,2}(z)}{f_{11,4}(z)}\sum_{r=1}^{5} c_{0,r} j(11, \pi_{r}(\overrightarrow{n}),z) + \frac{f_{11,5}(z)}{f_{11,2}(z)}\sum_{r=1}^{5} d_{0,r} j(11, \pi_{r}(\overrightarrow{n}),z), 
\end{multline}
\begin{multline}\label{equation:garvanformula1}
q^{\frac{6}{11}}(q^{11};q^{11})_{\infty} \sum_{n=0}^{\infty}\left(\sum_{k=0}^{10}N(k,11,11n+1)\zeta_{11}^{k}\right ) q^{n} \\= \frac{f_{11,2}(z)}{f_{11,5}(z)}\sum_{r=1}^{5} c_{1,r} j(11, \pi_{r}(\overrightarrow{n}),z) + \frac{f_{11,4}(z)}{f_{11,2}(z)}\sum_{r=1}^{5} d_{1,r} j(11, \pi_{r}(\overrightarrow{n}),z),
\end{multline} 
\begin{multline}\label{equation:garvanformula2}
q^{\frac{7}{11}}(q^{11};q^{11})_{\infty} \sum_{n=0}^{\infty}\left(\sum_{k=0}^{10}N(k,11,11n+2)\zeta_{11}^{k}\right ) q^{n} \\= \frac{f_{11,1}(z)}{f_{11,3}(z)}\sum_{r=1}^{5} c_{2,r} j(11, \pi_{r}(\overrightarrow{n}),z) + \frac{f_{11,2}(z)}{f_{11,1}(z)}\sum_{r=1}^{5} d_{2,r} j(11, \pi_{r}(\overrightarrow{n}),z),
\end{multline}
\begin{multline}\label{equation:garvanformula3}
q^{\frac{8}{11}}(q^{11};q^{11})_{\infty} \sum_{n=0}^{\infty}\left(\sum_{k=0}^{10}N(k,11,11n+3)\zeta_{11}^{k}\right ) q^{n} \\= \frac{f_{11,3}(z)}{f_{11,2}(z)}\sum_{r=1}^{5} c_{3,r} j(11, \pi_{r}(\overrightarrow{n}),z) + \frac{f_{11,5}(z)}{f_{11,3}(z)}\sum_{r=1}^{5} d_{3,r} j(11, \pi_{r}(\overrightarrow{n}),z),
\end{multline}
\begin{multline} \label{equation:garvanformula4}
q^{\frac{9}{11}}(q^{11};q^{11})_{\infty} \left(\sum_{n=0}^{\infty}\left(\sum_{k=0}^{10}N(k,11,11n+4)\zeta_{11}^{k}\right ) q^{n} + a_{4}q^{-1}\Phi_{11,4}(q)\right) \\= \frac{f_{11,4}(z)}{f_{11,3}(z)}\sum_{r=1}^{5} c_{4,r} j(11, \pi_{r}(\overrightarrow{n}),z) + \frac{f_{11,1}(z)}{f_{11,4}(z)}\sum_{r=1}^{5} d_{4,r} j(11, \pi_{r}(\overrightarrow{n}),z), 
\end{multline}
\begin{multline}\label{equation:garvanformula5}
q^{\frac{10}{11}}(q^{11};q^{11})_{\infty} \sum_{n=0}^{\infty}\left(\sum_{k=0}^{10}N(k,11,11n+5)\zeta_{11}^{k}\right ) q^{n} \\= \frac{f_{11,5}(z)}{f_{11,4}(z)}\sum_{r=1}^{5} c_{5,r} j(11, \pi_{r}(\overrightarrow{n}),z) + \frac{f_{11,1}(z)}{f_{11,5}(z)}\sum_{r=1}^{5} d_{5,r} j(11, \pi_{r}(\overrightarrow{n}),z),
\end{multline}
\begin{multline} \label{equation:garvanformula9}
q^{\frac{3}{11}}(q^{11};q^{11})_{\infty} \left(\sum_{n=0}^{\infty}\left(\sum_{k=0}^{10}N(k,11,11n+9)\zeta_{11}^{k}\right ) q^{n+1} + a_{9}\Phi_{11,3}(q)\right) \\=\frac{f_{11,3}(z)}{f_{11,5}(z)}\sum_{r=1}^{5} c_{9,r} j(11, \pi_{r}(\overrightarrow{n}),z) + \frac{f_{11,2}(z)}{f_{11,3}(z)}\sum_{r=1}^{5} d_{9,r} j(11, \pi_{r}(\overrightarrow{n}),z), 
\end{multline}
\begin{multline} \label{equation:garvanformula10}
q^{\frac{4}{11}}(q^{11};q^{11})_{\infty} \left(\sum_{n=0}^{\infty}\left(\sum_{k=0}^{10}N(k,11,11n+10)\zeta_{11}^{k}\right ) q^{n+1} + a_{10}\Phi_{11,1}(q)\right) \\= \frac{f_{11,1}(z)}{f_{11,2}(z)}\sum_{r=1}^{5} c_{10,r} j(11, \pi_{r}(\overrightarrow{n}),z) + \frac{f_{11,3}(z)}{f_{11,1}(z)}\sum_{r=1}^{5} d_{10,r} j(11, \pi_{r}(\overrightarrow{n}),z),
\end{multline}
where $a_{i}, c_{i,r}, d_{i,r} \in \Z[\zeta_{11}]$ can be found explicitly from $a_{7}, c_{7,r}, d_{7,r}$ and $c_{8,r}, d_{8,r}$ by \cite[Theorem 4.11]{GS}. 

Then we transfer formulas (\ref{equation:garvanformula6})-(\ref{equation:garvanformula10}) to a different notation using
\begin{equation*}
f_{p,k}(z) = q^{\frac{(p-2k)^2}{8p}} J_{k, p}
\end{equation*}
and
\begin{equation*}
\Phi_{p,a}(q) = \begin{cases}
    1+q^ag(q^a;q^p), &\text{if} \ 0<6a<p,\\
    q^ag(q^a;q^p), &\text{if} \ p<6a<3p.
\end{cases}
\end{equation*}
Recall the formula for the deviation of the rank \cite[(2.10)]{M}
\begin{equation*}
D(a,r) = \frac{1}{r} \sum_{j=1}^{r-1} \zeta_{r}^{-aj}\Big(1-\zeta_r^{j}\Big)\Big(1+\zeta_r^{j}g(\zeta_r^{j};q)\Big) = \frac{1}{r} \sum_{j=1}^{r-1} \zeta_{r}^{-aj}R(\zeta_r^{j};q).
\end{equation*}
As with the crank, if we know the $11$-dissection elements $P_{j,m}(q)$ for $R(\zeta_r^{j};q)$, we can deduce the $11$-dissection elements $Q_{a,m}(q)$ for $D(a,r)$ using the previous formula. We defined earlier 
\begin{equation} \label{definition:theta2def}
\Theta(a_1,a_2,a_3,a_4,a_5) := \frac{J_{11}^6}{J_{1}^2} \Big[a_1\frac{q^{2}}{P_{4}P_{5}^2}+a_2\frac{1}
{P_{1}^2P_{3}}+a_3\frac{q}{P_{1}P_{4}^2}+a_4\frac{q}{P_{2}^2P_{5}}+a_5\frac{q}{P_{2}P_{3}^2}\Big].
\end{equation}
Also recall an important identity, which can be obtained by rearranging terms:
\begin{equation} \label{identity:rearrangingtermsP}
P_1P_2P_3P_4P_5 = J_{1}J_{11}^4.
\end{equation}
Note that (\ref{definition:theta2def}) corresponds to the right-hand side of (\ref{equation:garvanformula6}) after changing the notation and applying (\ref{identity:rearrangingtermsP}):
\begin{align*}
j(11, \pi_{1}(\overrightarrow{n}),z) &= \frac{qJ_{11}^{15}}{P_1^{4}P_2^{2}P_3^{3}P_4^{2}P_5^{2}} = \frac{J_{11}^7}{J_1^2} \cdot \frac{q}{P_1^2P_3},\\
j(11, \pi_{2}(\overrightarrow{n}),z) &= \frac{q^2J_{11}^{15}}{P_1^{2}P_2^{4}P_3^{2}P_4^{2}P_5^{3}} = \frac{J_{11}^7}{J_1^2} \cdot \frac{q^2}{P_2^2P_5},\\
j(11, \pi_{3}(\overrightarrow{n}),z) &= \frac{q^2J_{11}^{15}}{P_1^{2}P_2^{3}P_3^{4}P_4^{2}P_5^{2}}) = \frac{J_{11}^7}{J_1^2} \cdot \frac{q^2}{P_2P_3^2},\\
j(11, \pi_{4}(\overrightarrow{n}),z)&= \frac{q^2J_{11}^{15}}{P_1^{3}P_2^{2}P_3^{2}P_4^{4}P_5^{2}} = \frac{J_{11}^7}{J_1^2} \cdot \frac{q^2}{P_1P_4^2},\\
j(11, \pi_{5}(\overrightarrow{n}),z)&= \frac{q^3J_{11}^{15}}{P_1^{2}P_2^{2}P_3^{2}P_4^{3}P_5^{4}} = \frac{J_{11}^7}{J_1^2} \cdot \frac{q^3}{P_4P_5^2}.
\end{align*}
By calculations we see from (\ref{equation:garvanformula6}) that
\begin{align*}
Q_{0,6}(q) &= \frac{1}{11}\sum_{j=1}^{10} P_{j,6}(q) = \Theta(0,0,2,2,-2),\\
Q_{1,6}(q) &= \frac{1}{11}\sum_{j=1}^{10} \zeta_{11}^{-j}P_{j,6}(q) =\Theta(-1,1,-1,-2,1),\\
Q_{2,6}(q) &= \frac{1}{11}\sum_{j=1}^{10} \zeta_{11}^{-2j}P_{j,6}(q) =\Theta(1,0,-1,2,0),\\
Q_{3,6}(q) &= \frac{1}{11}\sum_{j=1}^{10} \zeta_{11}^{-3j}P_{j,6}(q) =\Theta(1,0,1,-1,-1),\\
Q_{4,6}(q) &= \frac{1}{11}\sum_{j=1}^{10} \zeta_{11}^{-4j}P_{j,6}(q) =\Theta(0,-1,1,0,2),\\
Q_{5,6}(q) &= \frac{1}{11}\sum_{j=1}^{10} \zeta_{11}^{-5j}P_{j,6}(q) =\Theta(-1,0,-1,0,-1) . 
\end{align*} 
In the same way from (\ref{equation:garvanformula7}) we see that
\begin{align}
\begin{split}\label{equation:thetaformQ7}
Q_{0,7}(q) &= q^{-1}\frac{P_5}{11P_1}\Theta(-18,0,-18,26,10) + \frac{P_4}{11P_5}\Theta(0,0,0,0,4),\\
Q_{1,7}(q) &= -q^{3}g(q^{5};q^{11}) + q^{-1}\frac{P_5}{11P_1}\Theta(37,0,4,-29,21) + \frac{P_4}{11P_5}\Theta(11,0,0,0,-7),\\
Q_{2,7}(q) &= q^{3}g(q^{5};q^{11})+q^{-1}\frac{P_5}{11P_1}\Theta(-29,0,15,26,-34) + \frac{P_4}{11P_5}\Theta(-11,0,0,0,4),\\
Q_{3,7}(q) &= q^{-1}\frac{P_5}{11P_1}\Theta(15,0,-7,-18,21) + \frac{P_4}{11P_5}\Theta(0,0,0,0,4),\\
Q_{4,7}(q) &= q^{-1}\frac{P_5}{11P_1}\Theta(4,0,-7,15,-12) + \frac{P_4}{11P_5}\Theta(0,0,0,0,-7),\\
Q_{5,7}(q) &= q^{-1}\frac{P_5}{11P_1}\Theta(-18,0,4,-7,-1) + \frac{P_4}{11P_5}\Theta(0,0,0,0,4).
\end{split}
\end{align}
From (\ref{equation:garvanformula8}) we see that
\begin{align} 
\begin{split}\label{equation:thetaformQ8}
Q_{0,8}(q) &= q^{-1}\frac{P_4}{11P_1}\Theta(6,0,42,-36,-6) + \frac{P_3}{11P_4}\Theta(0,0,-16,6,0),\\
Q_{1,8}(q) &= q^{-1}\frac{P_4}{11P_1}\Theta(6,0,-24,30,5) + \frac{P_3}{11P_4}\Theta(0,0,6,-5,0),\\
Q_{2,8}(q) &= q^{-1}\frac{P_4}{11P_1}\Theta(-5,0,9,-14,5) + \frac{P_3}{11P_4}\Theta(0,0,6,6,0),\\
Q_{3,8}(q) &= q^{-1}\frac{P_4}{11P_1}\Theta(-5,0,9,8,-17) + \frac{P_3}{11P_4}\Theta(0,0,-5,-5,0),\\
Q_{4,8}(q) &= q^{-1}\frac{P_4}{11P_1}\Theta(-5,0,-13,-3,16) + \frac{P_3}{11P_4}\Theta(0,0,-5,6,0),\\
Q_{5,8}(q) &= q^{-1}\frac{P_4}{11P_1}\Theta(6,0,-2,-3,-6) + \frac{P_3}{11P_4}\Theta(0,0,6,-5,0).
\end{split}
\end{align}
Formulas for $Q_{a,m}(q)$ in terms of $\Theta(a_1,a_2,a_3,a_4,a_5)$ for other residues can be also of interest, so we write them explicitly. From (\ref{equation:garvanformula0}) we see that
\begin{align*} 
Q_{0,0}(q) &= -2q^2g(q^{2},q^{11})+q\frac{P_2}{11P_4}\Theta(-10,-52,0,56,32) + \frac{P_5}{11P_2}\Theta(0,10,0,22,0),\\
Q_{1,0}(q) &= q^2g(q^{2},q^{11})+q\frac{P_2}{11P_4}\Theta(1,25,0,-10,-23) + \frac{P_5}{11P_2}\Theta(0,-1,0,-11,0),\\
Q_{2,0}(q) &= q\frac{P_2}{11P_4}\Theta(1,14,0,-32,-1) + \frac{P_5}{11P_2}\Theta(0,-1,0,0,0),\\
Q_{3,0}(q) &= q\frac{P_2}{11P_4}\Theta(12,-8,0,23,21) + \frac{P_5}{11P_2}\Theta(0,-1,0,0,0),\\
Q_{4,0}(q) &= q\frac{P_2}{11P_4}\Theta(1,3,0,-10,-23) + \frac{P_5}{11P_2}\Theta(0,-1,0,0,0),\\
Q_{5,0}(q) &= q\frac{P_2}{11P_4}\Theta(-10,-8,0,1,10) + \frac{P_5}{11P_2}\Theta(0,-1,0,0,0).
\end{align*}
From (\ref{equation:garvanformula1}) we see that
\begin{align*} 
Q_{0,1}(q) &= q\frac{P_2}{11P_5}\Theta(0,-6,-12,-4,-12) + \frac{P_4}{11P_2}\Theta(0,10,0,-12,0),\\
Q_{1,1}(q) &= q\frac{P_2}{11P_5}\Theta(0,5,-1,-4,10) + \frac{P_4}{11P_2}\Theta(0,-1,0,10,0),\\
Q_{2,1}(q) &= q\frac{P_2}{11P_5}\Theta(0,16,10,-15,-1) + \frac{P_4}{11P_2}\Theta(0,-1,0,-12,0),\\
Q_{3,1}(q) &= q\frac{P_2}{11P_5}\Theta(0,-6,-1,18,-12) + \frac{P_4}{11P_2}\Theta(0,-1,0,10,0),\\
Q_{4,1}(q) &= q\frac{P_2}{11P_5}\Theta(0,-6,-12,18,10) + \frac{P_4}{11P_2}\Theta(0,-1,0,-1,0),\\
Q_{5,1}(q) &= q\frac{P_2}{11P_5}\Theta(0,-6,10,-15,-1) + \frac{P_4}{11P_2}\Theta(0,-1,0,-1,0).
\end{align*}
From (\ref{equation:garvanformula2}) we see that
\begin{align*} 
Q_{0,2}(q) &= q\frac{P_1}{11P_3}\Theta(-12,30,2,-20,0) + \frac{P_2}{11P_1}\Theta(2,-2,0,0,0),\\
Q_{1,2}(q) &= q\frac{P_1}{11P_3}\Theta(-12,-36,2,13,0) + \frac{P_2}{11P_1}\Theta(2,9,0,0,0),\\
Q_{2,2}(q) &= q\frac{P_1}{11P_3}\Theta(21,19,-9,-9,0) + \frac{P_2}{11P_1}\Theta(2,-2,0,0,0),\\
Q_{3,2}(q) &= q\frac{P_1}{11P_3}\Theta(-1,8,2,13,0) + \frac{P_2}{11P_1}\Theta(-9,-2,0,0,0),\\
Q_{4,2}(q) &= q\frac{P_1}{11P_3}\Theta(10,-3,13,-9,0) + \frac{P_2}{11P_1}\Theta(2,-2,0,0,0),\\
Q_{5,2}(q) &= q\frac{P_1}{11P_3}\Theta(-12,-3,-9,2,0) + \frac{P_2}{11P_1}\Theta(2,-2,0,0,0).
\end{align*}
From (\ref{equation:garvanformula3}) we see that
\begin{align*} 
Q_{0,3}(q) &= \frac{P_3}{11P_2}\Theta(14,8,4,0,-12) + \frac{P_5}{11P_3}\Theta(0,0,-8,0,-8),\\
Q_{1,3}(q) &= \frac{P_3}{11P_2}\Theta(3,-3,4,0,21) + \frac{P_5}{11P_3}\Theta(0,0,3,0,3),\\
Q_{2,3}(q) &= \frac{P_3}{11P_2}\Theta(-8,8,-18,0,-12) + \frac{P_5}{11P_3}\Theta(0,0,3,0,3),\\
Q_{2,3}(q) &= \frac{P_3}{11P_2}\Theta(-8,-3,15,0,-1) + \frac{P_5}{11P_3}\Theta(0,0,3,0,-8),\\
Q_{2,4}(q) &= \frac{P_3}{11P_2}\Theta(3,-3,15,0,-23) + \frac{P_5}{11P_3}\Theta(0,0,-8,0,14),\\
Q_{2,5}(q) &= \frac{P_3}{11P_2}\Theta(3,-3,-18,0,21) + \frac{P_5}{11P_3}\Theta(0,0,3,0,-8).\\
\end{align*}
From (\ref{equation:garvanformula4}) we see that
\begin{align*} 
Q_{0,4}(q) &= \frac{P_4}{11P_3}\Theta(16,6,-28,26,0) + q\frac{P_1}{11P_4}\Theta(0,0,0,-6,0),\\
Q_{1,4}(q) &= \frac{P_4}{11P_3}\Theta(5,6,5,-18,0) + q\frac{P_1}{11P_4}\Theta(0,0,0,5,0),\\
Q_{2,4}(q) &= \frac{P_4}{11P_3}\Theta(-6,-5,27,4,0) + q\frac{P_1}{11P_4}\Theta(0,0,0,-6,0),\\
Q_{3,4}(q) &= \frac{P_4}{11P_3}\Theta(-17,6,-17,-7,0) + q\frac{P_1}{11P_4}\Theta(0,0,0,5,0),\\
Q_{4,4}(q) &= q^{3}g(q^{4};q^{11})+\frac{P_4}{11P_3}\Theta(-6,-5,27,-18,0) + q\frac{P_1}{11P_4}\Theta(0,0,-11,-6,0),\\
Q_{5,4}(q) &= -q^{3}g(q^{4};q^{11})+\frac{P_4}{11P_3}\Theta(16,-5,-28,26,0) + q\frac{P_1}{11P_4}\Theta(0,0,11,5,0).\\
\end{align*}
From (\ref{equation:garvanformula5}) we see that
\begin{align*} 
Q_{0,5}(q) &= \frac{P_5}{11P_4}\Theta(6,4,0,-18,24) + q\frac{P_1}{11P_5}\Theta(-4,0,0,0,4),\\
Q_{1,5}(q) &= \frac{P_5}{11P_4}\Theta(-5,4,0,15,-9) + q\frac{P_1}{11P_5}\Theta(-4,0,0,0,-7),\\
Q_{2,5}(q) &= \frac{P_5}{11P_4}\Theta(17,4,0,-7,-9) + q\frac{P_1}{11P_5}\Theta(7,0,0,0,4),\\
Q_{3,5}(q) &= \frac{P_5}{11P_4}\Theta(6,-7,0,4,24) + q\frac{P_1}{11P_5}\Theta(7,0,0,0,4),\\
Q_{4,5}(q) &= \frac{P_5}{11P_4}\Theta(-27,4,0,-7,-20) + q\frac{P_1}{11P_5}\Theta(-4,0,0,0,-7),\\
Q_{5,5}(q) &= \frac{P_5}{11P_4}\Theta(6,-7,0,4,2) + q\frac{P_1}{11P_5}\Theta(-4,0,0,0,4).
\end{align*}
From (\ref{equation:garvanformula9}) we see that
\begin{align*} 
Q_{0,9}(q) &= \frac{P_3}{11P_5}\Theta(0,14,2,-8,-36) + \frac{P_2}{11P_3}\Theta(0,0,8,0,0),\\
Q_{1,9}(q) &= \frac{P_3}{11P_5}\Theta(0, 3,-9,14,19) + \frac{P_2}{11P_3}\Theta(0,0,-3,0,0),\\
Q_{2,9}(q) &= \frac{P_3}{11P_5}\Theta(0, 3,13,-19,8) + \frac{P_2}{11P_3}\Theta(0,0,-3,0,0),\\
Q_{3,9}(q) &= q^{2}g(q^{3};q^{11})+\frac{P_3}{11P_5}\Theta(0,3,-31,14,19) + \frac{P_2}{11P_3}\Theta(0,0,-3,0,-11),\\
Q_{4,9}(q) &= -q^{2}g(q^{3};q^{11})+\frac{P_3}{11P_5}\Theta(0,-8,13,-8,-14) + \frac{P_2}{11P_3}\Theta(0,0,8,0,11),\\
Q_{5,9}(q) &= \frac{P_3}{11P_5}\Theta(0,-8,13,3,-14) + \frac{P_2}{11P_3}\Theta(0,0,-3,0,0).
\end{align*}
From (\ref{equation:garvanformula10}) we see that
\begin{align*} 
Q_{0,10}(q) &= \frac{P_1}{11P_2}\Theta(16,2,20,0,20) + q^{-1}\frac{P_3}{11P_1}\Theta(-2,0,0,0,0),\\
Q_{1,10}(q) &= \frac{P_1}{11P_2}\Theta(5,13,-2,0,-13) + q^{-1}\frac{P_3}{11P_1}\Theta(-2,0,0,0,0),\\
Q_{2,10}(q) &= -[q^{-1}+g(q;q^{11})]+\frac{P_1}{11P_2}\Theta(-39,-42,-24,0,-2) + q^{-1}\frac{P_3}{11P_1}\Theta(-2,11,0,0,0),\\
Q_{3,10}(q) &= q^{-1}+g(q;q^{11})+\frac{P_1}{11P_2}\Theta(27,46,9,0,9) + q^{-1}\frac{P_3}{11P_1}\Theta(9,-11,0,0,0),\\
Q_{4,10}(q) &= \frac{P_1}{11P_2}\Theta(-6,-9,20,0,-13) + q^{-1}\frac{P_3}{11P_1}\Theta(-2,0,0,0,0),\\
Q_{5,10}(q) &= \frac{P_1}{11P_2}\Theta(5,-9,-13,0,9) + q^{-1}\frac{P_3}{11P_1}\Theta(-2,0,0,0,0).\\
\end{align*}
To obtain $Q_{a,m}$ in the form $[c_1, c_2, c_3, c_4, c_5; c_6]_{m}$, we need to use the three-term Weierstrass relation for theta functions \cite[Proposition 2.1]{HM14}
\begin{equation*}
P_{a+c}P_{a-c}P_{b+d}P_{b-d} = P_{a+d}P_{a-d}P_{b+c}P_{b-c} + q^{b-c} P_{a+b}P_{a-b}P_{c+d}P_{c-d}.
\end{equation*}
Let us consider ten cases of the three-term Weierstrass relation, mentioned in \cite[(4.6)-(4.10)]{ASD} and \cite[(b1)-(b5)]{E2000}:
\begin{equation} \label{identity:wr1}
P_2P_4P_5^2 - P_3^2P_4P_5 + q^2P_1^2P_2P_3=0,
\end{equation} 
\begin{equation}\label{identity:wr2}
P_1P_4P_5^2 - P_2P_3P_4^2 + qP_1P_2P_3^2=0,
\end{equation}
\begin{equation}\label{identity:wr3}
P_1P_3P_5^2 - P_2^2P_4P_5 + qP_1^2P_3P_4=0,
\end{equation}
\begin{equation}\label{identity:wr4}
P_1P_4^2P_5 - P_2P_3^2P_5 + qP_1P_2^2P_4=0,
\end{equation}
\begin{equation}\label{identity:wr5}
P_1P_3P_4^2 - P_2^2P_3P_5 + qP_1^2P_2P_5=0,
\end{equation}
and
\begin{equation} \label{identity:wr6}
P_3P_5^3 - P_5P_4^3+q^3P_2P_1^3=0,
\end{equation} 
\begin{equation}\label{identity:wr7}
P_2P_5^3 - P_3P_4^3 + q^2P_1P_2^3=0,
\end{equation}
\begin{equation}\label{identity:wr8}
P_2P_4^3 - P_5P_3^3 + q^2P_4P_1^3 =0 ,
\end{equation}
\begin{equation}\label{identity:wr9}
P_1P_5^3 - P_4P_3^3 + qP_3P_2^3 = 0,
\end{equation}
\begin{equation}\label{identity:wr10}
P_1P_3^3 - P_4P_2^3 + qP_5P_1^3 = 0.
\end{equation}
Also note that another form of $\Theta$ can be obtained by applying (\ref{identity:rearrangingtermsP}):
\begin{equation*}
\Theta(a_1,a_2,a_3,a_4,a_5) = \frac{J_{11}^2}{J_{1}^3} \Big[a_1\frac{q^{2}P_1P_2P_3}{P_5}+a_2\frac{P_2P_4P_5}
{P_1}+a_3\frac{qP_2P_3P_5}{P_4}+a_4\frac{qP_1P_3P_4}{P_2}+a_5\frac{qP_1P_4P_5}{P_3}\Big].
\end{equation*}
Let us consider the case residue 7. From (\ref{equation:thetaformQ7}) we have
\begin{multline*}
Q^{\text{th}}_{a,7}(q)=q^{-1}\frac{P_5}{11P_1}\Theta(a_1,0,a_3,a_4,a_5) + \frac{P_4}{11P_5}\Theta(b_1,0,0,0,b_5) 
=\\  \frac{J_{11}^2}{J_{1}^3}\left(a_1qP_2P_3 + a_3\frac{P_2P_3P_5^2}{P_1P_4} +a_4\frac{P_3P_4P_5}{P_2}+ a_5\frac{P_4P_5^2}{P_3} + b_1\frac{q^2P_1P_2P_3P_4}{P_5^2}+ b_5\frac{qP_1P_4^2}{P_3}\right).
\end{multline*}
Then we consider
\begin{align*}
(\ref{identity:wr4})\times \frac{P_5}{P_1P_3P_4} \ \ &: \ \ \frac{P_2P_3P_5^2}{P_1P_4} = \frac{P_4P_5^2}{P_3} + \frac{qP_2^2P_5}{P_3},\\
(\ref{identity:wr1})\times \frac{1}{P_2P_3} \ \ &: \ \ \frac{P_3P_4P_5}{P_2} = q^2P_1^2+\frac{P_4P_5^2}{P_3}.
\end{align*}
After changing the terms to the new ones using the above expressions, we obtain
\begin{equation*}
Q^{\text{th}}_{a,7}(q)= \frac{J_{11}^2}{J_{1}^3}\left(c_1\frac{P_4P_5^2}{P_3} + c_2q^2P_1^2 + c_3\frac{qP_1P_4^2}{P_3}+ c_4\frac{qP_2^2P_5}{P_3} + c_5qP_2P_3+ c_6\frac{q^2P_1P_2P_3P_4}{P_5^2}\right).
\end{equation*}
As a result, we obtain $[c_1, c_2, c_3, c_4, c_5; c_6]_{7}$. Similar calculations can be done for residues $0,4,9,10$.

Another case we are going to consider as an example is residue $8$. From (\ref{equation:thetaformQ8}) we have
\begin{multline*}
Q_{a,8}(q)=q^{-1}\frac{P_4}{11P_1}\Theta(a_1,0,a_2,a_3,a_4) + \frac{P_3}{11P_4}\Theta(0,0,b_3,b_4,0) =\\ \frac{J_{11}^2}{J_{1}^3} \left(a_1\frac{qP_2P_3P_4}{P_5} + a_2\frac{P_2P_3P_5}{P_1} + a_3\frac{P_3P_4^2}{P_2} + a_4\frac{P_4^2P_5}{P_3} + b_3\frac{qP_2P_3^2P_5}{P_4^2}+b_4\frac{qP_1P_3^2}{P_2}\right).
\end{multline*}
Then we consider
\begin{align*}
(\ref{identity:wr2})\times \frac{q}{P_4P_5} \ \ &: \ \ \frac{qP_2P_3P_4}{P_5} = qP_1P_5 + \frac{q^2P_1P_2P_3^2}{P_4P_5},\\
(\ref{identity:wr5})\times \frac{1}{P_1P_2} \ \ &: \ \ \frac{P_2P_3P_5}{P_1} = \frac{P_3P_4^2}{P_2} + qP_1P_5,\\
(\ref{identity:wr1})\times \frac{P_4}{P_2P_3P_5} \ \ &: \ \ \frac{P_4^2P_5}{P_3} = \frac{P_3P_4^2}{P_2} - \frac{q^2P_1^2P_4}{P_5},\\
(\ref{identity:wr4})\times \frac{q}{P_4^2} \ \ &: \ \ \frac{qP_2P_3^2P_5}{P_4^2} = qP_1P_5 + \frac{q^2P_1P_2^2}{P_4},\\
(\ref{identity:wr1})\times \frac{qP_1}{P_2P_4P_5} \ \ &: \ \ \frac{qP_1P_3^2}{P_2} = \frac{q^3P_1^3P_3}{P_4P_5} + qP_1P_5.
\end{align*}
After changing the terms to the new ones using the above expressions, we obtain
\begin{equation*}
Q_{a,8}(q)=
c_1qP_1P_5 + c_2\frac{q^3P_1^3P_3}{P_4P_5} + c_3 \frac{q^2P_1^2P_4}{P_5}+c_4 \frac{q^2P_1P_2^2}{P_4}+c_5 \frac{q^2P_1P_2P_3^2}{P_4P_5}+c_6 \frac{P_3P_4^2}{P_2}.
\end{equation*}
As a result, we obtain $[c_1, c_2, c_3, c_4, c_5; c_6]_{8}$. Similar calculations can be done for residues $1,2,3,5$.
\end{proof}

\section{Properties of partition function, crank and rank modulo $11$}\label{section:oldresultsproofs}
\subsection{Equalities between cranks modulo $11$} \label{subsection:crankequal}Garvan \cite[(1.51)-(1.67)]{G88} found equalities between cranks modulo $11$. We present his results in the following theorem.
\begin{theorem}[{\cite[(1.51)-(1.67)]{G88}}]\label{theorem:crankeq}  
Consider $n\geq 0$. For $M_i = M(i,11,11n)$ we have 
\begin{equation*}
M_1 = M_2 = M_3 = M_4 = M_5.
\end{equation*}
For $M_i = M(i,11,11n+1)$ we have 
\begin{equation*}
M_0 + M_1 = 2M_2 \ \ \text{and} \ \ M_2 = M_3 = M_4 = M_5.
\end{equation*}
For $M_i = M(i,11,11n+2)$ we have 
\begin{equation*}
M_0 = M_1 = M_3 = M_4 = M_5.
\end{equation*}
For $M_i = M(i,11,11n+3)$ we have 
\begin{equation*}
M_0=M_3 \ \ \text{and} \ \ M_1 = M_2 = M_4 = M_5.
\end{equation*}
For $M_i = M(i,11,11n+4)$ we have 
\begin{equation*}
M_0=M_2=M_4 \ \ \text{and} \ \ M_1=M_3=M_5.
\end{equation*}
For $M_i = M(i,11,11n+5)$ we have 
\begin{equation*}
M_0=M_1=M_3=M_5 \ \ \text{and} \ \ M_2=M_4.
\end{equation*}
For $M_i = M(i,11,11n+6)$ we have 
\begin{equation*}
M_0=M_1=M_2=M_3=M_4=M_5=\frac{p(11n+6)}{11}.
\end{equation*}
For $M_i = M(i,11,11n+7)$ we have 
\begin{equation*}
M_0=M_2=M_3=M_5 \ \ \text{and} \ \ M_1=M_4.
\end{equation*}
For $M_i = M(i,11,11n+8)$ we have 
\begin{equation*}
M_0=M_2=M_5 \ \ \text{and} \ \ M_1=M_3=M_4.
\end{equation*}
For $N_i = N(i,11,11n+9)$ and $M_i = M(i,11,11n+9)$ we have 
\begin{equation*}
M_0=M_4 \ \ \text{and} \ \ M_1=M_2=M_3=M_5.
\end{equation*}
For $N_i = N(i,11,11n+10)$ and $M_i = M(i,11,11n+10)$ we have 
\begin{equation*}
M_0=M_1=M_2=M_4=M_5.
\end{equation*}
\end{theorem}
\begin{proof}[Proof of Theorem \ref{theorem:crankeq}] By comparing respective coefficients in the $11$-dissections in Theorem \ref{theorem:crankdissection} we can directly deduce the following equalities. Also note that there is no element corresponding to $q^6$ in $11$-dissection of the deviation of the crank, so we have the corresponding equality for residue $6$.\qedhere
\end{proof}

\subsection{Partition function congruences modulo $11$} \label{subsection:partfunccong} Atkin and Swinnerton-Dyer in \cite[Theorem 3]{ASD} realised the following congruences.
\begin{theorem}[{\cite[Theorem 3]{ASD}}] \label{theorem:pcong}We have
\begin{align*} 
\sum_{n=0}^{\infty} p(11n)q^n &\equiv \frac{J_{11}^2}{P_1} \Mod {11},\\
\sum_{n=0}^{\infty}p(11n+1)q^n &\equiv \frac{J_{11}^2 P_5}{P_2P_3} \Mod {11},\\
\sum_{n=0}^{\infty}p(11n+2)q^n &\equiv 2\frac{J_{11}^2 P_3}{P_1P_4} \Mod {11},\\
\sum_{n=0}^{\infty}p(11n+3)q^n &\equiv 3\frac{J_{11}^2 P_2}{P_1P_3}   \Mod {11},\\
\sum_{n=0}^{\infty}p(11n+4)q^n &\equiv 5\frac{J_{11}^2}{P_2}    \Mod {11}, \\
\sum_{n=0}^{\infty}p(11n+5)q^n &\equiv 7\frac{J_{11}^2 P_4}{P_2P_5}    \Mod {11}\\
\sum_{n=0}^{\infty}p(11n+7)q^n &\equiv 4\frac{J_{11}^2}{P_3} \Mod {11}, \\
\sum_{n=0}^{\infty}p(11n+8)q^n &\equiv 6\frac{J_{11}^2 qP_1}{P_4P_5}  \Mod {11}, \\
\sum_{n=0}^{\infty}p(11n+9) q^n &\equiv 8\frac{J_{11}^2}{P_4}    \Mod {11}, \\
\sum_{n=0}^{\infty}p(11n+10)q^n &\equiv 9\frac{J_{11}^2}{P_5}    \Mod {11}.
\end{align*}
\end{theorem}

\begin{proof}[Proof of Theorem \ref{theorem:pcong}] Using Theorem \ref{theorem:crankdissection} it is obvious how to obtain such congruences. For example, we can take
\begin{equation*}
\sum_{n=0}^{\infty}\left(M(0,11,11n) - \frac{p(11n)}{11}\right)q^n = \frac{10}{11}\frac{J_{11}^2}{P_1},
\end{equation*}
multiply it by $-11$ and take modulo $11$.
\end{proof}
\begin{remark}
Elements of $11$-dissection for $p(n)$ can be obtained using \cite[Lemma 4]{K} and they are found explicitly in terms of theta quotients by Bilgici and Ekin \cite{GE}. Another forms for the $11$-dissection element $\sum_{n\geq 0}p(11n+6)q^n$, which represent explicitly its modular properties, can be found in \cite{PR}.
\end{remark}

\subsection{Linear rank congruences modulo $11$}\label{subsection:linearrangcong} Atkin and Hussain \cite{AH} studied the rank modulo $11$ and for each residue they found linear congruences between ranks \cite[(9.16)]{AH}. We present their results in the following theorem. 
\begin{theorem} [{\cite[(9.16)]{AH}}] \label{theorem:linrankcong} Consider $n\geq 0$. For $N_i = N(i,11,11n)$ we have
\begin{equation*}
N_2-5N_3-2N_4+6N_5 \equiv 0 \Mod {11}.
\end{equation*}
For $N_i = N(i,11,11n+1)$ we have
\begin{equation*}
N_1-6N_2+4N_3+3N_4-2N_5 \equiv 0 \Mod {11}.
\end{equation*}
For $N_i = N(i,11,11n+2)$ we have
\begin{equation*}
 N_0 + 4N_2 - 6N_4 + N_5 \equiv 0 \Mod {11}.   
\end{equation*}
For $N_i = N(i,11,11n+3)$ we have
\begin{equation*}
N_0 + 3N_1 - N_2 + 2N_3 - N_4 -4N_5 \equiv 0 \Mod {11}.    
\end{equation*}
For $N_i = N(i,11,11n+4)$ we have
\begin{equation*}
N_0 + 3N_1 -2N_2 - 4N_3 + N_4 + N_5 \equiv 0 \Mod {11}.    
\end{equation*}
For $N_i = N(i,11,11n+5)$ we have
\begin{equation*}
 N_0 - 5N_1 - N_2 + N_3 + 5N_4 - N_5 \equiv 0 \Mod {11}.   
\end{equation*}
For $N_i = N(i,11,11n+6)$ we have
\begin{equation*}
N_1 -5N_2 - N_3 +N_4 +4N_5 \equiv 0 \Mod {11}.    
\end{equation*}
For $N_i = N(i,11,11n+7)$ we have
\begin{equation*}
    N_0 - 2N_1 -2N_2 + 5N_3 +2N_4 -4N_5 \equiv 0 \Mod {11}.
\end{equation*}
For $N_i = N(i,11,11n+8)$ we have
\begin{equation*}
 N_0 + 5N_1 + 2N_2 + N_3 -3N_4 - 6N_5 \equiv 0 \Mod {11}.   
\end{equation*}
For $N_i = N(i,11,11n+9)$ we have
\begin{equation*}
 N_0 -4N_1 +3N_2 -N_3 -N_4 +2N_5 \equiv 0 \Mod {11}.   
\end{equation*}
For $N_i = N(i,11,11n+10)$ we have
\begin{equation*}
N_0 -6N_1 + N_4 + 4N_5 \equiv 0 \Mod {11}.
\end{equation*}
\end{theorem}

\begin{proof}[Proof of Theorem \ref{theorem:linrankcong}] It is straightforward now to obtain such congruences by using the calculations of the $11$-dissection of the deviation of the rank found in Section \ref{section:devrank}. For example, let us consider the case of residue $0$. We find that 
\begin{equation*}
Q_{2,0}(q)-5Q_{3,0}(q)-2Q_{4,0}(q)+6Q_{5,0}(q) = 11[0,-1,0,1,-1;0]_{0}.
\end{equation*}
As the sum of coefficients in the above sum is zero, we see that the coefficient of $q^n$ on the left-hand side of the previous sum is
\begin{equation*}
N(2,11,11n)-5N(3,11,11n)-2N(4,11,11n)+6N(5,11,11n).
\end{equation*}
Taking expression modulo $11$ we obtain the desired result. 
\end{proof}

\section{Rank-crank inequalities, positivity techniques, and positivity conjectures}
Firstly we give new proofs of work of Ekin \cite{E} in order to motivate our use of positivity techniques and to give context to our positivity conjectures.

\subsection{Inequalities between cranks modulo $11$}\label{subsection:crankinequal} Ekin  \cite[(21)-(28)]{E} considered inequalities between cranks modulo $11$. We state the results in the following theorem and conjecture.
\begin{theorem}[{\cite[(21)-(28)]{E}}] \label{theorem:crankineq} For $n\geq 0$ we have
\begin{gather*}
M(0,11,11n) \geq \frac{p(11n)}{11} \geq M(1,11,11n), \\
M(1,11,11n+1) \geq \frac{p(11n+1)}{11} \geq M(2,11,11n+1) \geq M(0,11,11n+1),\\
M(2,11,11n+2) \geq \frac{p(11n+2)}{11} \geq M(0,11,11n+2), \\
M(0,11,11n+3) \geq \frac{p(11n+3)}{11} \geq M(1,11,11n+3),  \\
M(0,11,11n+4) \geq \frac{p(11n+4)}{11} \geq M(1,11,11n+4), \\
M(0,11,11n+5) \geq \frac{p(11n+5)}{11} \geq M(2,11,11n+5),  \\
M(1,11,11n+7) \geq \frac{p(11n+7)}{11} \geq M(0,11,11n+7), \\
M(1,11,11n+9) \geq \frac{p(11n+9)}{11} \geq M(0,11,11n+9),  \\
M(0,11,11n+10) \geq \frac{p(11n+10)}{11} \geq M(3,11,11n+10).
\end{gather*}
\end{theorem}
\begin{conjecture} \label{conjecture:crankineq8} For $n \geq 0$, $n \neq 2$ we have
\begin{equation*}
M(1,11,11n+8) \geq \frac{p(11n+8)}{11} \geq M(0,11,11n+8).
\end{equation*}
\end{conjecture}
\begin{lemma} \label{lemma:blockscrankpos} We have 
\begin{equation*}
\frac{J_{11}^2}{P_a} \geq 0
\end{equation*}
for $a \in \{1,2,3,4,5 \}$ and 
\begin{equation*}
\frac{J_{11}^2P_{2a}}{P_aP_{3a}} \geq 0
\end{equation*}
for $a \in \{1,2,3,4\}$.
\end{lemma}
\begin{proof}[Proof of Lemma \ref{lemma:blockscrankpos}]
By the Jacobi's triple product identity we see that for $a \in \{1,2,3,4,5 \}$
\begin{gather*}
\frac{J_{11}^2}{P_a} = \frac{(q^{11};q^{11})_{\infty}}{(q^{a};q^{11})_{\infty}(q^{11-a};q^{11})_{\infty}} = \frac{(-q^{a};q^{11})_{\infty}(-q^{11-a};q^{11})_{\infty}(q^{11};q^{11})_{\infty}}{(q^{2a};q^{22})_{\infty}(q^{22-2a};q^{22})_{\infty}} \\ = \frac{1}{(q^{2a};q^{22})_{\infty}(q^{22-2a};q^{22})_{\infty}} \cdot \sum_{n=0}^{\infty} q^{11\binom{n}{2}+an}.
\end{gather*}
By the quintuple product identity \cite[(42)]{E} we see that for $a \in \{1,2,3\}$
\begin{equation} \label{identity:quinprod}
\frac{J_{11}^2P_{2a}}{P_aP_{3a}} = \frac{J_{33}^3}{J_{3a,33}J_{11-3a,33}} +  \frac{q^aJ_{33}^3}{J_{3a,33}J_{22-3a,33}}.
\end{equation}
In the same way as in the previous case but with change $q \rightarrow q^3$ we see that both of the terms in the right-hand side of the expression above have non-negative Fourier coefficients. The case of $a=4$ is proved in \cite[Corollary 4.8]{BG} using a different representation of $\frac{J_{11}^2P_3}{P_1P_4}$.
\end{proof}
\begin{proof}[Proof of Theorem \ref{theorem:crankineq}] To obtain Theorem \ref{theorem:crankineq} we use Lemma \ref{lemma:blockscrankpos} and refer to the positivity or negativity of the coefficients in Theorem \ref{theorem:crankdissection}. 
\end{proof}
\begin{remark} Conjecture \ref{conjecture:crankineq8} is partially solved. Using analytic methods it is known from \cite{R} that there is explicit $N$, such that for $n > N$ we have
\begin{equation*}
M(1,11,11n+8) \geq M(0,11,11n+8).
\end{equation*}

\end{remark}

\subsection{Positivity of theta quotients} \label{subsection:positivitytechniques} Denote 
\begin{equation*}
\Tilde{P}_a := (q^a;q^{11})_{\infty}(q^{11-a};q^{11})_{\infty} = \frac{P_a}{J_{11}}.  
\end{equation*}
\begin{definition} We define the notation $F(q) \geq G(q)$
if the Fourier coefficients of $F(q)-G(q)$ are non-negative.
\end{definition}
\begin{lemma} \label{lemma:posthetablock} We have
\begin{equation} \label{equation:thetaproduct}
J_{11}^{a} q^b \Tilde{P}_1^{\alpha_1}\Tilde{P}_2^{\alpha_2}\Tilde{P}_3^{\alpha_3}\Tilde{P}_4^{\alpha_4}\Tilde{P}_5^{\alpha_5} \geq 0   
\end{equation}
where $\alpha_i \leq 0$ for $1 \leq i \leq M$, $\alpha_1 + \alpha_2 + \alpha_3+\alpha_4+\alpha_5 = r$ and $0 \leq a \leq r$, $b \in \Z$.
\end{lemma}
\begin{proof}[Proof of Lemma \ref{lemma:posthetablock}] By Lemma \ref{lemma:blockscrankpos} we see that for $a \in \{1,2,3,4,5\}$
\begin{equation*}
\frac{J_{11}}{\Tilde{P}_a} \geq 0
\end{equation*}
and obviously 
\begin{equation*}
\frac{1}{\Tilde{P}_a} \geq 0.
\end{equation*}
Lemma \ref{lemma:posthetablock} follows directly from inequalities above. 
\end{proof}
\begin{remark} By Lemma \ref{lemma:posthetablock} all terms in sums of theta quotients in (\ref{expression:vterm}), Definition \ref{definition:mainbracket} and Definition \ref{definition:residuebrackets} have non-negative Fourier coefficients. For example, by applying (\ref{identity:rearrangingtermsP}) we have
\begin{equation*}
\Theta(1,0,0,0,0) = \frac{J_{11}^6}{J_1^2} \cdot \frac{q^2}{P_4P_5^2} = J_{11}^{14} \cdot  \frac{q^2}{P_1^2P_2^2P_3^2P_4^3P_5^4} = J_{11} \cdot  \frac{q^2}{\Tilde{P}_1^2\Tilde{P}_2^2\Tilde{P}_3^2\Tilde{P}_4^3\Tilde{P}_5^4} \geq 0.
\end{equation*}
    
\end{remark}
\begin{remark} Let $\alpha_1 < 0$ in (\ref{equation:thetaproduct}). Then we can establish not just non-negativity of Fourier coefficients of (\ref{equation:thetaproduct}) but also that the coefficients of (\ref{equation:thetaproduct}) are monotonically non-decreasing by multiplying the expression by $(1-q)$. In this sense we have that if
\begin{equation*}
P(q) := \sum_{i\in \Z} c_i q^i
\end{equation*}
and
\begin{equation*}
(1-q)P(q) = \sum_{i\in \Z} (c_i-c_{i-1})q^i \geq 0,
\end{equation*}
then we have $c_i \geq c_{i-1}$ for any $i\in \Z$.
\end{remark}

Let us compare Fourier coefficients of theta quotients in  $[c_1, c_2, c_3, c_4, c_5]$.
We have the following inequalities
\begin{proposition} \label{proposition:comp} We have
\begin{equation} \label{equation:comp}
[0, 1, 0, 0, 0] \leq [0, 0, 1, 0, 0]  \leq  [0, 0, 0, 1, 0]  \leq [0, 0, 0, 0, 1].  
\end{equation}
\end{proposition}
\begin{proof}[Proof of Proposition \ref{proposition:comp}] Inequalities can be proved using the three-term Weierstrass relations (\ref{identity:wr1})-(\ref{identity:wr5}) in the following form
\begin{align*} 
[0, -1, 1, 0, 0] &= \frac{J_{11}^2}{J_1^3}q P_3^2P_2 - \frac{J_{11}^2}{J_1^3}q P_2^2P_5 = \frac{J_{11}^2}{J_1^3}\frac{q^3P_1^2P_2^2P_3}{P_4P_5} \geq 0,\\
[0, 0, -1, 1, 0] &= \frac{J_{11}^2}{J_1^3}q P_2^2P_5 - \frac{J_{11}^2}{J_1^3}q P_4^2 P_1 = \frac{J_{11}^2}{J_1^3}\frac{q^2P_1^2P_2P_5}{P_3} \geq 0,\\
[0, 0, 0, -1, 1] &= \frac{J_{11}^2}{J_1^3}q P_4^2 P_1 -\frac{J_{11}^2}{J_1^3}q^2 P_1^2 P_3 = \frac{J_{11}^2}{J_1^3}\frac{qP_1^2P_4P_5^2}{P_2P_3} \geq 0.
\end{align*}
Here we can see the positivity of the Fourier coefficients of the right-hand
side of equations above by Lemma \ref{lemma:posthetablock} as the theta quotients in the right-hand side can be transformed into (\ref{equation:thetaproduct}) by applying (\ref{identity:rearrangingtermsP}).
\end{proof} 
\begin{remark} \label{remark:strengthencomp} We can strengthen Proposition \ref{proposition:comp} to  
\begin{equation} \label{equation:strengthencomp}
J_1[0, 1, 0, 0, 0] \leq J_1[0, 0, 1, 0, 0]  \leq  J_1[0, 0, 0, 1, 0]  \leq J_1[0, 0, 0, 0, 1].  
\end{equation}
\end{remark}
We are able to compare Fourier coefficients of theta quotients in $[c_1, c_2, c_3, c_4, c_5; c_6]_{i}$ which was defined in Definition \ref{definition:residuebrackets}.
\begin{proposition}
\label{proposition:maincomparison} For every residue $i \neq 6$ we have
\begin{equation*}
 [0,1,0,0,0;0]_{i} \leq [0,0,1,0,0;0]_{i} \leq [0,0,0,1,0;0]_{i} \leq [0,0,0,0,1;0]_{i}.\\
\end{equation*}
\end{proposition} 
\begin{proof}[Proof of Proposition \ref{proposition:maincomparison}] For $i \neq 8$ this result can directly obtained from Proposition \ref{proposition:comp}. For example for residue $0$ we need to multiply (\ref{equation:comp}) by $\frac{1}{P_1} \geq 0$.  For $i = 8$ we use Remark \ref{remark:strengthencomp} and multiply (\ref{equation:strengthencomp}) by $\frac{qP_1}{J_1P_4P_5} \geq 0$.
\end{proof}
For residue $6$ we have the following comparison formulas.
\begin{proposition}
\label{proposition:maincomparisonres6} We have
\begin{equation*}
 \Theta(0,0,0,0,1) \leq \Theta(0,0,0,1,0) \leq \Theta(0,0,1,0,0) \leq \Theta(0,1,0,0,0). 
\end{equation*}
\end{proposition} 
\begin{proof}[Proof of Proposition \ref{proposition:maincomparisonres6}]
Inequalities can be proved using the three-term Weierstrass relations (\ref{identity:wr1})-(\ref{identity:wr10}) in the following form.
\begin{align*} 
 \Theta(0,0,0,1,-1) &= \frac{J_{11}^2}{J_1^3}\frac{qP_1P_3P_4}{P_2} - \frac{J_{11}^2}{J_1^3}\frac{qP_1P_4P_5}{P_3} = \frac{J_{11}^2}{J_1^3}\frac{q^3P_1^3}{P_5} \geq 0,\\
\Theta(0,0,1,-1,0) &= \frac{J_{11}^2}{J_1^3} \frac{qP_2P_3P_5}{P_4}-\frac{J_{11}^2}{J_1^3}\frac{qP_1P_3P_4}{P_2} = \frac{J_{11}^2}{J_1^3} \frac{q^2P_1^2P_5}{P_4} \geq 0,\\
\Theta(0,1,-1,0,0) &=\frac{J_{11}^2}{J_1^3}  \frac{P_2P_4P_5}{P_1}-\frac{J_{11}^2}{J_1^3} \frac{qP_2P_3P_5}{P_4} = \frac{J_{11}^2}{J_1^3} \frac{P_5^3}{P_3}\geq 0.\qedhere
\end{align*}
\end{proof}
We also provide some additional comparison formulas among theta quotients. 
\begin{lemma} \label{lemma:addcomparison} We have
\begin{align*}
[0,0,0,0,-1;1]_{1} &\geq 0,\\
[0,1,0,0,0;-1]_{2} &\geq 0,\\
[0,0,0,1,0;-1]_{3} &\geq 0,\\
[0,0,-1,0,0;1]_{3} &\geq0 ,\\
[0,0,1,0,0;-1]_{5} &\geq 0,\\
[0,0,0,-1,0;1]_{8} &\geq 0.
\end{align*}
\end{lemma}
\begin{proof}[Proof of Lemma \ref{lemma:addcomparison}] We need to use the three-term Weierstrass relations (\ref{identity:wr1})-(\ref{identity:wr10})
\begin{align*}
[0,0,0,0,-1;1]_{1} &= \frac{J_{11}^2}{J_1^3} \left(\frac{qP_2^2P_4}{P_1} - \frac{P_5}{P_2P_3}\cdot qP_3^2P_2 \right)= \frac{J_{11}^2}{J_1^3}\frac{q^2P_1P_3P_4}{P_5} \geq 0,\\
[0,1,0,0,0;-1]_{2} &= \frac{J_{11}^2}{J_1^3} \left(\frac{P_3}{P_1P_4} \cdot q^2P_1^2P_3 - \frac{q^3P_1^2P_2}{P_5} \right)=\frac{J_{11}^2}{J_1^3} \frac{q^2P_1^2P_4}{P_2} \geq 0,\\
[0,0,0,1,0;-1]_{3} &= \frac{J_{11}^2}{J_1^3} \left(\frac{P_2}{P_1P_3} \cdot qP_2^2P_5 -   \frac{qP_3^2P_5}{P_4} \right)=\frac{J_{11}^2}{J_1^3}  \frac{q^2P_1^2P_5^2}{P_3P_4} \geq 0,\\
[0,0,-1,0,0;1]_{3} &= \frac{J_{11}^2}{J_1^3} \left(\frac{qP_3^2P_5}{P_4} -  \frac{P_2}{P_1P_3} \cdot qP_4^2P_1 \right)  =\frac{J_{11}^2}{J_1^3} \frac{q^3P_1^3}{P_3} \geq 0,\\
[0,0,1,0,0;-1]_{5} &= \frac{J_{11}^2}{J_1^3} \left(\frac{P_4}{P_2P_5} \cdot qP_4^2P_1 - \frac{qP_1P_5^2}{P_3} \right)=\frac{J_{11}^2}{J_1^3}\frac{q^3P_1^2P_2^2}{P_3P_5} \geq 0,\\
[0,0,0,-1,0;1]_{8} &= \frac{J_{11}^2}{J_1^3} \left(\frac{P_3P_4^2}{P_2} - \frac{qP_1}{P_4P_5} \cdot qP_2^2P_5 \right)=\frac{J_{11}^2}{J_1^3} \frac{P_5^3}{P_4} \geq 0.\qedhere
\end{align*}
\end{proof}

\subsection{Proofs of rank-crank inequalities} \label{subsection:rankcrankinequal}
Define $Q^{C}_{a,m}(q)$ to be the elements of the $11$-dissection of the deviation of the crank:
\begin{equation*}
D_{C}(a,11) =: \sum_{m=0}^{10}  Q^{C}_{a,m}(q^{11})q^m.
\end{equation*} 
So the reformulation of this definition is
\begin{equation} \label{definition:crankdisselem}
Q^{C}_{a,m}(q) := \sum^{\infty}_{n = 0} \left(M(a,11,11n+m) - \frac{p(11n+m)}{11}\right)q^n.
\end{equation}
\begin{proof}[Proof of Theorem \ref{theorem:rankcrankbasicineq}] As an example we consider inequalities corresponding to residue $0$. For $n \geq 0$, $N_i=N(0,11,11n)$ and $M_i=M(i,11,11n)$ we want to establish
\begin{align*}
N_0 + 2N_1 + M_1 &\geq 2N_2+N_4+M_0,\\
N_0 +2N_1+3N_2 + M_1 &\geq 3N_3 + 3N_5+M_0,\\
2N_2+N_3 + N_5&\geq 4N_4,\\
N_2+5N_3+3N_4+M_0 &\geq N_0+2N_1+6N_5+M_1.
\end{align*}
We sum up
\begin{align*} 
Q_{0,0}(q)+2Q_{1,0}(q)+Q^{C}_{1,0}(q)-2Q_{2,0}(q)-Q_{4,0}(q)-Q^{C}_{0,0}(q) &= [0,11,0,0,0;0]_{0} \geq 0,\\
Q_{0,0}(q)+2Q_{1,0}(q)+3Q_{2,0}(q)+Q^{C}_{1,0}(q)-3Q_{3,0}(q)-3Q_{5,0}(q)-Q^{C}_{0,0}(q) &= [0,-11,11,0,0;0]_{0} \geq 0,\\
2Q_{2,0}(q)+Q_{3,0}(q)+Q_{5,0}(q)-4Q_{4,0}(q) &= [0,0,-11,11,0;0]_{0} \geq 0
\end{align*}
and
\begin{multline*} 
Q_{2,0}(q)+5Q_{3,0}(q)+3Q_{4,0}(q)+Q^{C}_{0,0}(q)-Q_{0,0}(q)-2Q_{1,0}(q)-6Q_{5,0}(q)-Q^{C}_{1,0}(q) \\= [0,0,0,-11,11;0]_{0} \geq 0.
\end{multline*}
and apply Proposition \ref{proposition:maincomparison}. For other residues Lemma \ref{lemma:addcomparison} is also used.
\end{proof}

\begin{proof}[Proof of Corollary \ref{corollary:inequallevel2level4}] 
As an example let us consider the inequality
\begin{equation*} 
N(2,11,11n+1) \geq M(2,11,11n+1),
\end{equation*}
where $n \geq 0$. To prove it we sum up
\begin{equation*} 
Q_{2,1}(q) - Q^{C}_{2,1}(q) = [0,1,0,-1,0;1]_{1} \geq 0
\end{equation*}
and apply Proposition \ref{proposition:maincomparison} and Lemma \ref{lemma:addcomparison}.
\end{proof}

\begin{proof}[Proof of Corollary \ref{corollary:inequallevel6}] 
As an example let us consider the inequality
\begin{equation*} 
N(1,11,11n+1)+2 N(4,11,11n+1) \geq 2 N(5,11,11n+1)+M(0,11,11n+1),
\end{equation*}
where $n \geq 0$. To prove it we sum up
\begin{multline*} 
Q_{1,1}(q) + 2Q_{4,1}(q) - Q_{5,1}(q) - Q^{C}_{0,1}(q) = [1, -2, -4, 3, -3; 2]_{1}= \\= [1, -1, -1, -1, -1, 0]_{1} + [0, -1, -3, 4, -2; 2]_{1} \geq 0
\end{multline*}
and apply Proposition \ref{proposition:maincomparison}, Lemma \ref{lemma:addcomparison} and note that 
\begin{equation*} 
[1, -1, -1, -1, -1; 0]_{1}=\frac{J_{11}^2P_5}{P_2P_3} \geq 0
\end{equation*}
by Lemma \ref{lemma:blockscrankpos}.
\end{proof}

\begin{proof}[Proof of Corollary \ref{corollary:ineqrank6level2}] We consider
\begin{align*}
Q_{0,6}(q) &= \Theta(0,0,1,1,-1) \geq 0,\\
-Q_{5,6}(q) &= \Theta(1,0,1,0,1) \geq 0,
\end{align*}
where $n \geq 0$ and apply Proposition \ref{proposition:maincomparisonres6}.
\end{proof}

\begin{proof}[Proof of Corollary \ref{corollary:ineqrank6level4level6}]
As an example let us consider the inequality
\begin{equation*} 
N(0,11,11n+6)+N(3,11,11n+6)  \geq N(1,11,11n+6)+N(4,11,11n+6),
\end{equation*}
where $n \geq 0$. To prove it we sum up
\begin{equation*} 
Q_{0,6}(q) + Q_{3,6}(q) - Q_{1,6}(q) - Q_{4,6}(q)= \Theta(2, 0, 3, 3, -6) \geq 0
\end{equation*}
and apply Proposition \ref{proposition:maincomparisonres6}.
\end{proof}
\begin{remark} \label{remark: constructineq} If it is known that
\begin{equation*} 
[c_1, c_2, c_3, c_4, c_5; c_6]_{m} \geq 0
\end{equation*} 
with some $c_k \in \Q$, $1\leq k\leq 6$ and residue $m \neq 6$, then you can construct some rank-crank inequality with $N_i = N(i,11,11n+m)$ and $M_i = M(i,11,11n+m)$. We need to look at $\{ Q_{a,m} \ | \ 0\leq a\leq 4 \}$ and $Q^{C}_{0,m}$ as linear independent vectors in $\R^6$ as described in Section \ref{section:devrank}, that is, we need to consider equation
\begin{equation*} 
\sum_{j=0}^{4} a_jQ_{j,m} + b_0Q^{C}_{0,m} = [c_1, c_2, c_3, c_4, c_5; c_6]_{m}
\end{equation*} 
as a linear system over unknowns $\{a_0,a_1,a_2,a_3,a_4, b_0\}$ with a unique rational solution. The same is true in case of residue $6$. If it is known that
\begin{equation*} 
\Theta(c_1, c_2, c_3, c_4, c_5) \geq 0
\end{equation*} 
with some $c_k \in \Q$, $1\leq k\leq 5$, then you can construct some rank inequality with $N_i = N(i,11,11n+6)$ and $M_i = M(i,11,11n+6)$. We need to look at $\{ \Theta_{a,6} \ | \ 0\leq a\leq 4 \}$ as linear independent vectors in $\R^5$ as described in Theorem \ref{theorem:rankdissection}, that is, we need to consider equation
\begin{equation*} 
\sum_{j=0}^{4} a_j\Theta_{j,6} = \Theta(c_1, c_2, c_3, c_4, c_5)
\end{equation*} 
as a linear system over unknowns $\{a_0,a_1,a_2,a_3,a_4\}$ with a unique rational solution. Then we might need to remove a partition function term by applying
\begin{equation*} 
p(n) = N(0,11,n) + 2N(1,11,n) + \cdots + 2N(5,11,n).
\end{equation*} 
\end{remark}

\subsection{Conjectural rank-crank inequalities} \label{subsection:conjinequal}
In the sense of rank-crank inequalities in Corollary \ref{corollary:inequallevel2level4} and Corollary \ref{corollary:ineqrank6level2} we can conjecture the following stronger two-term rank-crank inequalities.
\begin{conjecture} \label{conjecture:rankcrankconj} For $N_i = N(i,11,11n)$ and $M_i = M(i,11,11n)$ we have 
\begin{equation*}
N_0 \geq_{3} N_1 \geq N_2 \geq_{1} M_0 \geq \frac{p(11n)}{11} \geq M_1 \geq N_3 \geq_{2} N_4 \geq N_5.
\end{equation*}
For $N_i = N(i,11,11n+1)$ and $M_i = M(i,11,11n+1)$ we have 
\begin{equation*}
N_0 \geq N_1 \geq N_2 \geq_{1} M_1 \geq \frac{p(11n+1)}{11} \geq M_2 \geq M_0 \geq_{1} N_3 \geq N_4 \geq N_5.
\end{equation*}
For $N_i = N(i,11,11n+2)$ and $M_i = M(i,11,11n+2)$ we have 
\begin{equation*}
N_0 \geq_{3} N_1 \geq N_2 \geq_{1} M_2 \geq \frac{p(11n+2)}{11} \geq M_0 \geq N_3 \geq N_4 \geq N_5.
\end{equation*}
For $N_i = N(i,11,11n+3)$ and $M_i = M(i,11,11n+3)$ we have 
\begin{equation*}
N_0 \geq_{2} N_1 \geq_{1} N_2 \geq M_0 \geq \frac{p(11n+3)}{11} \geq M_1 \geq N_3 \geq N_4 \geq N_5.
\end{equation*}
For $N_i = N(i,11,11n+4)$ and $M_i = M(i,11,11n+4)$ we have 
\begin{equation*}
N_0 \geq_{3} N_1 \geq N_2 \geq_{1} M_0 \geq \frac{p(11n+4)}{11} \geq M_1 \geq_{1} N_3 \geq N_4 \geq N_5.
\end{equation*}
For $N_i = N(i,11,11n+5)$ and $M_i = M(i,11,11n+5)$ we have 
\begin{equation*}
N_0 \geq N_1 \geq N_2 \geq M_0 \geq \frac{p(11n+5)}{11} \geq M_2 \geq N_3 \geq_{1} N_4 \geq N_5.
\end{equation*}
For $N_i = N(i,11,11n+6)$ and $M_i = M(i,11,11n+6)$ we have 
\begin{equation*}
N_0 \geq_{1} N_1 \geq N_2 \geq  \frac{p(11n+6)}{11} \geq N_3 \geq N_4 \geq_{1} N_5.
\end{equation*}
For $N_i = N(i,11,11n+7)$ and $M_i = M(i,11,11n+7)$ we have 
\begin{equation*}
N_0 \geq N_1 \geq_{1} N_2 \geq M_1 \geq \frac{p(11n+7)}{11} \geq M_0 \geq N_3 \geq N_4 \geq N_5.
\end{equation*}
For $N_i = N(i,11,11n+8)$ and $M_i = M(i,11,11n+8)$ we have 
\begin{equation*}
N_0 \geq_{3} N_1 \geq N_2 \geq M_1 \geq_{3} \frac{p(11n+8)}{11} \geq_{3} M_0 \geq N_3 \geq N_4 \geq N_5.
\end{equation*}
For $N_i = N(i,11,11n+9)$ and $M_i = M(i,11,11n+9)$ we have 
\begin{equation*}
N_0 \geq_{2} N_1 \geq N_2 \geq M_1 \geq \frac{p(11n+9)}{11} \geq M_0 \geq_{1} N_3 \geq N_4 \geq N_5.
\end{equation*}
For $N_i = N(i,11,11n+10)$ and $M_i = M(i,11,11n+10)$ we have 
\begin{equation*}
N_0 \geq_{3} N_1 \geq N_2 \geq M_0 \geq \frac{p(11n+10)}{11} \geq M_3 \geq_{1} N_3 \geq N_4 \geq N_5.
\end{equation*}
Notation $A_n \geq B_n$ means that $A_n \geq B_n$ for all $n \geq 0$ and $A_n \geq_m B_n$ means that $A_n \geq B_n$ for all $n \geq m$.
\end{conjecture}
\begin{remark} Some of the inequalities of Conjecture \ref{conjecture:rankcrankconj} are proved in Corollary \ref{corollary:inequallevel2level4}. 
\end{remark}
\begin{remark} Inequalities between ranks where considered by  Bringmann and Kane \cite{BK} by using analytic methods. For $0 \leq a < b \leq 5$ and for $n > N_{a,b}$, where $N_{a,b}$ is an explicit constant, we have the inequality
\begin{equation*} 
N(a,11,n) > N(b,11,n).
\end{equation*}
\end{remark} 
As a generalization of rank-crank inequalities we can state the following conjectures in terms of Definition \ref{definition:residuebrackets} and Definition \ref{definition:v11G11Theta}.
\begin{definition} We define notation $F(q) \geq_{m} G(q)$
if for $F(q) = \sum_{n \geq 0} a(n)q^n$ and for $G(q) = \sum_{n \geq 0} b(n)q^n$ we have $a(n)\geq b(n)$ for $n \geq m$.
\end{definition}
\begin{conjecture} \label{conjecture:positivitycor} For any $c_k \in \Q$, $1 \leq k \leq 6$ and residue $i \neq 6$ modulo $11$ there is $N \in \N_{0}$ such that
\begin{equation*}
    [c_1,c_2,c_3,c_4,c_5;c_6]_{i} \geq_N 0  \ \ \text{or} \ \ [c_1,c_2,c_3,c_4,c_5;c_6]_{i} \leq_N 0,
\end{equation*}
For any $a_k \in \Q$, $1 \leq k \leq 5$ there is $N \in \N_{0}$ such that
\begin{equation*}
    \Theta(a_1,a_2,a_3,a_4,a_5) \geq_N 0  \ \ \text{or} \ \ \Theta(a_1,a_2,a_3,a_4,a_5) \leq_N 0.
\end{equation*}
\end{conjecture}
We have the following corollary of Conjecture \ref{conjecture:positivitycor}.
\begin{conjecture} \label{conjecture:arbinequal} For any $a_k, b_k \in \Z$, $0 \leq k \leq 5$, such that 
\begin{equation*}
\sum_{k=0}^{5}(a_k+b_k)=0,
\end{equation*}
and residue $m$ modulo $11$ there is $N \in \N_{0}$ such that 
\begin{gather*}
\sum_{k=0}^{5}\Big [a_kN(k,11,11n+m) + b_kM(k,11,11n+m)\Big] \geq_N 0 \ \ \text{or}\\
\sum_{k=0}^{5}\Big [a_kN(k,11,11n+m) + b_kM(k,11,11n+m)\Big] \leq_N 0.
\end{gather*}
Notation $A_n \geq_m B_n$ means that $A_n \geq B_n$ for all $n \geq m$.
\end{conjecture}

\section{Proofs of new congruences} \label{section:momentscalc}
\begin{proof}[Proof of Theorem \ref{theorem:crankmomentcong}] By properties of the crank we are able to deduce
\begin{equation*} 
M_{k}(n) \equiv \sum_{m = 1}^{10} m^k M(m,11,n) \Mod {11}.
\end{equation*}
Then consider
\begin{align*}
M_{2}(n) \equiv 2M(1,11,n)-3M(2,11,n)-4M(3,11,n)-M(4,11,n)+6M(5,11,n) \Mod {11},\\
M_{4}(n) \equiv 2M(1,11,n)-M(2,11,n)-3M(3,11,n)+6M(4,11,n)-4M(5,11,n) \Mod {11},\\
M_{6}(n) \equiv 2M(1,11,n)-4M(2,11,n)+6M(3,11,n)-3M(4,11,n)-M(5,11,n) \Mod {11},\\
M_{8}(n) \equiv 2M(1,11,n)+6M(2,11,n)-M(3,11,n)-4M(4,11,n)-3M(5,11,n) \Mod {11}.
\end{align*}
Taking $n=11l+m$ in congruences above and using Theorem \ref{theorem:crankdissection} we obtain Theorem \ref{theorem:crankmomentcong}. For example using notation (\ref{definition:disselemcrankmoment}) with $m=1$ we have
\begin{multline*}
T^{C}_{2,1}(q) = \sum_{n=0}^{\infty}M_{2}(11n+1)q^n \equiv  2Q^{C}_{1,1}(q)-3Q^{C}_{2,1}(q)-4Q^{C}_{3,1}(q)-Q^{C}_{4,1}(q)+6Q^{C}_{5,1}(q)  \\ \equiv 2 \frac{J_{11}^2P_5}{P_2P_3} \Mod {11}. \qedhere
\end{multline*}
\end{proof}

\begin{proof}[Proof of Theorem \ref{theorem:rankmomentcong}]
By properties of the rank we are able to deduce
\begin{equation*} 
N_{k}(n) \equiv \sum_{m = 1}^{10} m^k N(m,11,n) \Mod {11}.
\end{equation*}
Then consider
\begin{align*}
N_{2}(n) \equiv 2N(1,11,n)-3N(2,11,n)-4N(3,11,n)-N(4,11,n)+6N(5,11,n) \Mod {11},\\
N_{4}(n) \equiv 2N(1,11,n)-N(2,11,n)-3N(3,11,n)+6N(4,11,n)-4N(5,11,n) \Mod {11},\\
N_{6}(n) \equiv 2N(1,11,n)-4N(2,11,n)+6N(3,11,n)-3N(4,11,n)-N(5,11,n) \Mod {11},\\
N_{8}(n) \equiv 2N(1,11,n)+6N(2,11,n)-N(3,11,n)-4N(4,11,n)-3N(5,11,n) \Mod {11}.
\end{align*}
Taking $n=11l+m$ in congruences above and using Theorem \ref{theorem:rankdissection} and calculations in Section \ref{section:devrank} we obtain Theorem \ref{theorem:rankmomentcong}. For example using notation (\ref{definition:disselemrankmoment}) with $m=1$ we have
\begin{multline*}
T_{2,1}(q) = \sum_{n=0}^{\infty}N_{2}(11n+1)q^n \equiv  2Q_{1,1}(q)-3Q_{2,1}(q)-4Q_{3,1}(q)-Q_{4,1}(q)+6Q_{5,1}(q)  \\ \equiv [0,-1,-5,-4,-3;-2]_{1} \Mod {11}.\qedhere
\end{multline*}
\end{proof}

\begin{proof}[Proof of Theorem \ref{theorem:sptcong}] Andrews \cite{A} showed that $\operatorname{spt}(n)$  is related to the second rank moment
\begin{equation} \label{identity:andrewssptform}
\operatorname{spt}(n) = np(n) - \frac{1}{2}N_2(n).
\end{equation}
As an example let us consider congruence
\begin{equation*}
\sum_{n=0}^{\infty}\operatorname{spt}(11n+1)q^n \equiv [1,5,-4,1,-5;1]_{1} \Mod {11}.
\end{equation*}
We know that
\begin{equation*}
\sum_{n=0}^{\infty} (11n+1)p(11n+1)q^n \equiv [1,-1,-1,-1,-1;0]_{1} \Mod {11}
\end{equation*}
by Theorem \ref{theorem:pcong} and
\begin{equation*}
\sum_{n=0}^{\infty} N_2(11n+1)q^n \equiv [0,-1,-5,-4,-3;-2]_{1} \Mod {11}
\end{equation*}
by Theorem \ref{theorem:rankmomentcong}. Using (\ref{identity:andrewssptform}) we obtain the desired congruence.
\end{proof}

\begin{proof}[Proof of Corollary \ref{corollary:eisensteincong}]
In terms of notation (\ref{definition:disselemrankmoment}) from \cite[Theorem 5.1]{G10} we know
\begin{align*}
T_{2,6}(q) &\equiv 3J_1^{13} \Mod {11},\\
T_{6,6}(q) &\equiv J_1^{13}(4+E_4(q)) \Mod {11},\\
T_{8,6}(q) &\equiv J_1^{13}(5+6E_4(q)+6E_6(q)) \Mod {11}.
\end{align*}  
Using Theorem \ref{theorem:rankmomentcong} and applying
\begin{equation*}
J_1^{13} \equiv J_1^2J_{11} \Mod {11},
\end{equation*}
we obtain Corollary \ref{corollary:eisensteincong}.
\end{proof}

\begin{remark} We also can deduce congruences modulo $11$ for rank moments and $\operatorname{spt}(n)$ corresponding to residues $i \in \{0,4,7,9,10\}$, but they consist of the universal mock theta functions $g(x;q)$. For example for residue $0$ in terms of notation (\ref{definition:disselemrankmoment}) we have \begin{align*}
T_{2,0}(q) \equiv T_{4,0}(q) &\equiv 2q^2g(q^2;q^{11})+[0,0,4,5,1;-2]_{0} \Mod {11},\\
T_{6,0}(q) &\equiv 2q^2g(q^2;q^{11})+[0,3,-2,1,-4;-2]_{0} \Mod {11},\\
T_{8,0}(q) &\equiv 2q^2g(q^2;q^{11})+[0,4,1,0,2;-2]_{0} \Mod {11},
\end{align*}
and
\begin{equation*}
\sum_{n=0}^{\infty}\operatorname{spt}(11n)q^n \equiv -q^2g(q^2;q^{11})+[0,0,-2,3,5;1]_{0} \Mod {11}.
\end{equation*}
\end{remark}

\section{Full proofs for new crank-rank inequalities} \label{section:fullcalc}
In this section we give the full account of calculations for the proofs of Theorem \ref{theorem:rankcrankbasicineq}, Corollary \ref{corollary:inequallevel2level4}, Corollary \ref{corollary:inequallevel6}, Theorem \ref{theorem:inequalrank6basis} and Corollary \ref{corollary:ineqrank6level4level6}.

\begin{definition} Using the notation of (\ref{definition:rankdisselem}) and (\ref{definition:crankdisselem}), we define
\begin{align*}
(a_0,a_2,a_3,a_4,a_5;b_0,b_1)_{0}&:=\sum_{i=0}^{5} a_i Q_{i,0}(q) + b_0 Q^{C}_{0,0}(q) + b_1 Q^{C}_{1,0}(q),\\
(a_0,a_2,a_3,a_4,a_5;b_0,b_1,b_2)_{1}&:=\sum_{i=0}^{5} a_i Q_{i,1}(q) + \sum_{i=0}^{2} b_i Q^{C}_{i,1}(q),\\
(a_0,a_2,a_3,a_4,a_5;b_0,b_2)_{2}&:=\sum_{i=0}^{5} a_i Q_{i,2}(q) + b_0 Q^{C}_{0,2}(q) + b_2 Q^{C}_{2,2}(q),\\
(a_0,a_2,a_3,a_4,a_5;b_0,b_1)_{3}&:=\sum_{i=0}^{5} a_i Q_{i,3}(q) + b_0 Q^{C}_{0,3}(q) + b_1 Q^{C}_{1,3}(q),\\
(a_0,a_2,a_3,a_4,a_5;b_0,b_1)_{4}&:=\sum_{i=0}^{5} a_i Q_{i,4}(q) + b_0 Q^{C}_{0,4}(q) + b_1 Q^{C}_{1,4}(q),\\
(a_0,a_2,a_3,a_4,a_5;b_0,b_2)_{5}&:=\sum_{i=0}^{5} a_i Q_{i,5}(q) + b_0 Q^{C}_{0,5}(q) + b_2 Q^{C}_{2,5}(q),\\
(a_0,a_1,a_2,a_3,a_4,a_5)_{6}&:=\sum_{i=0}^{5} a_i Q_{i,6}(q),\\
(a_0,a_1,a_2,a_3,a_4,a_5;b_0,b_1)_{7}&:=\sum_{i=0}^{5} a_i Q_{i,7}(q) + b_0 Q^{C}_{0,7}(q) + b_1 Q^{C}_{1,7}(q),\\
(a_0,a_1,a_2,a_3,a_4,a_5;b_0,b_1)_{8}&:=\sum_{i=0}^{5} a_i Q_{i,8}(q) + b_0 Q^{C}_{0,8}(q) + b_1 Q^{C}_{1,8}(q),\\
(a_0,a_1,a_2,a_3,a_4,a_5;b_0,b_1)_{9}&:=\sum_{i=0}^{5} a_i Q_{i,9}(q) + b_0 Q^{C}_{0,9}(q) + b_1 Q^{C}_{1,9}(q),\\
(a_0,a_1,a_2,a_3,a_4,a_5;b_0,b_3)_{10}&:=\sum_{i=0}^{5} a_i Q_{i,10}(q) + b_0 Q^{C}_{0,10}(q) + b_3 Q^{C}_{3,10}(q).
\end{align*} 
\end{definition}

\begin{proof}[Full proof of Theorem \ref{theorem:rankcrankbasicineq}] The calculations below can be derived directly from calculations of the dissection elements $Q_{a,m}(q)$ from Section \ref{section:devrank} and calculations of the dissection elements $Q^{C}_{a,m}(q)$ from Theorem \ref{theorem:crankdissection}. The positivity of sums of theta quotients can be derived from Proposition \ref{proposition:maincomparison} and Lemma \ref{lemma:addcomparison}. The inequalities with $N_i = N(i,11,11n)$ and $M_i = M(i,11,11n)$ are equivalent to 
\begin{align*}
(1,2,-2,0,-1,0;-1,1)_{0} &= [0,11,0,0,0;0]_{0} \geq 0,\\
(1,2,3,-3,0,-3;-1,1)_{0} &= [0,-11,11,0,0;0]_{0} \geq 0,\\
(0,0,2,1,-4,1;0,0)_{0} &= [0,0,-11,11,0;0]_{0} \geq 0,\\
(-1,-2,1,5,3,-6;1,-1)_{0} &= [0,0,0,-11,11;0]_{0} \geq 0.
\end{align*} 
The inequalities with $N_i = N(i,11,11n+1)$ and $M_i = M(i,11,11n+1)$ are equivalent to 
\begin{align*}
(1,-2,4,-3,1,-1;1,0,-1)_{1} &= [0,11,0,0,0;0]_{1} \geq 0,\\
(1,3,-4,6,-6,0;1,0,-1)_{1} &= [0,-11,11,0,0;0]_{1} \geq 0,\\
(0,2,-1,-3,6,-4;0,0,0)_{1} &= [0,0,-11,11,0;0]_{1} \geq 0,\\
(-1,1,2,-1,-4,3;-2,-1,3)_{1} &= [0,0,0,-11,11;0]_{1} \geq 0,\\
(-1,-1,3,2,1,-4;-2,-1,3)_{1} &= [0,0,0,0,-11;11]_{1} \geq 0.
\end{align*} 
The inequalities with $N_i = N(i,11,11n+2)$ and $M_i = M(i,11,11n+2)$ are equivalent to 
\begin{align*}
(-2,0,3,0,1,-2;0,0)_{2} &= [0,0,0,0,0;11]_{2} \geq 0\\
(0,2,-2,2,-3,1;2,-2)_{2} &= [0,11,0,0,0;-11]_{2} \geq 0\\
(3,-1,2,-1,0,-3;-1,1)_{2} &= [0,-11,11,0,0;0]_{2} \geq 0\\
(2,1,-4,1,3,-3;1,-1)_{2} &= [0,0,-11,11,0;0]_{2} \geq 0\\
(1,3,1,-8,-5,8;3,-3)_{2} &= [0,0,0,-11,11;0]_{2} \geq 0.
\end{align*} 
The inequalities with $N_i = N(i,11,11n+3)$ and $M_i = M(i,11,11n+3)$ are equivalent to 
\begin{align*}
(1,-1,0,2,0,-2;-1,1)_{3} &= [0,11,0,0,0;0]_{3} \geq 0,\\
(-2,5,2,-4,2,-3;0,0)_{3} &= [0,-11,11,0,0;0]_{3} \geq 0,\\
(2,-1,-3,4,-3,1;1,-1)_{3} &= [0,0,-11,0,0;11]_{3} \geq 0,\\
(-1,-1,6,-2,-5,3;-5,5)_{3} &= [0,0,0,11,0;-11]_{3} \geq 0,\\
(4,2,-7,-3,4,0;3,-3)_{3} &= [0,0,0,-11,11;0]_{3} \geq 0.
\end{align*} 
The inequalities with $N_i = N(i,11,11n+4)$ and $M_i = M(i,11,11n+4)$ are equivalent to 
\begin{align*}
(-2,4,-3,3,-1,-1;-5,5)_{4} &= [0,11,0,0,0;0]_{4} \geq 0,\\
(4,-5,5,-2,-1,-1;3,-3)_{4} &= [0,-11,11,0,0;0]_{4} \geq 0,\\
(-2,3,1,-2,0,0;1,-1)_{4} &= [0,0,-11,11,0;0]_{4} \geq 0,\\
(3,0,-3,-2,1,1;-1,1)_{4} &= [0,0,0,-11,11;0]_{4} \geq 0.
\end{align*} 
The inequalities with $N_i = N(i,11,11n+5)$ and $M_i = M(i,11,11n+5)$ are equivalent to 
\begin{align*}
(3,-2,1,-3,0,1;-2,2)_{5} &= [0,11,0,0,0;0]_{5} \geq 0,\\
(-3,7,-2,-1,1,-2;-3,3)_{5} &= [0,-11,11,0,0;0]_{5} \geq 0,\\
(2,-3,1,3,-4,1;4,-4)_{5} &= [0,0,-11,11,0;0]_{5} \geq 0,\\
(-4,-2,4,7,2,-7;0,0)_{5} &= [0,0,0,-11,11;0]_{5} \geq 0,\\
(-2,4,1,-6,2,1;-5,5)_{5} &= [0,0,11,0,0;-11]_{5} \geq 0.
\end{align*} 
The inequalities with $N_i = N(i,11,11n+7)$ and $M_i = M(i,11,11n+7)$ are equivalent to 
\begin{align*}
(1,0,0,-2,2,-1;2,-2)_{7} &= [0,11,0,0,0;0]_{7} \geq 0,\\
(3,1,1,7,-5,4;-4,-7)_{7} &= [0,-11,11,0,0;0]_{7} \geq 0,\\
(-4,4,4,-6,3,-1;-4,4)_{7} &= [0,0,-11,11,0;0]_{7} \geq 0,\\
(1,-2,-2,5,2,-4;0,0)_{7} &= [0,0,0,-11,11;0]_{7} \geq 0.
\end{align*} 
The inequalities with $N_i = N(i,11,11n+8)$ and $M_i = M(i,11,11n+8)$ are equivalent to 
\begin{align*}
(1,-3,2,0,3,-3;5,-5)_{8} &= [0,11,0,0,0;0]_{8} \geq 0,\\
(-1,2,-2,4,-5,2;-3,3)_{8} &= [0,-11,11,0,0;0]_{8} \geq 0,\\
(-3,2,5,-5,-1,2;-1,1)_{8} &= [0,0,-11,11,0;0]_{8} \geq 0,\\
(7,4,-8,-1,5,4;-4,-7)_{8} &= [0,0,0,-11,11;0]_{8} \geq 0,\\
(5,6,-1,4,2,-5;-6,-5)_{8} &= [0,0,0,-11,0;11]_{8} \geq 0.
\end{align*} 
The inequalities with $N_i = N(i,11,11n+9)$ and $M_i = M(i,11,11n+9)$ are equivalent to 
\begin{align*}
(0,-2,4,-1,-1,0;1,-1)_{9} &= [0,11,0,0,0;0]_{9} \geq 0,\\
(-1,4,-3,1,1,-2;0,0)_{9} &= [0,-11,11,0,0;0]_{9} \geq 0,\\
(4,-2,-5,3,3,-3;4,-4)_{9} &= [0,0,-11,11,0;0]_{9} \geq 0,\\
(-2,3,4,-6,-6,7;-3,3)_{9} &= [0,0,0,-11,11;0]_{9} \geq 0.
\end{align*} 
The inequalities with $N_i = N(i,11,11n+10)$ and $M_i = M(i,11,11n+10)$ are equivalent to 
\begin{align*}
(0,3,-2,-2,-1,2;-2,2)_{10} &= [0,11,0,0,0;0]_{10} \geq 0,\\
(3,-3,1,1,-2,0;1,-1)_{10} &= [0,-11,11,0,0;0]_{10} \geq 0,\\
(-1,2,-1,-1,4,-3;-1,1)_{10} &= [0,0,-11,11,0;0]_{10} \geq 0,\\
(-6,-6,6,6,-3,-8;6,5)_{10} &= [0,0,0,-11,11;0]_{10} \geq 0.\qedhere
\end{align*} 
\end{proof}

\begin{proof}[Full proof of Corollary \ref{corollary:inequallevel2level4}] The calculations below can be derived directly from calculations of the dissection elements $Q_{a,m}(q)$ from Section \ref{section:devrank} and calculations of the dissection elements $Q^{C}_{a,m}(q)$ from Theorem \ref{theorem:crankdissection}. The positivity of sums of theta quotients can be derived from Proposition \ref{proposition:maincomparison} and Lemma \ref{lemma:addcomparison}. The inequalities with $N_i = N(i,11,11n)$ and $M_i = M(i,11,11n)$ are equivalent to 
\begin{align*}
(0,0,0,0,-1,0;0,1)_{0} &= [0,1,-2,2,0;0]_{0} \geq 0,\\
(0, 0, 1, 1, -1, 0; 0, -1)_{0} &= [0, 0, -4, 3, 1; 0]_{0} \geq 0.
\end{align*} 
The inequalities with $N_i = N(i,11,11n+1)$ and $M_i = M(i,11,11n+1)$ are equivalent to 
\begin{align*}
(0,0,1,0,0,0;0,0,-1)_{1} &=  [0,1,0,-1,0;1]_{1} \geq 0,\\
(0,0,0,0,-1,0;0,0,1)_{1} &=  [0,0,1,-1,1;0]_{1} \geq 0,\\
(0, 0, 1, -1, 1, -1; 0, 0, 0)_{1} &= [0, 2, -2, 1, 0; 1]_{1} \geq 0.
\end{align*} 
The inequalities with $N_i = N(i,11,11n+2)$ and $M_i = M(i,11,11n+2)$ are equivalent to 
\begin{align*}
(0,1,0,0,0,0;0,-1)_{2} &= [0,2,-2,1,1;1]_{2} \geq 0\\
(0,0,0,0,0,-1;1,0)_{2} &= [0,0,0,1,0;1]_{2} \geq 0,\\
(0, 0, -1, 0, 0, -1; 2, 0)_{2} &= [0, 1, -2, 2, 0; 0]_{2} \geq 0,\\
(-1, 0, 1, 0, 0, -1; 1, 0)_{2} &= [0, 1, 0, 0, 0; 4]_{2} \geq 0,\\
(0, 1, 0, 1, 0, 0; -1, -1)_{2} &= [0, 3, -1, 1, 0; -1]_{2} \geq 0,\\
(0, 1, 0, 0, -1, 1; 0, -1)_{2} &= [0, 3, 0, -1, 1; -2]_{2} \geq 0.
\end{align*} 
The inequalities with $N_i = N(i,11,11n+3)$ and $M_i = M(i,11,11n+3)$ are equivalent to 
\begin{align*}
(1,0,0,0,0,0;-1,0)_{3} &= [0,4,0,0,2;-2]_{3} \geq 0,\\
(0,0,0,0,-1,0;0,1)_{3} &= [0,-1,-1,1,0;1]_{3}  \geq 0,\\
(1, 0, -2, 0, 0, 0; 1, 0)_{3} &= [0, 0, -2, -2, 2; 2]_{3} \geq 0,\\
(0, 0, 1, 0, -1, 0; -1, 1)_{3} &= [0, 1, 0, 2, 0; -1]_{3} \geq 0,\\
(1, 0, -1, 1, 0, 0; 0, -1)_{3} &= [0, 2, -2, -1, 1; 2]_{3} \geq 0.
\end{align*} 
The inequalities with $N_i = N(i,11,11n+4)$ and $M_i = M(i,11,11n+4)$ are equivalent to 
\begin{align*}
(0,1,0,0,0,0;-1,0)_{4} &= [0,1,-1,1,1;0]_{4} \geq 0,\\
(-1, 2, 0, 0, 0, 0;-1, 0)_{4} &= [0, 2, -4, 4, 0; 0]_{4} \geq 0,\\
(0, 1, -1, 0, 0, 0; -1, 1)_{4} &= [0, 2, -1, -1, 2; 0]_{4} \geq 0,\\
(0, 1, 0, 1, 0, 0; -2, 0)_{4} &= [0, 2, 1, 0, 0; 0]_{4} \geq 0.
\end{align*} 
The inequalities with $N_i = N(i,11,11n+5)$ and $M_i = M(i,11,11n+5)$ are equivalent to 
\begin{align*}
(0,0,1,0,0,0;-1,0)_{5} &= [0,1,0,1,1;-2]_{5}\geq 0,\\
(0,0,0,0,0,-1;0,1)_{5} &= [0,1,1,-1,1;1]_{5} \geq 0,\\
(0, 0, 1, 1, 0, 0; -1, -1)_{5} &= [0, 0, -1, 1, 1; -1]_{5} \geq 0,\\
(-1, 0, 1, 0, 0, -1; 0, 1)_{5} &= [0, 0, 1, 0, 2; -3]_{5} \geq 0,\\
(1, -1, 0, 0, -1, 0; 1, 0)_{5} &= [0, 2, -2, 2, 0; 1]_{5} \geq 0,\\
(0, -1, 1, 1, 0, -1; 0, 0)_{5} &= [0, 2, -1, -1, 2; 0]_{5} \geq 0,\\
(1, -1, 1, 0, 0, 0; -1, 0)_{5} &= [0, 4, -1, 0, 1; 0]_{5} \geq 0.
\end{align*} 
The inequalities with $N_i = N(i,11,11n+7)$ and $M_i = M(i,11,11n+7)$ are equivalent to 
\begin{align*}
(1,0,0,0,0,0;0,-1)_{7} &= [1,3,1,-1,-1;0]_{7} \geq 0,\\
(0,0,0,0,-1,0;1,0)_{7} &= [0,-1,1,1,0;0]_{7} \geq 0,\\
(0, 0, 0, 1, 1, -1; -1, 0)_{7} &= [0, 0, -1, -2, 3; 0]_{7} \geq 0.
\end{align*} 
The inequalities with $N_i = N(i,11,11n+8)$ and $M_i = M(i,11,11n+8)$ are equivalent to 
\begin{align*}
(0,0,1,0,0,0;0,-1)_{8} &= [1,1,0,1,0;0]_{8} \geq 0,\\
(0,0,0,-1,0,0;1,0)_{8} &= [0,1,-1,1,1;0]_{8} \geq 0,\\
(0, 0, -1, 0, -1, 0; 1, 1)_{8} &= [0, -1, 2, 0, 1; 0]_{8} \geq 0,\\
(0, 1, 0, -2, 0, 1; 0, 0)_{8} &= [0, 0, -3, 2, 2; 0]_{8} \geq 0,\\
(1, 0, -2, 0, 0, 0; 1, 0)_{8} &= [0, 0, 2, -2, 2; 0]_{8} \geq 0,\\
(0, 0, 0, 1, -1, 0; 1, -1)_{8} &= [0, 0, 4, 1, 1; 0]_{8} \geq 0,\\
(0, 0, 1, 1, 0, 0; 0, -2)_{8} &= [0, 1, 2, 1, 0; 0]_{8} \geq 0,\\
(0, -1, 0, 0, 0, -1; 2, 0)_{8} &= [0, 2, 1, 0, 0; 0]_{8} \geq 0.
\end{align*} 
The inequalities with $N_i = N(i,11,11n+9)$ and $M_i = M(i,11,11n+9)$ are equivalent to 
\begin{align*}
(0,0,1,0,0,0;0,-1)_{9} &= [0,2,1,0,0;0]_{9} \geq 0,\\
(0,0,0,0,0,-1;1,0)_{9} &= [0,1,2,1,0;0]_{9} \geq 0,\\
(0, 0, -1, 0, 0, -1; 1, 1)_{9} &= [0, -1, 1, 1, 0; 0]_{9} \geq 0,\\
(1, 0, -1, 0, 0, 0; 1, -1)_{9} &= [0, 0, -3, 2, 2; 0]_{9} \geq 0,\\
(0, 0, 1, -1, -1, 1; 0, 0)_{9} &= [0, 2, -1, -1, 2; 0]_{9} \geq 0.
\end{align*} 
The inequalities with $N_i = N(i,11,11n+10)$ and $M_i = M(i,11,11n+10)$ are equivalent to 
\begin{align*}
(0,1,0,0,0,0;-1,0)_{10} &= [1,2,-1,0,0;0]_{10} \geq 0,\\
(0,0,0,0,-1,0;0,1)_{10} &= [0,2,2,-1,1;0]_{10} \geq 0,\\
(1, -1, 0, 0, -1, 0; 1, 0)_{10} &= [0, -3, 4, 0, 0; 0]_{10} \geq 0,\\
(0, 0, 1, 1, 0, 0; -2, 0)_{10} &= [0, 0, 1, -1, 1; 0]_{10} \geq 0,\\
(0, 1, 0, 0, 1, 0; -2, 0)_{10} &= [0, 1, -2, 2, 0; 0]_{10} \geq 0,\\
(0, 1, -1, -1, 0, 0; 0, 1)_{10} &= [0, 3, -1, 2, 0; 0]_{10} \geq 0.\qedhere
\end{align*} 
\end{proof}

\begin{proof}[Full proof of Corollary \ref{corollary:inequallevel6}] The calculations below can be derived directly from calculations of the dissection elements $Q_{a,m}(q)$ found in Section \ref{section:devrank} and calculations of the dissection elements $Q^{C}_{a,m}(q)$ found in Theorem \ref{theorem:crankdissection}. The positivity of sums of theta quotients can be derived from Proposition \ref{proposition:maincomparison} and Lemma \ref{lemma:addcomparison}. The inequalities with $N_i = N(i,11,11n)$ and $M_i = M(i,11,11n)$ are equivalent to 
\begin{align*}
(0, 0, 1, 2, 0, -1; 0, -2)_{0} &= [0, 1, -3, 0, 3; 0]_{0} \geq 0,\\
(0, 0, 0, 2, 1, -2; 0, -1)_{0} &= [0, 3, 0, -4, 4; 0]_{0} \geq 0,\\
(0, 0, -1, 0, -1, -1; 0, 3)_{0} &= [0, 4, -1, 0, 1; 0]_{0} \geq 0,\\
(0, 0, -1, 1, 0, -2; 0, 2)_{0} &= [0, 5, 0, -3, 3; 0]_{0} \geq 0.
\end{align*} 
The inequalities with $N_i = N(i,11,11n+1)$ and $M_i = M(i,11,11n+1)$ are equivalent to 
\begin{align*}
(0, 2, -1, 1, -2, 0; 0, 0, 0)_{1} &= [0, -4, 1, -1, 4; 0]_{1} \geq 0,\\
(0, 2, -1, 0, 1, -2; 0, 0, 0)_{1} &= [0, -3, -3, 4, 1; 1]_{1} \geq 0,\\
(0, 1, 0, 0, 2, -2; -1, 0, 0)_{1} &= [1, -2, -4, 3, -3; 2]_{1} \geq 0,\\
(0, 1, 0, 1, -3, 1; 0, 0, 0)_{1} &= [0, -2, 3, -4, 4; 0]_{1} \geq 0,\\
(0, 0, 2, 0, 1, -2; -1, 0, 0)_{1} &= [1, 1, -2, 0, -4; 4]_{1} \geq 0.
\end{align*} 
The inequalities with $N_i = N(i,11,11n+2)$ and $M_i = M(i,11,11n+2)$ are equivalent to 
\begin{align*}
(1, 1, -1, 0, 1, 0; -1, -1)_{2} &= [0, 0, -4, 3, 1; 0]_{2} \geq 0,\\
(0, 0, -1, -1, -1, 0; 3, 0)_{2} &= [0, 1, -1, 0, 1; -1]_{2} \geq 0,\\
(0, 1, 1, 0, 0, 1; -2, -1)_{2} &= [0, 1, 0, -1, 1; 1]_{2} \geq 0,\\
(-1, 0, 0, 0, -1, -1; 3, 0)_{2} &= [0, 3, 0, 0, 0; 1]_{2} \geq 0,\\
(-1, 1, 0, 1, 0, -2; 1, 0)_{2} &= [1, 4, -4, 2, -1; 3]_{2} \geq 0,\\
(0, 1, -1, 1, 0, -1; 1, -1)_{2} &= [0, 4, -3, 3, 0; -1]_{2} \geq 0,\\
(0, 1, -1, 0, -1, 0; 2, -1)_{2} &= [0, 4, -2, 1, 1; -2]_{2} \geq 0,\\
(-1, 1, 1, 1, 0, -1; 0, -1)_{2} &= [0, 4, -1, 1, 0; 3]_{2} \geq 0,\\
(-1, 1, 1, 0, -1, 0; 1, -1)_{2} &= [0, 4, 0, -1, 1; 2]_{2} \geq 0.
\end{align*} 
The inequalities with $N_i = N(i,11,11n+3)$ and $M_i = M(i,11,11n+3)$ are equivalent to 
\begin{align*}
(0, 2, 0, -1, 1, -1; 0, -1)_{3} &= [0, -3, 3, -1, 1; 1]_{3} \geq 0,\\
(1, 1, 0, 0, 0, 1; -1, -2)_{3} &= [0, -1, -1, 0, 2; 0]_{3} \geq 0,\\
(1, 1, -1, 0, 1, 0; 0, -2)_{3} &= [0, 0, 0, -2, 2; 1]_{3} \geq 0,\\
(0, 0, -1, -1, 0, -1; 1, 2)_{3} &= [0, 0, 1, -1, 1; 0]_{3} \geq 0,\\
(2, 0, -2, 0, -1, 1; 0, 0)_{3} &= [0, 1, -4, -1, 4; 1]_{3} \geq 0,\\
(1, 0, 0, 1, -1, 1; -1, -1)_{3} &= [0, 1, -3, 1, 1; 1]_{3} \geq 0,\\
(0, -1, 0, 0, -2, 0; 0, 3)_{3} &= [0, 1, -2, 2, 0; 0]_{3} \geq 0,\\
(0, 0, 2, 0, -1, 1; -2, 0)_{3} &= [0, 1, 0, 3, 0; -3]_{3} \geq 0,\\
(1, -1, 0, 1, -2, 0; 0, 1)_{3} &= [1, 4, -4, 1, 0; 0]_{3} \geq 0,\\
(1, 0, 0, 2, 0, 0; -1, -2)_{3} &= [0, 4, -2, 0, 0; 2]_{3} \geq 0.
\end{align*} 
The inequalities with $N_i = N(i,11,11n+4)$ and $M_i = M(i,11,11n+4)$ are equivalent to 
\begin{align*}
(2, -1, 0, -2, 0, 0; 1, 0)_{4} &= [0, -3, 1, -3, 5; 0]_{4} \geq 0,\\
(2, -1, 1, -1, 0, 0; 0, -1)_{4} &= [0, -3, 3, -2, 3; 0]_{4} \geq 0,\\
(1, -1, 1, 0, -1, -1; 0, 1)_{4} &= [0, -1, 5, 0, 0; 0]_{4} \geq 0,\\
(1, 0, 1, 1, -1, -1; -1, 0)_{4} &= [1, 0, 5, -1, -1; 0]_{4} \geq 0,\\
(1, 0, -2, -1, 0, 0; 0, 2)_{4} &= [0, 1, 0, -5, 5; 0]_{4} \geq 0,\\
(0, 0, -1, 0, -1, -1; 0, 3)_{4} &= [0, 2, 2, -1, 1; 0]_{4} \geq 0.
\end{align*} 
The inequalities with $N_i = N(i,11,11n+5)$ and $M_i = M(i,11,11n+5)$ are equivalent to 
\begin{align*}
(0, 0, 0, 2, 0, -2; 1, -1)_{5} &= [1, -1, -1, -3, 1; 4]_{5} \geq 0,\\
(-1, 0, 1, 2, 0, -2; 0, 0)_{5} &= [0, -1, 0, -1, 3; 0]_{5} \geq 0,\\
(-1, 0, 0, 1, 0, -2; 1, 1)_{5} &= [0, -1, 1, -2, 2; 1]_{5} \geq 0,\\
(0, -1, 0, 1, -1, -1; 2, 0)_{5} &= [0, 0, -2, 1, 1; 1]_{5} \geq 0,\\
(-1, 0, 2, 1, 0, -1; -1, 0)_{5} &= [0, 0, 0, 1, 3; -4]_{5} \geq 0,\\
(0, 1, 1, -1, 0, 1; -2, 0)_{5} &= [0, 0, 1, 3, 0; -4]_{5} \geq 0,\\
(-1, 0, 0, -1, 0, -1; 1, 2)_{5} &= [0, 0, 2, -1, 1; -2]_{5} \geq 0,\\
(0, 1, 1, 0, 1, 0; -3, 0)_{5} &= [0, 1, 2, -1, 1; -1]_{5} \geq 0,\\
(1, 0, 1, 0, 0, 1; -2, -1)_{5} &= [0, 2, -1, 2, 0; -1]_{5} \geq 0.
\end{align*} 
The inequalities with $N_i = N(i,11,11n+7)$ and $M_i = M(i,11,11n+7)$ are equivalent to 
\begin{align*}
(-1, 1, 1, 0, 0, -1; -1, 1)_{7} &= [0, -3, -2, 2, 3; 0]_{7} \geq 0,\\
(-1, 1, 1, -1, -1, 0; 0, 1)_{7} &= [0, -3, -1, 4, 0; 0]_{7} \geq 0,\\
(-1, 0, 0, 0, 0, -2; 1, 2)_{7} &= [0, -2, -2, 0, 4; 0]_{7} \geq 0,\\
(1, 0, 0, 2, 0, 0; -1, -2)_{7} &= [0, 0, 2, -2, 2; 0]_{7} \geq 0,\\
(0, 1, 1, 0, 1, 0; -2, -1)_{7} &= [0, 1, -1, 1, 1; 0]_{7} \geq 0.
\end{align*} 
The inequalities with $N_i = N(i,11,11n+8)$ and $M_i = M(i,11,11n+8)$ are equivalent to 
\begin{align*}
(0, 1, -1, -1, -1, 1; 0, 1)_{8} &= [0, -2, 0, 1, 2; 0]_{8} \geq 0,\\
(-1, 0, 0, 0, -2, 0; 1, 2)_{8} &= [0, -2, 2, 2, 0; 0]_{8} \geq 0,\\
(-1, 1, 1, 0, -2, 1; 0, 0)_{8} &= [0, -2, 2, 4, 1; 0]_{8} \geq 0,\\
(0, 1, 0, 1, -1, 1; -1, -1)_{8} &= [0, -2, 3, 1, 1; 0]_{8} \geq 0,\\
(0, 1, -1, 0, -2, 1; 1, 0)_{8} &= [0, -2, 4, 2, 3; 0]_{8} \geq 0,\\
(1, 1, -2, -1, 0, 1; 0, 0)_{8} &= [0, -1, 0, -1, 3; 0]_{8} \geq 0,\\
(-1, 0, 1, 1, -2, 0; 1, 0)_{8} &= [0, -1, 4, 3, 0; 0]_{8} \geq 0,\\
(-1, 0, 1, -1, -1, 0; 1, 1)_{8} &= [0, 0, -1, 3, 0; 0]_{8} \geq 0,\\
(0, 1, 1, 0, 0, 1; -1, -2)_{8} &= [0, 0, 0, 2, 1; 0]_{8} \geq 0,\\
(1, 1, -1, 0, 0, 1; 0, -2)_{8} &= [0, 0, 2, 0, 3; 0]_{8} \geq 0,\\
(-1, 0, 2, 0, -1, 0; 1, -1)_{8} &= [0, 1, 1, 4, 0; 0]_{8} \geq 0,\\
(1, 0, -1, 1, 0, 0; 1, -2)_{8} &= [0, 1, 4, -1, 2; 0]_{8} \geq 0,\\
(-1, 0, 2, -2, 0, 0; 1, 0)_{8} &= [0, 2, -4, 4, 0; 0]_{8} \geq 0,\\
(1, 0, -1, -1, 1, 0; 1, -1)_{8} &= [0, 2, -1, -1, 2; 0]_{8} \geq 0,\\
(1, 0, 0, 1, 1, 0; 0, -3)_{8} &= [0, 2, 2, -1, 1; 0]_{8} \geq 0,\\
(0, 0, 2, 0, 1, 0; 0, -3)_{8} &= [0, 3, -1, 2, 0; 0]_{8} \geq 0.
\end{align*} 
The inequalities with $N_i = N(i,11,11n+9)$ and $M_i = M(i,11,11n+9)$ are equivalent to
\begin{align*}
(-1, 2, 0, -1, -1, 0; 0, 1)_{9} &= [0, -3, 6, -2, 2; 0]_{9} \geq 0,\\
(1, 0, -1, 1, 1, -1; 0, -1)_{9} &= [1, -1, -2, 2, -1; 0]_{9} \geq 0,\\
(0, 0, -1, -1, -1, 0; 1, 2)_{9} &= [0, -1, -1, 0, 2; 0]_{9} \geq 0,\\
(1, 0, -1, -1, -1, 1; 1, 0)_{9} &= [0, 0, -5, 1, 4; 0]_{9} \geq 0,\\
(0, 1, 1, 0, 0, 1; -1, -2)_{9} &= [0, 0, 1, -1, 1; 0]_{9} \geq 0,\\
(1, 0, 0, 1, 1, 0; 0, -3)_{9} &= [0, 1, -2, 2, 0; 0]_{9} \geq 0,\\
(1, -1, -1, 0, 0, -1; 2, 0)_{9} &= [0, 2, -3, 3, 1; 0]_{9} \geq 0,\\
(1, 0, 1, 0, 0, 1; 0, -3)_{9} &= [0, 3, -3, 1, 2; 0]_{9} \geq 0,\\
(1, -1, 0, -1, -1, 0; 2, 0)_{9} &= [0, 4, -4, 2, 3; 0]_{9} \geq 0,\\
(0, -1, 1, -1, -1, 0; 1, 1)_{9} &= [0, 4, -1, 0, 1; 0]_{9} \geq 0,\\
(1, -1, 1, 0, 0, 0; 1, -2)_{9} &= [0, 5, -3, 2, 1; 0]_{9} \geq 0,\\
(1, -2, 1, 0, 0, -1; 1, 0)_{9} &= [1, 6, -4, 2, -1; 0]_{9} \geq 0.
\end{align*} 
The inequalities with $N_i = N(i,11,11n+10)$ and $M_i = M(i,11,11n+10)$ are equivalent to 
\begin{align*}
(1, -1, 1, 1, -1, 0; -1, 0)_{10} &= [0, -3, 5, -1, 1; 0]_{10} \geq 0,\\
(0, -1, 0, 0, -1, -1; 2, 1)_{10} &= [0, -1, 2, 0, 1; 0]_{10} \geq 0,\\
(0, -1, 1, 1, -1, -1; 0, 1)_{10} &= [0, -1, 3, -1, 2; 0]_{10} \geq 0,\\
(-1, 0, 0, 0, 1, -2; 1, 1)_{10} &= [0, 0, -4, 3, 1; 0]_{10} \geq 0,\\
(-1, 0, 1, 1, 0, -1; -1, 1)_{10} &= [0, 2, -1, -1, 2; 0]_{10} \geq 0,
\end{align*}
\begin{align*}
(0, 0, -1, -1, -1, 0; 2, 1)_{10} &= [0, 2, 1, 0, 0; 0]_{10} \geq 0,\\
(1, 1, -1, -1, -1, 1; 0, 0)_{10} &= [1, 2, 2, 0, -1; 0]_{10} \geq 0,\\
(-1, 0, 0, 0, -1, -1; 1, 2)_{10} &= [0, 4, 0, -1, 2; 0]_{10} \geq 0,\\
(0, 2, -1, -1, 0, 1; -1, 0)_{10} &= [1, 5, -2, 0, -1; 0]_{10} \geq 0,\\
(-1, 1, 0, 0, 0, -1; -1, 2)_{10} &= [0, 5, -2, 1, 2; 0]_{10} \geq 0,\\
(0, 2, 0, 0, 0, 1; -3, 0)_{10} &= [1, 5, -1, -1, 0; 0]_{10} \geq 0. \qedhere
\end{align*} 
\end{proof}

\begin{proof}[Full proof of Theorem \ref{theorem:inequalrank6basis}]
The calculations below can be derived directly from calculations of the dissection elements $Q_{a,m}(q)$ from Theorem \ref{theorem:rankdissection}. The positivity of sums of theta quotients can be derived from Proposition \ref{proposition:maincomparisonres6}. The inequalities are equivalent to
\begin{align*}
(0,2,1,-2,2,-3)_{6} &= \Theta(0,0,0,0,11) \geq 0,\\
(2,-2,1,-1,-2,2)_{6} &= \Theta(0,0,0, 11,-11) \geq 0,\\
(1,1,-4,3,1,-2)_{6} &= \Theta(0,0,11,-11,0) \geq 0,\\
(1,6,4,-2,-5,-4)_{6} &= \Theta(0,11,-11,0,0) \geq 0. \qedhere
\end{align*} 
\end{proof}

\begin{proof}[Full proof of Corollary \ref{corollary:ineqrank6level4level6}]
The calculations below can be derived directly from calculations of the dissection elements $Q_{a,m}(q)$ from Theorem \ref{theorem:rankdissection}. The positivity of sums of theta quotients can be derived from Proposition \ref{proposition:maincomparisonres6}. The inequalities are equivalent to
\begin{align*}
(1, -1, 0, 1, -1, 0)_{6} &= \Theta(2, 0, 3, 3, -6) \geq 0,\\
(0, 1, 1, -1, 1, -2)_{6} &= \Theta(1, 0, 0, 1, 6) \geq 0,\\
(0, -1, 1, 2, -1, -1)_{6} &= \Theta(5, 0, 2, 2, -4) \geq 0,\\
(1, 1, -1, 0, 1, -2)_{6} &= \Theta(0, 0, 5, -2, 3) \geq 0,\\
(0, 0, -1, 3, 0, -2)_{6} &= \Theta(4, 0, 6, -5, -1) \geq 0,\\
(1, 0, -2, 2, 0, -1)_{6} &= \Theta(1, 0, 7, -4, -3) \geq 0. \qedhere
\end{align*} 
\end{proof}

\section*{Acknowledgements}
This work was performed at the Saint Petersburg Leonhard Euler International Mathematical Institute and supported by the Ministry of Science and Higher Education of the Russian Federation (agreement no. 075–15–2022–287). I would like to offer my gratitude to my scientific advisor Eric T. Mortenson for the idea to work through deviations of the crank and rank modulo $11$ and for helpful comments and crucial suggestions on references. I also want to thank Wadim Zudilin for pointing out \cite{PR} and Jeremy Lovejoy for suggesting the corrections to the article.


\begin{thebibliography}{999999}

\bibitem{A} G. E. Andrews, {\em The number of smallest parts in the partitions of n}, J. Reine Angew. Math. 624 (2008): 133--142.

\bibitem{AG} G. E. Andrews, F. G. Garvan, {\em Dyson’s crank of a partition}, Bull. Amer. Math. Soc. 18 (1988): 167--171.

\bibitem{AL} G. E. Andrews, R. Lewis, {\em The ranks and cranks of partitions moduli 2, 3, and 4}, J. Number Theory 85.1 (2000): 74--84.

\bibitem{AtG} A. O. L. Atkin, F. G. Garvan, {\em Relations between the ranks and cranks of partitions}, Ramanujan J. 7 (2003): 343--366.

\bibitem{AH} A. O. L. Atkin, S. M. Hussain, {\em Some properties of partitions. II}, Trans. Amer. Math. Soc. 89.1 (1958): 184--200.

\bibitem{ASD} A. O. L. Atkin, H. P. F. Swinnerton-Dyer, {\em Some properties of partitions}, Proc. Lond. Math. Soc. 3.1 (1954): 84--106.

\bibitem{BG} A. Berkovich, F. G. Garvan, {\em K. Saito's Conjecture for nonnegative eta products and analogous results for other infinite products}, J. Number Theory 128.6 (2008): 1731--1748.

\bibitem{B} B. C. Berndt, H. H. Chan, S. H. Chan, W. C. Liaw, {\em Cranks and dissections in Ramanujan's lost notebook}, J. Combin. Theory Ser. A, 109.1 (2005): 91--120. 

\bibitem{Bia} A. J. Biagioli, {\em A proof of some identities of Ramanujan using modular forms}, Glasg. Math. J. 31.3 (1989): 271--295.

\bibitem{GE} G. Bilgici, A. B. Ekin, {\em 11-Dissection and modulo 11 congruences properties for partition generating function}, Int. J. Contemp. Math. Sciences 9.1 (2014): 1--10.

\bibitem{BK} K. Bringmann, B. Kane, {\em Inequalities for differences of Dyson’s rank for all odd moduli}, Math. Res. Lett. 17.5 (2010): 927--942.

\bibitem{D} F. J. Dyson, {\em Some guesses in the theory of partitions}, Eureka (Cambridge) 8.10 (1944): 10--15.

\bibitem{E} A. B. Ekin, {\em Inequalities for the crank}, J. Combin. Theory Ser. A 83.2 (1998): 283--289.

\bibitem{E2000} A. B. Ekin, { \em Some properties of partitions in terms of crank}, Trans. Amer. Math. Soc. 352.5 (2000): 2145--2156.

\bibitem{G10} F. G. Garvan, {\em Congruences for Andrews' smallest parts partition function and new congruences for Dyson's rank}, Int. J. Number Theory 6.02 (2010): 281--309.

\bibitem{G88} F. G. Garvan, {\em New combinatorial interpretations of Ramanujan’s partition congruences mod 5, 7 and 11}, Trans. Amer. Math. Soc. 305.1 (1988): 47--77.

\bibitem{G19} F. G. Garvan, {\em Transformation properties for Dyson’s rank function}, Trans. Amer. Math. Soc. 371.1 (2019): 199--248.

\bibitem{GS} F. G. Garvan, R. Sarma, {\em New symmetries for Dyson's rank function}, arXiv preprint arXiv:2301.08960 (2023).

\bibitem{HM17} D. R. Hickerson, E. T. Mortenson, {\em Dyson’s ranks and Appell--Lerch sums}, Math. Ann. 367.1-2 (2017): 373--395.

\bibitem{HM14} D. R. Hickerson, E. T. Mortenson, {\em Hecke-type double sums, Appell--Lerch sums, and mock theta functions, I}, Proc. Lond. Math. Soc. 109.2 (2014): 382--422.

\bibitem{K} O. Kolberg, {\em Some identities involving the partition function}, Math. Scand. (1957): 77--92.

\bibitem{L} R. Lewis, {\em The generating functions of the rank and crank modulo 8}, Ramanujan J. 18.2 (2009): 121--146.

\bibitem{Ma} R. Mao, {\em Ranks of partitions modulo 10}, J. Number Theory 133, no. 11 (2013): 3678--3702.

\bibitem{M} E. T. Mortenson, {\em On ranks and cranks of partitions modulo 4 and 8}, J. Combin. Theory Ser. A 161 (2019): 51--80.

\bibitem{OB} J. N. O'Brien, {\em Some properties of partitions, with special reference to primes other than 5, 7 and 11}, Diss. Durham University, 1965.

\bibitem{PR} P. Paule, C.-S. Radu, {\em A unified algorithmic framework for Ramanujan’s congruences modulo powers of 5, 7, and 11}, Preprint (2018).

\bibitem{R} J. Z. Rolon, {\em Asymptotische Werte von Crank-Differenzen (Asymptotic values of crank differences)}, Diss. Ph. D. thesis, 2013.

\end{thebibliography}
\end{document}